\documentclass[11pt,letterpaper]{article}
\usepackage{amsfonts}
\usepackage{amsmath,amssymb,color}
\usepackage{graphicx}
\usepackage[singlespacing]{setspace}

\newtheorem{theorem}{Theorem}

\newtheorem{algorithm}{Algorithm}

\newtheorem{definition}{Definition}
\newtheorem{example}{Example}

\newtheorem{lemma}{Lemma}

\newtheorem{proposition}{Proposition}
\newtheorem{remark}{Remark}

\newenvironment{proof}[1][Proof]{\noindent\textbf{#1.} }{\ \rule{0.5em}{0.5em}}

\usepackage{dsfont}
\usepackage{bbold}
\def\definedas{\stackrel{\Delta}{=}}
\begin{document}

%\begin{frontmatter}
%\title{The Asymptotic Distribution of Bootstrap Variance Estimator for Weighted Quantile}
%
%\runtitle{Bootstrap Variance Estimator for Weighted Quantile}
%\begin{aug}
%
%\author{\fnms{Jingchen} \snm{Liu}\thanksref{t1}}
%\and
%\author{\fnms{Xuan} \snm{Yang}}%\thanksref{t2}}
%%\author{\fnms{Jennifer} \snm{Hill}\thanksref{t3}}
%%\and
%%\author{\fnms{Yu-Sung} \snm{Su}\thanksref{t4}}
%
%\thankstext{t1}{Research supported in part by Institu\textbf{}te of Education Sciences,
%U.S. Department of Education, through Grant R305D100017}
%%\thankstext{t2}{}
%%\thankstext{t3}{}
%%\thankstext{t4}{}
%\runauthor{Liu and Yang}
%\affiliation{Columbia University}
%\end{aug}
%
%\begin{abstract}
%In this paper, we provide the asymptotic distributions of the bootstrap variance estimators for quantile based on (weighted) empirical distributions. Under regularity conditions, we show that the bootstrap variance estimator is asymptotically normal and has relative standard deviation of order $n^{-1/4}$.
%\end{abstract}
%
%
%
%\begin{keyword}[class=AMS]
%\kwd[Primary ]{}
%\end{keyword}
%
%\begin{keyword}
%\kwd{Bootstrap, quantile, importance sampling}
%\end{keyword}
%\end{frontmatter}

%\baselineskip 18pt
%

\title{Rare-event Simulation and Efficient Discretization for the Supremum of Gaussian Random Fields}

\author{Xiaoou Li and Jingchen Liu\\\\ Columbia University}

\maketitle

\begin{abstract}
	In this paper, we consider a classic problem concerning the high excursion probabilities  of a Gaussian random field $f$ living on a compact set $T$.
	We develop efficient  computational methods for the tail probabilities $P(\sup_T f(t) > b)$ and the conditional expectations $E(\Gamma(f) | \sup _T f(t) > b)$ as $b\rightarrow \infty$.
	For each $\varepsilon$ positive, we present Monte Carlo algorithms that run in \emph{constant} time and compute the interesting quantities with $\varepsilon$ relative error  for arbitrarily large $b$.
%admit  $\varepsilon-\delta$ relative accuracy, that is, the estimators have $\varepsilon$ relative error for at least $1-\delta$ probability.
%	We prove that, to achieve the $\varepsilon-\delta$ relative accuracy, the total computation complexity is bounded by a \emph{constant} for an arbitrarily large $b$.
	The efficiency results are applicable to a large class of H\"older continuous Gaussian random fields.
	Besides computations, the proposed change of measure and its analysis techniques have several theoretical and practical indications in the asymptotic analysis of extremes of Gaussian random fields.
\end{abstract}

\section{Introduction}\label{SecIntro}

	In this paper, we consider the design and the analysis of efficient Monte Carlo methods for the high excursion events of Gaussian random fields.
	Consider a probability space $(\Omega, \mathcal F, P)$ and  a Gaussian random field
	$$f:T\times \Omega \rightarrow R$$
	living on a $d$-dimensional compact subset $T\subset R^d$.
	Most of the time, we omit the second argument and write $f(t)$.
	Let $M=\sup_{t\in T} f(t)$. In this paper, we are interested in the efficient computation of the high excursion probabilities of $f(t)$, that is,
\begin{equation}\label{TailProb}
w(b) \triangleq P(M > b)
\end{equation}
and the corresponding conditional expectations
\begin{equation}\label{CondExp}
v(b) \triangleq E(\Gamma(f) | M > b)
\end{equation}
in the asymptotic regime that $b$ tends to infinity, where $\Gamma(\cdot)$ is a functional (possibly a random functional)  mapping  from the space of continuous functions to the real line.

	The proposed algorithms are based on importance sampling that is associated with an appropriately designed change of measure mimicking the conditional distribution  $P(f\in \cdot ~| M > b)$.
%	Once such a change of measure is in place, the importance sampling algorithms for both \eqref{TailProb} and \eqref{CondExp} are naturally developed.
%
%Once such a change of measure has been chosen, one can construct an efficient Monte Carlo estimators that efficiently compute  the conditional expectations given the high excursion event, that is,
%$$E[\Gamma (f) | M > b],$$
% where $\Gamma(\cdot)$ is a functional of the field mapping from the space of continuous function space to the real line. This can be done by means of the same change of measure designed for the computation of $w(b)$.
	Much of this paper will focus on the design and the implementation of the algorithm for the tail probability $w(b)$.
	For the conditional expectation, we present an efficient algorithm and its analysis for one specific example: the integral on the excursion set with respect to positive processes.
	It turns out that the computations of $w(b)$ and $v(b)$ are closely related, which will be discussed in details in Section \ref{SecIS}.
	Most of the time, we are interested in computing small quantities converging to zero. Thus, it is sensible to consider relative accuracy that is defined as follows.
\begin{definition}
For some positive $\varepsilon$ and $\delta$, a Monte Carlo estimator $Z$ of $w$ is said to admit  $\varepsilon-\delta$ relative accuracy if
\begin{equation}\label{accuracy}
P(|Z - w| < \varepsilon w) > 1- \delta.
\end{equation}
\end{definition}

	We propose  Monte Carlo estimators admitting $\varepsilon-\delta$ relative accuracy   for computing the tail probabilities $w(b)$ and the conditional expectations $v(b)$.
	One notable feature of this estimator is that  \emph{the total computational complexity to generate one such estimator is bounded by a constant $C(\varepsilon,\delta)$ that is independent of the excursion level $b$.}
	Thus, to compute $w(b)$ and $v(b)$  with any prescribed relative accuracy as in \eqref{accuracy}, the total computational complexity remains bounded as the event becomes arbitrarily rare.
    With such an algorithm, the computation of rare event is at the same level of complexity as the computation of regular expectations.
	In addition, this efficiency result is applicable to a large class of H\"older continuous Gaussian random fields and thus is very generally applicable.

	The analysis mainly consists of two components. First, we propose a change of measure on the continuous sample path space (denoted by $Q_b$). The corresponding importance sampling estimators  are unbiased.
	The first step of the analysis is to show that the  estimators admit standard deviations on the order of $O(w(b))$ or $O(v(b))$. Such
estimators are said to be \emph{strongly efficient} that is a common efficiency concept in the rare-event simulation literature (\cite{BUC04,ASMGLY07}).
	
	The second part of the analysis concerns  the implementations. The simulation of the estimators in the previous paragraph requires the generation of the entire sample path of $f$. Under the current context, the process $f$ is a continuous function. Computer can only perform discrete simulations. Therefore, we need to seek for an appropriate discretization scheme to perform the simulations. For instance, a natural approach is to  choose a subset
	\begin{equation}\label{Tn}
	T_m=(t_1,...,t_m) \subset T
	\end{equation}
		and use the discrete field living on $T_m$ to approximate the continuous field. Thanks to continuity and under certain regularity conditions of $T_m$, one can show that $P(\sup_{T_m}f(t) > b)/w(b)\rightarrow 1$ as $m\rightarrow \infty$, i.e., the bias vanishes as the size of the discretization increases. 	
%	Another important contribution of this paper is to introduce an adaptive discretization scheme.  For a given level $\varepsilon$, the size of the discretization remains bounded as $b$ tends to infinity while ensuring that the relative bias is bounded by $\varepsilon$.
%	Using a discretization scheme with a fixed size to estimate $P(M>b)$ as $b\rightarrow\infty$ and maintaining an $\varepsilon$-relative bias, is an important  improvement and contribution.
%	    Thanks to continuity, if $T_m$ is appropriately chosen, we have that $P(\sup_{T_m}f(t) > b)/w(b)\rightarrow 1$ as $m\to \infty$.
    However, it is well understood that this convergence is not uniform in $b$. The smaller $w(b)$ is, the slower it converges, indicating that the set $T_m$ needs to grow in order to maintain a prefixed relative bias.
	In fact, as discussed in \cite{adler2012efficient}, for any deterministic subset $T_m$, the size $m$ \emph{must increase} at least polynomially with $b$ to ensure a given relative accuracy. In this paper,  the discretization scheme is  random and adapted to (correlated with) the random field $f$. This adaptive scheme substantially reduces the computation complexity, in particular, to a constant level.

	The high level excursion  of  Gaussian random fields is a classic topic in probability.
	There is a wealth of literature that contains general bounds on $P(\sup f(t)>b)$ as well as sharp asymptotic approximations as $ b\rightarrow \infty$. A partial literature contains \cite{LS70,MS70,ST74,Bor75,Bor03,LT91,TA96,Berman85}.
Several methods have been introduced to obtain bounds and asymptotic
approximations, each of which imposes different regularity conditions on the
random fields. General upper bound for the tail of $\max f(t)$ is developed in \cite{Bor75,CIS}, which is known as the Borel--TIS lemma. For asymptotic results,
there are several methods. The double sum method (\cite{Pit95}) requires an
expansion of the covariance function around its global maximum and
also locally stationary structure. The Euler--Poincar\'{e} Characteristics
of the excursion set approximation (denoted by $\chi(A_{b})$, where $A_{b}$
is the excursion set) uses the fact $P(M>b)\approx E(\chi(A_{b}))$ and
requires the random field to be at least twice differentiable (\cite%
{Adl81,TTA05,AdlTay07,TayAdl03}). The tube method (\cite{Sun93}) uses the
Karhunen-Lo\`{e}ve expansion and imposes differentiability assumptions on
the covariance function (fast decaying eigenvalues) and regularity
conditions on the random field. The Rice method (\cite{AW08,AW09})
represents the distribution of $M$ (density function) in an implicit form.
For other convex functionals, the exact tail approximation of integrals of exponential functions of Gaussian random fields is developed by \cite{Liu10,LiuXu11}.  Recently, \cite{AST09} studied the geometric properties of high level excursion set for infinitely divisible non-Gaussian fields as well as the conditional distributions of such properties given the high excursion.
Numerical methods are recently discussed by  \cite{adler2012efficient}  who proposes  importance sampling estimators of $w(b)$. In particular, the authors show that the proposed estimator is a fully polynomial randomized approximation scheme (FPRAS), that is, to achieve the $\varepsilon-\delta$ relative accuracy, the total computation complexity is of order $O(\varepsilon^{-q_1}\delta^{-q_2}|\log w(b)|^q)$ (\cite{TraubWW88,Wos96,MitzUpf05}).
%	FPRAS usually yields fast computations.
When $w(b)$ is very small, the complexity $O(|\log w(b)|^q)$ could be computationally heavy.
%	The  algorithm in this paper is also a FPRAS. Compare to the algorithm in \cite{adler2012efficient}, we substantially reduces the computational complexity by   reducing the power $q$ to zero.

	This paper is a nontrivial and substantial generalization of \cite{adler2012efficient}. In particular, the contributions are as follows.
    First, we introduce an adaptive discretization scheme that reduces the overall computational cost to a constant level. This is a substantial improvement of \cite{adler2012efficient}  who requires the discretization size grow polynomially in $b$ for both  differentiable and non-differentiable fields.
	Second, we show that the continuous importance sampling estimator is strongly efficient to compute $w(b)$ for both H\"older continuous fields and differentiable cases (by imposing mild regularity conditions). This generalizes the results in \cite{adler2012efficient} who establishes that their relative error grows polynomially fast with $b$ unless the process is twice differentiable for which the exact Slepian model is available.
	Third, we present an algorithm with constant complexity for the computation of the conditional expectations of integrals on the excursion sets.
%This is discussed in details in \cite{adler2012efficient}.
	Lastly, from the technical and methodological point of view, the development of this paper mostly deals with change of measures defined on the continuous sample path space. In contrast, the analysis of \cite{adler2012efficient} relies heavily on the discrete nature of the estimators (multivariate Gaussian random vectors). The methodological contribution of this paper is developing techniques to deal with change of measures defined on the continuous sample path space.
	As we shall see in the technical development, with moderate adaptations, the analysis techniques can be applied to the analysis of a large class of conditional expectations $E[\Gamma (f)|M>b]$. In particular, in Theorem \ref{ThmInt}, we employ the change of measure to derive the asymptotic approximations of the expected conditional integrals.

	As the total complexity of the Monte Carlo estimator is constant, the computational cost is comparable to that of the closed form approximation of $w(b)$.
One advantage of our method is that it can yield arbitrarily small relative error at the expense of more computational costs; while the error of closed form approximations are prefixed for each $b$ and they are usually not straightforward to obtain -- requiring second order approximations.
In addition to $w(b)$, the current Monte Carlo methods can also be used to compute the conditional expectations whose asymptotic analyses are case-by-case.
	Another advantage of the  Monte Carlo estimators is that their implementations do not require the computation of various constants appearing in the closed form approximations (such as Pickands constant, Lipschitz-Killing Curvature, etc), neither do they require the fine knowledge of the local expansions. In addition to the tail probabilities and the conditional expectations, the proposed estimator also provides means to compute the Pickcands constants. This application will be discussed in Remark \ref{rempick}.

The rest of this paper is organized as follows. In Section \ref{SecPre}, we present the problem settings and some existing results that we will refer to in the later analysis.
Section \ref{SecMain} presents the Monte Carlo methods and their efficiency results.
Numerical implementations are included in Section \ref{SecNum}.
Sections \ref{SecCont}, \ref{SecDis}, and \ref{SecInt} include the proofs of the  theorems.

\section{Preliminaries: Gaussian random fields and rare-event simulation}\label{SecPre}

%In this section, we discuss several existing results of Gaussian random fields and notions of rare-event simulations.

\subsection{Gaussian random fields}

Throughout this paper, we consider a Gaussian random field living on a $d$-dimensional compact subset $T\subset R^d$, that is, for any finite subset $(t_1,...,t_n)\subset T$, $(f(t_1),...,f(t_n))$ is a multivariate Gaussian random vector. For each $s,t\in T$, we define the following functions,
\begin{eqnarray*}
&&\mu(t) = E(f(t)),\qquad C(s,t) = Cov( f(s),f(t) ), \qquad\mu_T = \sup_{t\in T} |\mu(t)|,\\
&&\sigma^2(t) = C(t,t), \qquad \sigma^2_T = \sup_{t\in T} \sigma^2(t), \qquad r(s,t)=\frac{C(s,t)}{\sigma(s)\sigma(t)}.
\end{eqnarray*}
In this paper, we are mostly interested in the high excursion probability
\begin{equation*}
w(b) = P(M > b).
\end{equation*}
In addition to the tail probabilities, we also present the analysis concerning the integrals on the  excursion set.
Let  $A_\gamma$ be the excursion set over the level $\gamma$
	\begin{equation}\label{ex}
	A_\gamma = \{t\in T: f(t) > \gamma\}.
	\end{equation}
We define the integral
\begin{equation}\label{alpha}
\alpha(b) = \int_{A_b} \xi(t)dt.
\end{equation}
where  $\xi(t)$ is another random field living on $T$. Then we are interested in computing the conditional expectation
\begin{equation*}
  v(b)=E(\alpha(b)|M>b)
\end{equation*}
We now state the technical conditions that require the following definition.
\begin{definition}
A function $L$ is said to be slowly varying at zero if
$$\lim_{x\to 0} \frac{L(tx)}{L(x)} = 1, \qquad \mbox{for all $t\in (0,1)$.}$$
\end{definition}
Throughout this paper, we impose the following technical conditions.
\begin{enumerate}

 \item[A1] The process $f(t)$ is almost surely continuous in $t$.

\item [A2]For some $\alpha_1 \in (0,2]$, the correlation function satisfies the following local expansion
\begin{equation}\label{corexp}
1 - r(s,t) \sim  \Delta_s L_1(|t-s|) |t-s|^{\alpha_1},\quad \mbox{as $t\to s$}
\end{equation}
where $\Delta_s\in (0,\infty)$ is continuous in $s$ and $L_1$ is a slowly varying function at zero.
Furthermore, there exist nonnegative constants $\kappa_r, \beta_0$, and positive constant $\beta_1>0$ satisfying $\beta_0+\beta_1\geq\alpha_1$ such that
\begin{equation}\label{corbd}
  |r(t,t+s_1)-r(t,t+s_2)|\leq \kappa_rL_1(|s_1|)|s_1|^{\beta_0}|s_1-s_2|^{\beta_1} \qquad \mbox{for $|s_1|\leq |s_2|$.}
\end{equation}

\item[A3] The correlation function is non-degenerate, that is,  $r(s,t)<1$ for all $s\neq t$.

\item[A4] The standard deviation $\sigma(t)$ belongs to either of the following two types.
\begin{enumerate}
  \item[Type 1] $\sigma(t)=1$ for all $t\in T$.
  \item[Type 2]  $\sigma(t)$  has a unique maximum attained at $t^*$ satisfies the following conditions
    \begin{eqnarray}\label{sdbd}
      |\sigma(t)-\sigma(s)|&\leq& \kappa_\sigma \times L_2(|t-s|)\times |t-s|^{\alpha_2}\qquad  \mbox{for all $s,t\in T$},\\
      \sigma(t^*) - \sigma(t)&\sim& \Lambda \times L_2(|t-t^*|)\times |t-t|^{\alpha_2}\qquad \mbox{as $t\to t^*$}, \label{sdexp}
    \end{eqnarray}
%       In addition, at its maximum, $\sigma(t)$ admits expansion
where $\alpha_2\in (0,1]$, $\Lambda>0$, and $L_2$ is a slowly varying function at zero such that the limit $\lim_{x\to 0+}\frac{L_1(x)}{L_2(x)}$ exists.
\end{enumerate}
\item[A5] There exists $\kappa_\mu>0$ such that if $\sigma(t)$ is of Type 1 then
$|\mu(s)-\mu(s+t)|\leq\kappa_\mu \sqrt{L_1(|t|)} |t|^{\alpha_1/2}$; if $\sigma(t)$ is of Type 2 then $|\mu(s)-\mu(s+t)|\leq\kappa_\mu \sqrt{L_2(|t|)} |t|^{\alpha_2/2}$.
\item[A6]There exist $ \kappa_m$  and $\epsilon$ small enough, such that  $mes(B(t,\epsilon)\cap T)\geq \kappa_m\epsilon^d\omega_d$, for any $ t\in T$,  where $B(t,\epsilon)$ is the $\epsilon$-ball centered around $t$ and $\omega_d$ is the volume of the $d-$dimensional unit ball.

\end{enumerate}

Condition A2 ensures that the normalized process $\frac{f(t) - \mu(t)}{\sigma(t)}$ is H\"older continuous with coefficient $\alpha_1/2$.
	The bound in \eqref{corbd} imposes slightly more conditions. For instance, in case when $1- r(s,t) = |t-s|^{\alpha_1}$, we can choose that $\beta_0 =\alpha_1-1$ and $\beta_1=1$ if $\alpha_1 \geq 1$; $\beta_0 =0$ and $\beta_1 = \alpha_1$  if $0<\alpha_1 <1$.
	Condition A3 excludes the degenerated cases that are not essential and it makes the technical development more concise. Conditions A4 and A5 require that the mean and the standard deviation functions are also H\"older continuous.
In Condition A4, we can adjust the constant  $\Lambda$  such that the limit $\lim_{x\to 0+} {L_1(x)}/{L_2(x)}$ belongs to the set $\{0,1,\infty\}$.
Condition A5 ensures that the variation of the mean function is bounded by those of $f(t)$ and $\sigma(t)$.
In the later technical developments, the analysis is divided into two cases $\alpha_1 < \alpha_2$ and $\alpha_1\geq\alpha_2$.
%Condition A6 imposes continuity requirements on $T$.

Throughout this paper, we use the following notations for the asymptotics. We write $h(b) = o(g(b))$ if $h(b)/g(b) \to 0$ as $b\to \infty$; $h(b) = O(g(b))$ if $h(b) \leq \kappa g(b)$ for some $\kappa>0$; $h(b)=\Theta(g(b))$ if $h(b) =O(g(b))$ and $g(b) = O(h(b))$; $h(b)\sim g(b)$ if $h(b) / g(b) \to 1$ as $b\to \infty$.

\subsection{Rare-event simulation and importance sampling}\label{SecIS}
	\subsubsection{Rare-event simulation.}
	The research focus of rare-event simulation is on estimating $w=P(B)$, where $P(B)\approx0$. It is customary to introduce a parameter, say $b>0$, with a meaningful interpretation from an applied standpoint such that  $w(b)\rightarrow0$ as $b\rightarrow \infty$.
	Consider an estimator $Z_{b}$ such that $EZ_{b}=w(b)$.
	A popular efficiency concept in the rare-event simulation literature is the so-called strong efficiency that is defined as follows (c.f.~\cite{ASMGLY07,BUC04,JunSha06}).

\begin{definition}
A Monte Carlo estimator $Z_b$ is said to be \emph{strongly efficient} in estimating $w(b)$ if $E (Z_b) = w(b)$ and there exists a $\kappa_0 \in (0,\infty)$ such that
$$\sup_{b>0}\frac{Var(Z_{b})}{w^{2}(b)}<\kappa_0.$$
\end{definition}
	Strong efficiency measures mean squared error in relative terms for an unbiased estimator.
	Suppose that a strongly efficient estimator of $w(b)$ has been constructed, denoted by $Z_b$,  and $n$ i.i.d.~replicates of $Z_b$ are generated $Z_b^{(1)},...,Z_b^{(n)}$. Let
	\begin{equation*}
	\bar Z_{b,n} \triangleq \frac 1 n \sum_{i=1}^n Z_b^{(i)}
	\end{equation*}
	be the averaged estimator that has an variance
	$$Var (\bar Z_{b,n}) = \frac{Var(Z_b)}{n}.$$
	By means of the Chebyshev's inequality, we obtain that
	$$P\left(|\bar Z_{b,n} - w(b)| > \varepsilon w(b)\right) \leq \frac{Var (Z_b)}{n\varepsilon^2 w^2(b)}.$$
	For any $\delta > 0$, to achieve the $\varepsilon -\delta$ accuracy, we need to generate $$n = \frac{Var (Z_b)}{\delta\varepsilon^2 w^2(b)}\leq \frac{\kappa_0}{\delta\varepsilon^2 }$$
	replicates of $Z_b$. This choice of $n$ is uniform in the rarity parameter $b$. We will later show that the proposed continuous importance sampling estimator is strongly efficient.
	Besides strong efficiency, another weaker concept is the so-called weak/asymptotic efficiency, that is,
	 $$\lim_{b\rightarrow \infty}\frac{\log Var(Z_{b})}{2\log w(b)} =1.$$
    Weak efficiency implies that $Var(Z_b)=o(w(b)^\varepsilon)$ for any $\varepsilon>0$.

	\subsubsection{Importance sampling and variance reduction.}

	Importance sampling is based on the basic identity,
\begin{equation}
P(  B)  =\int I\left(  \omega\in B\right)  dP\left(  \omega\right)
=\int I\left(  \omega\in B\right)  \frac{dP}{dQ}\left(  \omega\right)
dQ\left(  \omega\right)  \quad \mbox{for a measurable set $B$,} \label{Id1}%
\end{equation}
where we assume that the probability measure $Q $ is such
that $Q(  \cdot\cap B)  $ is absolutely continuous with respect to
the measure $P\left(  \cdot\cap B\right)  $. If we use $E^{Q}  $ to denote
 expectation under  $Q $, then (\ref{Id1}) trivially yields
 that the random variable
\[
Z\left(  \omega\right)  =I\left(  \omega\in B\right)  \frac{dP}{dQ}\left(
\omega\right)
\]
is an unbiased estimator of $P\left(  B\right)  >0$ under the measure
$Q$, or symbolically, $E^{Q}Z=P\left(  B\right)  $.
%
%An averaged importance sampling estimator based on the measure $Q$, which is often referred as an \textit{importance sampling
%distribution} or a \textit{change-of-measure}, is obtained by simulating $n$
%i.i.d.\ copies $L^{(1)},...,L^{(n)}$ of $L$ under $Q $ and
%computing the empirical average $\widehat{L}_{n}=(L^{(1)}+...+L^{(n)})/n$.

A central component lies in the selection of $Q$ in order to minimize the variance of $Z$. It is easy to verify that if we choose $\mathcal{Q}^{\ast}(\cdot)=P(\cdot|B)=P(\cdot\cap B)/P(B)$ then the
corresponding estimator has zero variance and thus it is usually referred to as the the \emph{zero-variance change of measure}.
However, $\mathcal{Q}^{\ast}$ is clearly a change of measure that is  of no practical value, since
$P\left(  B\right)  $ -- the quantity that we are attempting to evaluate in
the first place -- is unknown. Nevertheless, when constructing a good
importance sampling distribution for a family of sets $\{B_{b}: b\geq b_{0}\}$
for which $0<P\left(  B_{b}\right)  \rightarrow0$ as $b\rightarrow\infty$, it
is often useful to analyze the asymptotic behavior of $\mathcal{Q}^{\ast}$ as $P\left(  B_{b}\right)  \rightarrow0$ in order to
guide the construction of a useful $Q$.

We now describe briefly how an efficient importance sampling estimator for
$P\left(  B_{b}\right)  $ can also be used to estimate a large class of
conditional expectations given $B_{b}$. Suppose that an importance sampling estimator has been constructed
\begin{equation*}
Z_{b}\ \definedas \ I(  \omega\in B_{b})  \frac{dP}{dQ},
\end{equation*}
% can be generated in $O(  1)$ function evaluations,
% for some $q_{0}>0$,  and that
such that $Var\left(  L_{b}\right)  =O\left(P(  B_{b})^{2}\right)$.
 Then, by noting that%
\begin{equation}
\frac{E^{Q}\left(  XZ_{b}\right)  }{E^{Q}\left(  Z_{b}\right)  }=\frac{E[X;B_{b}]}{P(B_{b})  }=E[\left.
X\right\vert B_{b}], \label{eq2}%
\end{equation}
it follows easily that an estimator can be naturally obtained; i.e.~the ratio of the
corresponding averaged importance sampling estimators suggested by the ratio
in the left of (\ref{eq2}). Of course, when $X$ is difficult to simulate
exactly, one must assume that the bias in estimating $E[X;B_{b}]$ can be reduced with certain computational costs.

	\subsubsection{The bias control.}
	In addition to the variance control, one also needs to account for the computational effort required to generate $Z_{b}$. This issue is especially important for the current study.
	The random objects in this analysis are continuous processes. For the implementation, we need to use a discrete object to approximate the continuous process. Inevitably, the discretization induces bias, though it vanishes as the size of the discretization increases.
	To ensure the $\varepsilon -\delta$ relative  accuracy, the bias  needs to be controlled to a level less than $\varepsilon w(b)$.
	% that converges to zero as $b\rightarrow \infty$. Usually, the size of the discretization needs to increase to infinity as $b$ tends to infinity.
	
%	\begin{remark}
	The discretized estimators in \cite{adler2012efficient} can be shown to be weakly efficient for general uniformly H\"older continuous Gaussian processes and it is strongly efficient when the process is  twice differentiable and homogeneous.
	The analysis of the H\"older continuous fields  relies heavily on the discrete nature of the estimators.
	For the implementation, it is established that, to ensure a bias of order $\varepsilon w(b)$, the size of the discretization must  grow at a polynomial rate of $b$ for both differentiable and non-differentiable fields. The authors also provide an optimality result.
	For twice differentiable and homogeneous fields, the size of a prefixed/deterministic set $T_m$ must be at least of order $O(b^d)$ so that the bias can be controlled to level $\varepsilon w(b)$.

\section{Main results} \label{SecMain}

The main results of this paper consist of the construction of a change of measure on the continuous sample path space on $T$, a random discretization scheme of $T$ associated with the change of measure $Q_b$, and lastly the efficiency results including the strong efficiency of the continuous and the discrete estimators and the complexity analysis of the discretization scheme.

\subsection {The change of measure}
	As discussed previously, a key element of the analysis is the construction of a change of measure $Q_b$ (indexed by the rarity parameter $b$) that approximates the conditional measure $P(f\in\cdot ~ | M > b)$. We should be able to compute the Radon-Nikodym derivative and also be able to simulate the process $f$ under $Q_b$.
	We describe the measure $Q_b$ from two aspects. First, we present  its Radon-Nikokym derivative with respect to $P$
	\begin{equation}\label{LR}
	\frac{dQ_b}{dP} (f) = \int_T h_b(t) \frac{q_{b,t}(f(t))}{\varphi_t(f(t))}dt,
	\end{equation}
	where $h_b(t)$ is a density function on  the set $T$, $q_{b,t}(x)$ is a density function on the real line, and $\varphi_t(x)$ is the density function of $f(t)$ under the measure $P$ evaluated at $f(t)=x$. We will need to choose $h_b(t)$ and $q_{b,t}(x)$ such that the measure $Q_b$ satisfies the absolute continuity condition to guarantee the unbiasedness.

	We will present the specific forms of $h_b(t)$ and $q_{b,t}(x)$ momentarily. Before that, we would like to complete the description of $Q_b$ by presenting the simulation method of $f$ under $Q_b$. 	
	
	\begin{algorithm}[Continuous simulation]\label{AlgCont}
	To generate a random sample path under the measure $Q_b$, we  need a three-step procedure.
	\begin{itemize}
	\item[Step 1.] Generate a random index $\tau\in T$ following the density $h_b(t)$.
	\item[Step 2.] Conditional on the realization of $\tau$, sample $f(\tau)$ from the density $q_{b,\tau}(x)$.
	\item[Step 3.] Conditional on the realization of $(\tau, f(\tau))$, generate $\{f(t): t\neq \tau\}$ from the original conditional distribution $P(f\in \cdot ~| f(\tau))$.
	\end{itemize}
	\end{algorithm}
	It is not difficult to verify that the above three-step procedure is consistent with the Randon-Nikodym derivative given as in \eqref{LR}. In particular, a random index $\tau$ is first sampled according to the density $h_b(t)$. Second, the random field at the location $\tau$, $f(\tau)$, is sampled from the distribution $q_{b,\tau}(x)$. Lastly, the rest of the  field is sampled according the nominal/original condition distribution. The process $f(t)$  mostly follows the distribution under $P$ except at one random location $\tau$ where the process is twisted to follow an alternative distribution $q_{b,\tau}(x)$. Therefore, the overall Randon-Nikodym derivative is an average of the likelihood ratio $q_{b,t}(f(t))/\varphi_t(f(t))$ with respect to the density $h_b(t)$.
	
	Now, we present the specific forms of $h_b(t)$ and $q_{b,t}(x)$ for the computation of $w(b)$.
	For some positive constant $a$, let $\gamma$ be
	\begin{equation}\label{gamma}
	\gamma = b - a/b.
	\end{equation}
	We choose
	\begin{equation}\label{g}
	q_{b,t}(x) = \varphi_t(x) \frac{I(f(t)> \gamma)}{P(f(t)> \gamma)}
	\end{equation}
	that is the conditional distribution of $f(t)$ given that $f(t) > \gamma$.
	The distribution of $\tau$ is chosen to be
	\begin{equation}\label{h}
	h_b(t) = \frac{P(f(t) > \gamma)}{ \int_T P(f(t) > \gamma )dt}.
	\end{equation}
	The choice of $a$ in \eqref{gamma} does not affect the  efficiency results, nor the complexity analysis. To simplify the discussion, we fix $a$ to be unity, that is,
	\begin{equation}\label{gamma1}
	\gamma = b - 1/b.
	\end{equation}
	As we explained previously, the random index $\tau$ indicates the location where the random field is twisted. Furthermore, $q_{b,t}(x)$ is chosen to be the conditional distribution given a high excursion.
	We emphasize that it is necessary to set $\gamma$ slightly lower than the target level $b$. This will technically provide a stochastic bound on the distribution of the likelihood ratio.
	The index $\tau$ basically localizes the maximum of $f(t)$. Note that $\tau$ is not necessarily, but is very close to, $t_* \triangleq \arg\sup f(t)$.
Thus, as an approximation of the zero-variance change of measure, the distribution $h_b(t)$ should be chosen close to the conditional distribution of the maximum $t_*$ given that $f(t_*) > b$. This is our guideline to  choose $h_b(t).$ For each $t\in T$, the conditional probability that $f(t) > b$ given  $M > b$ is
	\begin{equation*}
	P(f(t) > b | M > b) = \frac{P(f(t) > b)}{P(M > b)}.
	\end{equation*}
	The denominator $P(M>b)$ is free of $t$ and thus $P(f(t) > b | M > b)\propto P(f(t) > b)$. Our choice of $h_b(t)\propto P(f(t) > \gamma)$  approximates  $P(f(t) > b | M > b)$ by replacing $b$ with $\gamma$ which is a very minor twist. This twist allows quite a  lot of technical convenience.
	With such choices of $h_b(t)$ and $q_{b,t}(x)$, the Radon-Nikodym takes the following form
	\begin{equation}\label{LRG}
	\frac{dQ_b}{dP} = \frac{\int_T I(f(t) >\gamma)dt}{\int_T P(f(t) > \gamma )dt }= \frac{mes(A_\gamma)}{\int_T P(f(t) > \gamma )dt},
	\end{equation}
	where $mes(\cdot)$ is the Lebesgue measure. According to Fubini's theorem, the denominator has another representation:
    \begin{equation*}
      \int_TP(f(t)>\gamma)dt=E[mes(A_\gamma)].
    \end{equation*}
	
	\begin{remark}
	For different problems, we may choose different $h_b(t)$ and $q_{b,t}(x)$ to approximate various conditional distributions. For instance, $q_{b,t}(x)$ was chosen to be in the exponential family of $\varphi_t(x)$ in \cite{liu2012conditional} for the derivation of tail approximations of $\int e^{f(t)}dt$.
	\end{remark}

\subsection{An adaptive discretization scheme and the algorithms}

\subsubsection{The continuous estimator and the challenges}
Based on the change of measure $Q_b$, a natural estimator for $w(b)$ is given by
\begin{equation}\label{est}
Z_b \triangleq I(M>b)\frac{dP}{dQ_b} = I(M>b) \frac{\int_T P(f(t) > \gamma)dt}{mes(A_\gamma)}.
\end{equation}
It is straightforward to obtain that $E_b (Z_b) = w(b)$,
where we use $E_b(\cdot )$ to denote the expectation under the measure $Q_b$. The second moment of $Z_b$ is given by
\begin{equation*}
E_b (Z_b^2) = E_b \Big[\frac{(\int_T P(f(t) > \gamma)dt)^2}{mes^2(A_\gamma)};M> b\Big].
\end{equation*}
We will later show that this continuous estimator (under regularity conditions) is strongly efficient, that is, $E_b(Z_b ^2) =O(w^2(b))$.
Similarly, a natural estimator for the numerator $E(\alpha(b);M>b)$ in \eqref{eq2} is
\begin{equation}\label{estX}
Y_b \triangleq \frac{\alpha(b)}{mes(A_\gamma)}\int_T P(f(t) > \gamma)dt,
\end{equation}
which, under regularity conditions, will be shown to estimate $E(\alpha(b) ; M>b)$ with strong efficiency.

For the implementation, we are not able to simulate the continuous field $f$ and therefore have to adopt a simulatable estimator, $\hat Z_b$, that approximates the continuous estimator $Z_b$. A natural approach is to consider the random field on  a finite set $T_m = \{t_1,...,t_m\}\subset T$ and use $P(\max_{T_m} f(t_i) > b)$ as an approximation of $w(b) = P(\sup_T f(t) > b)$. The bias is given by
$$P(\sup_T f(t) > b) - P(\max_{T_m} f(t) > b) = P(T_m\cap A_b = \emptyset, M > b).$$

In what follows, we explain without rigorous derivation that the above scheme usually induces  a heavy computational overhead. To simplify the discussion, we consider the special case that $f$ is a stationary process and its covariance function  satisfies the local expansion (slightly abusing the notation)
\begin{equation}\label{covhomo}
C (t)\triangleq Cov(f(s), f(s+t))  = 1 - |t|^\alpha + o(|t|^\alpha)
\end{equation}
Then, the process is H\"older continuous with coefficient $\alpha/2$. Under this setting, standard results yield the following estimate of the excursion set
$$E(mes(A_b) | M > b) = \Theta (b^{-2d/\alpha}).$$
Thanks to stationarity, conditional on the event $\{M>b\}$, the excursion set $A_b$ is a random subset of $T$, whose volume is of order $\Theta(b^{-2d/\alpha})$ and which is approximately uniformly distributed over the domain $T$.

	Notice that the bias term $P(T_m\cap A_b = \emptyset, M > b)$ is the probability that $T_m$ does not intersect with $A_b$.
	Therefore, if $m\ll b^{2d/\alpha}$, $T_m$ is too sparse such that it is not able to catch the set $A_b$ no matter how $T_m$ is distributed over $T$. Therefore,  it is necessary to have a lattice of size at least of order $O(b^{2d/\alpha})$. This heuristic calculation was made rigorous for smooth fields in \cite{adler2012efficient}.
	Thus, the computational complexity to generate the process $f$ on the set $T_m$ grows at a polynomial rate with $b$.
	In this paper, we aim at further reduction of the discretization size to a constant level
while still maintaining the $\varepsilon$-relative bias. For this sake, we need to seek among the random discrete sets.

\subsubsection{A closer look at the excursion set $A_\gamma$}

	%\jc{To explain the discretization scheme, we  need the following reparameterization. First, we define $\alpha=\alpha_1$, $L(x)=L_1(x)$ for type 1 standard deviation, and for type 2 standard deviation, we define $\alpha$ and $L$ as follows.
%\begin{equation*}
%\alpha=\min(\alpha_1,\alpha_2),
%\end{equation*}
%$$
%L(x)=\left\{\begin{array}{rcl}
%                L_1(x) &,& \text{ if } \displaystyle\lim_{t\to 0}\frac{L_1(t)t^{\alpha_1}}{L_2(t)t^{\alpha_2}}=0\\
%                L_2(x) &,& \text{ otherwise. }
%                \end{array}
%                \right.
%$$
%To standardize the random process $f(x)$,  we introduce the scalars  $\zeta_i$,  $i=1,2$.
%	\begin{equation}\label{zeta1}
%	\zeta_i \triangleq \max\Big\{|s|^{-1}: L_i(|s|)|s|^{\alpha_i} \geq b^{-2}\Big\}, \quad i=1,2.
%	\end{equation}
%We also define scalar $\zeta$,
%	\begin{equation}\label{zeta}
%	\zeta \triangleq \max\Big\{|s|^{-1}: L(|s|)|s|^{\alpha} \geq b^{-2}\Big\},
%	\end{equation}
%that characterizes the cluster size of $f$.
%Asymptotically, we have
%\begin{equation*}
%  \lim_{b\to\infty}\frac{\max(\zeta_1,\zeta_2)}{\zeta}=1.
%\end{equation*}}

The proposed adaptive discretization scheme is closely associated with the three step simulation procedure under $Q_b$ and furthermore the distribution of $A_\gamma$. Among the three steps in Algorithm \ref{AlgCont}, Step 1 and Step 2 are implementable. It is Step 3, generating $\{f(t): t\neq \tau\}$ conditional on $(\tau, f(\tau))$, that requires discretization. In order to estimate $w(b)$ and to generate the estimator $Z_b$, we only need to simulate the random indicator $I(M>b)$ and  the volume of the excursion set $mes(A_\gamma)$ conditional on $(\tau,f(\tau))$. The term $\int_T P(f(t) > \gamma) dt$ is a deterministic number that can be computed via routine numerical methods.

	In what follows, we focus on the simulation and  approximation of $I(M>b)$ and $mes(A_\gamma)$. For illustration purpose, we provide the discussion for the homogeneous case with covariance function satisfying the expansion \eqref{covhomo}. We define
	$\zeta = b^{2/\alpha}$
	that characterizes the cluster size of $f$. Furthermore, we define the normalized process
	\begin{equation}\label{gt}
	g(t) = b(f(\tau + t/\zeta)-b).
	\end{equation}
	Note that $b\times (f(\tau) - \gamma )$ asymptotically follows an exponential distribution.
	Conditional on
	$f(\tau) = \gamma + z/b$
	the $g$ process has expectation
	\begin{eqnarray*}
	E_b[g(t) | f (\tau) = \gamma + z/b] &=& z-1- (1+o(1))|t/\zeta|^\alpha [b^2+ (z-1) ] .
	\end{eqnarray*}
	For all $z =o(b^2)$, we have that
	\begin{eqnarray*}
	E_b[g(t) | f (\tau) = \gamma + z/b] = z-1- (1+o(1))|t|^\alpha \qquad \mbox{as $b\to \infty.$}
	\end{eqnarray*}
%	The $O(|t|^{\alpha/2})$ term is in the regime that $t\rightarrow \infty$.
	In addition, the covariance of $g(t)$ is
	$$Cov(g(s), g(t)) = ( |s|^\alpha + |t|^\alpha - |s-t|^\alpha) + o(1)$$
	where $o(1)\rightarrow 0$ as $b\rightarrow \infty$.
	Therefore, the distribution of $g(t)$ converges weakly to a Gaussian process with the above mean and covariance function. In addition, $f(\tau + t/\zeta)\geq \gamma$ if and only if $g(t) > 0$. The excursion set $A_\gamma$ can be written as
	\begin{equation*}
	A_\gamma = \tau + \zeta^{-1}\cdot A_{-1}^g \triangleq \{\tau + \zeta^{-1} t: t\in A_{0}^g\}.
	\end{equation*}
	where $A_{-1}^g =\{t: g(t)> -1\}$. Note that the process $g(t)$ is a Gaussian process with standard deviation $O(|t|^{\alpha/2})$ and a negative drift of order $O(-|t|^\alpha)$. Therefore, in expectation,  $g(t)$ goes below $0$ when $z\ll |t|^\alpha  $ where $z$ is asymptotically an exponential random variable. Thus, the excursion set $A^g_{-1}$ is of order $O(1)$. Furthermore,  $A_\gamma$ is a random set within $O(\zeta^{-1})$ distance from the random index $\tau$. The volume $mes(A_\gamma)$ is of order $O(\zeta^{-d})$.
	The above discussion quantifies the intuition that $\tau$ localizes the global maximum of $f$. It also localizes the excursion set $A_\gamma$. Therefore, upon considering approximating/computing $mes(A_\gamma)$ and $I(M>b)$, we should focus on the region around $\tau$.
	
	Conditional on a specific realization of the process $f$, we formulate the approximation of $mes(A_\gamma)$ as estimation problem. Note that the ratio $mes(A_\gamma)/ mes(T) \in [0,1]$ corresponds to the following probability
	$$\frac{mes(A_\gamma)}{mes(T) }=  P(U\in A_\gamma)$$
	where $U$ is a uniform random variable on the set $T$ with respect to the Lebesgue measure. Estimating $mes(A_\gamma)$ constitutes another rare-event simulation problem.
	%As discussed previously, we have a lot of knowledge about $A_\gamma$.

\subsubsection{An adaptive discretization scheme.}

Based on the  understanding of the excursion set $A_\gamma$, we setup a discretization scheme adaptive to the realization of $\tau$. To proceed, we provide the general form of $\zeta$ in presence of  slowly varying functions
	\begin{equation}\label{zeta}
	\zeta \triangleq \max\left\{|s|^{-1}: L_1(|s|)|s|^{\alpha_1} \geq b^{-2}~\mbox{or }L_2(|s|)|s|^{\alpha_2} \geq b^{-2}\right\}.
	\end{equation}
	%Note that $L_2$ and $\alpha_2$ are defined for the function $\sigma(t)$.
	In the case of constant variance, we formally define $\alpha_2 = \infty$ and thus $\zeta$ is defined as $\zeta \triangleq \max\{|s|^{-1}: L_1(|s|)|s|^{\alpha_1} \geq b^{-2}\}.$
	To facilitate the later discussion, we define  two  other scale factors
	\begin{equation}\label{zeta1}
	\zeta_i \triangleq \max\left\{|s|^{-1}: L_i(|s|)|s|^{\alpha_i} \geq b^{-2}\right\}, \quad i=1,2.
	\end{equation}
	%Note that the slowly varying function $L_i(x) = o(x^{\varepsilon})$ varies slower than any polynomial.
	Thus, it is straightforward to verify that
	\begin{equation*}
	\zeta = \max(\zeta_1,\zeta_2).
	\end{equation*}
Consider an isotropic distribution (centered around zero) with density $k(t)$, that is, $k(t) = k(s)$ if $|s| = |t|$.  We choose $k(t)$ to be reasonably heavy-tailed such that for some $\varepsilon_1 >0$
\begin{equation*}
k(t) \sim |t|^{-d-\varepsilon_1},\qquad \mbox{as $t\rightarrow \infty$.}
\end{equation*}
In addition there exists a $\kappa_1>0$ such that $k(t)\leq \kappa_1$ for all $t$.
For instance, we can choose $k(t)$ to be, but not necessarily restricted to, the multivariate $t$-distribution.
Furthermore, conditional on $\tau$, we define the rescaled density
\begin{equation}\label{l}
k_{\tau,\zeta} (t) =\zeta^d \times k( \zeta(t-\tau))
\end{equation}
that centers around $\tau$ and has scale  $\zeta^{-1}$.
We construct a $\tau$-adapted random subset of $T$ by generating  i.i.d.~random variables from the density $k_{\tau,\zeta}(t)$, denoted by $t_1,...,t_m$. Then, define
\begin{equation}\label{vol}
\widehat{mes} (A_\gamma) \triangleq \frac{1}{m}\sum_{i=1}^m \frac{I(f(t_i)>\gamma)}{k_{\tau,\zeta}(t_i)}
\end{equation}
that is an unbiased estimator of $mes(A_\gamma)$ in the sense that for each realization of $f$
$$E_{\tau,\zeta} [\widehat{mes}(A_\gamma)| f ] = mes (A_\gamma)$$
where $E_{\tau,\zeta} (\cdot | f)$ is the expectation with respect to $t_1,...,t_m$  under the density $k_{\tau,\zeta}$  for a particular realization of $f$. Notationally, if $t_i\notin T$, then $I(f(t_i) > \gamma) = 0$.

Similar to the approximation of $mes(A_\gamma)$, we use the same $\tau$-adapted random subset to approximate $I(M > b)$, that is,
$$I(\max_{i=1}^m f(t_i) > b) \approx I(M > b).$$
Based on the above discussions, we present the final algorithm.
\begin{algorithm}\label{AlgDis}
The algorithm consists of the following steps.
\begin{itemize}
	\item[Step 1.] Generate a random index $\tau\in T$ following the density $h_b(t)$ in \eqref{h}.
	\item[Step 2.] Conditional on the realization of $\tau$, sample $f(\tau)$ from $q_{b,t}(x)$ in \eqref{g}.
	\item[Step 3.] Conditional on the realization of $\tau$, generate i.i.d.~random indices $t_1,...,t_m$ following density $k_{\tau,\zeta}(t).$
	\item[Step 4.] Conditional on the realization of $(\tau, f(\tau))$, generate multivariate normal random vector $(f(t_1),...,f(t_m))$ from the original/nominal conditional distribution of $P(\cdot | f(\tau))$.
	\item[Step 5.] Output
	\begin{equation*}
	\hat Z_b = \frac{I(\max_{i=1}^m f(t_i) > b)}{\widehat{mes} (A_\gamma)} \int_T P(f(t)> \gamma)dt,
	\end{equation*}
	where $\widehat{mes} (A_\gamma)$ is given as in \eqref{vol}.
\end{itemize}
\end{algorithm}

For the discrete version of the estimator $ Y_b$  as in \eqref{estX}, we approximate it in a similar way. In Step 4 of the above algorithm, we simulate $\{(f(t_i),\xi(t_i)): i=1,...,m\}$ jointly conditional on  $(\tau,f(\tau))$.
Then, we output the estimator
\begin{equation*}
\hat Y_b = \frac{\hat \alpha(A_b)}{\widehat{mes}(A_\gamma)} \int_T P(f(t)> \gamma)dt%I(\max_{i=1}^m f(t_i) > b)
\end{equation*}
where
\begin{equation}\label{vol2}
\hat \alpha(A_b) \triangleq \frac{1}{m}\sum_{i=1}^m \frac{\xi(t_i)}{k_{\tau,\zeta}(t_i)}I(f(t_i)>b).
\end{equation}

\subsection{The main results}

We present the efficiency results of the proposed algorithms. The first theorem establishes that the continuous estimator is strongly efficient.

\begin{theorem}\label{ThmCont}
Consider a Gaussian random field $f$ that satisfies conditions A1-6. Let $Z_{b}$ be given as in \eqref{est} and Algorithm \ref{AlgCont}. Then, $Z_{b}$ is strongly efficient in estimating $w(b)$, that is, there exists $\kappa_{0}$ such that
$$E_b(Z_{b}^{2}) \leq \kappa_0 w^2(b)$$
for all $b>0$.
\end{theorem}

%\begin{theorem}\label{ThmCont'}
%Consider a Gaussian random field $f$ that satisfies conditions A1-6. Let $Z_{b}$ be given as in \eqref{est} and Algorithm \ref{AlgCont}. Then, $Z_{b}$ be strongly efficient in estimating $w(b)$, that is, there exists a $\kappa_{0}$ such that
%$$E_b(Z_{b}^{2}) \leq \kappa_0 w^2(b)$$
%for all $b>0$.
%\end{theorem}

%\begin{remark}
%From our sampling method, the first step is to sample $Q$ is to sample $\tau$ according to $h_b(t)$. From \eqref{h}, $h_b(t)$ has different behavior for different type of $\sigma(t)$: if $\sigma(t)$ is a constant over $T$, then $h_b(t)$ is the density of a uniform distribution; if $\sigma(t)$ has a unique maximum at $t^*$ and has local expansion around the maximum, then $h_b(t)$ will tend to the point mass at $t^*$ as $b\to\infty$. We proved the efficiency of our sampling method under both two kinds of variance function.
%\end{remark}
The next theorem establishes the computation complexity of the discrete estimator.

\begin{theorem}\label{ThmDis}
	Consider a Gaussian random field $f$ that satisfies conditions A1-6.
	Let $\hat Z_b$ be the estimator given by Algorithm \ref{AlgDis}. There exists $\lambda>0$ such that for any $\varepsilon > 0$ if we choose $m=\lambda\varepsilon^{-d(2/\min(\alpha_1,\alpha_2)+2/\beta_1)}$, then
	$$|E_b(\hat Z_b)  - w(b)|\leq \varepsilon w(b)$$
	for all $b>0$. Furthermore, there exists $\kappa_0$ such that
	$$E_b(\hat Z_b^2) \leq \kappa_0 w^2(b). $$
\end{theorem}

With the above results, we generate $n$ i.i.d.~replicates of $\hat Z_b$, denoted by $\hat Z_b^{(1)}$, ...,$\hat Z_b^{(n)}$, with $m$ chosen as in the theorem such that the averaged estimator, $\frac 1 m \sum_{i=1}^n \hat Z_b^{(i)}$, has its bias bounded by $\varepsilon  w(b)/2$ and its variance is bounded by $\kappa_0 w^2(b) / n$. To achieve $\varepsilon$ relative error with at $(1-\delta)$ confidence, we need to choose $n= \frac{4\kappa_0}{\varepsilon^2\delta}$, that is,
$$P\Big(\Big|\frac 1 m \sum_{i=1}^n \hat Z_b^{(i)} - w(b)\Big|>\varepsilon w(b)\Big)< \delta$$
and the total computational complexity is of order $O(m^3\varepsilon^{-2}\delta^{-1})$, where $m^3$ is the complexity of computing the eigenvalue of an $m\times m$ covariance matrix.

\begin{theorem}\label{ThmInt}
	Consider a Gaussian random field $f$ that satisfies conditions A1-6. There exists $0<a_1<a_2<\infty$, such that $\xi(t) \in [a_1,a_2]$ almost surely. We have the following results
	\begin{enumerate}
	 \item Then, there exists  $\kappa_0$ such that for all $b>0$
\begin{equation*}
  E_b(Y_b^2 )\leq \kappa_0 u^2(b)
\end{equation*}
where
$u(b) = E(\alpha(b) ; M > b).$
	\item There exists $\lambda$ such that for each $\varepsilon>0$   if we choose $m=\lambda\varepsilon^{-d(2/\min(\alpha_1,\alpha_2)+2/\beta_1)}$
then
	$$|E_b(\hat Y_b) - u(b)| \leq \varepsilon u(b)$$
	and
	$$E_b(\hat Y_b^2) \leq \kappa_0 u^2(b).$$
	\end{enumerate}
\end{theorem}

In the previous theorem, we require that the process $\xi(t)$ take values in a positive interval $[a_1,a_2]$. This constraint is imposed   for technical convenience. There are several ways in which we can relax this condition. If $\xi(t)$ is independent of $f(t)$, then, we can relax the interval to be $(0,\infty)$.  In the case when $\xi(t)\in (0,\infty)$ and $\xi(t)$ and $f(t)$ are dependent, we may need to modify the algorithm. This is because $\xi(t)$ could be very close to zero on the excursion set $A_b$ and therefore
the estimator \eqref{vol2} may not be strongly efficient in estimating $\alpha(t)$. In this case, we may further change the sampling distribution of $\{(f(t_i),\xi(t_i)): i=1,...,m\}$ to reduce the variance of $\hat \alpha(t)$. These modifications have to be case-by-case and they can be handled by routine variance reduction techniques that we do not pursue in this paper.

\begin{remark}
  There are cases that the current setting does not cover. For instance, the process is anisotropic in the sense that  $\alpha$ depends on the direction; see, for instance, \cite{Pit95} for more discussions. We believe that the results of Theorem \ref{ThmCont} hold under this setting. We need to follow the same idea and apply our proof technique in different directions. For the discretization scheme, one needs to define the scale $\zeta$ for different directions and rescale the density $k(t)$ differently among different directions. Thus, we expect the results of Theorem \ref{ThmDis} and Theorem \ref{ThmInt} to hold.
\end{remark}

\begin{remark}\label{rempick}
The current work provides a means to numerically compute the Pickands constant. The basic idea is to numerically compute tail probability $w(b)$ for $b$ very large and for some stationary process living on $[0,1]$ with covariance function $C(t) = e^{-|t|^\alpha}$. Denote the estimate by $\hat w(b)$. Then, an estimate of the Pickands' constant is given by $$\hat H_\alpha = \frac{\hat w(b)} {b^{2/\alpha}P(Z>b)}.$$
\end{remark}

\section{Numerical analysis}\label{SecNum}

In this section, we present four numerical examples to show  the   performance of our algorithms.
First, we applied our algorithm to a one dimensional Gaussian field whose tail probability is in a closed form. For the discretization, we deploy $m=20$ points when $d=1$ and 40 points when $d=2$. To make sure that the bias is small enough, we have run the simulations with 10 times more points and the results didn't change substantially. We only report the results with fewer points to illustrate the efficiency.

\begin{example}
Consider $f(t)=X \cos t+Y\sin t,$ $T=[0,3/4]$, where $X$ and $Y$ are independent standard Gaussian variables. The probability $P(\sup_{t\in T}f(f)>b)$ is known to be in closed form (\cite{Adl81}), and is given by
\begin{equation}\label{Truevalue}
  P(\sup_{0\leq t\leq 3/4 }f(t)>b)=1-\Phi(b)+\frac{3}{8\pi}e^{-b^2/2}.
\end{equation}
Table 1 shows the simulation results. \end{example}
% Table generated by Excel2LaTeX from sheet 'f(t)=Xcos(t)+Ysin(t)'
\begin{table}[htbp]\label{Table1d}
  \centering
    \begin{tabular}{lccc}
	\hline
    b     & true value & est   & std err \\
    \hline
    3     & 2.7E-03 & 2.6E-03 & 1.1E-04 \\
    4     & 7.2E-05 & 7.2E-05 & 3.2E-06 \\
    5     & 7.3E-07 & 7.3E-07 & 3.4E-08 \\
    6     & 2.8E-09 & 2.8E-09 & 1.4E-10 \\
    7     & 4.0E-12 & 4.1E-12 & 2.0E-13 \\
    8     & 2.2E-15 & 2.1E-15 & 8.4E-17 \\
    \hline
    \end{tabular}%
      \caption{Simulation results for the cosine process where n=1000, m=20, $k(t)$ is chosen to be the density function of $t-$distribution with degrees of freedom $3$. The ``true value" is calculated from \eqref{Truevalue}}.

\end{table}%

The following three examples treat random fields over a two dimensional square.

\begin{example}\label{ex2}
  Consider a mean zero, unit variance, stationary and smooth Gaussian field over $T=[0,1]^2$, with covariance function
  \begin{equation*}
    C(t)=e^{-|t|^2}.
  \end{equation*}
  Let $\xi(t)=1$, then $E(\int_{A_b}\xi(t)dt)$ is in a closed form and is given by
  \begin{equation*}
  E\Big(\int_{A_b}\xi(t)dt\Big)=E(mes(A_b))=1-\Phi(b).
  \end{equation*}
  Table \ref{Tableex2} shows the simulation results.
\end{example}

\begin{table}[htbp]\label{Tableex2}
  \centering
    \begin{tabular}{lcccccc}
    \hline
   &\multicolumn{2}{c}{$P(\sup_T f(t)>b)$}&&\multicolumn{3}{c}{$E(mes(A_b))$}\\\cline{2-3}\cline{5-7}
    b     & est   & std err && true value & est   & std err \\
    \hline
3     & 9.3E-03 & 3.6E-04 && 1.3E-03 & 1.4E-03 & 4.0E-05 \\
4     & 3.4E-04 & 1.5E-05 && 3.2E-05 & 3.3E-05 & 9.2E-07 \\
5     & 4.2E-06 & 1.7E-07 && 2.9E-07 & 3.0E-07 & 8.2E-09 \\
6     & 1.9E-08 & 8.1E-10 && 9.9E-10 & 1.0E-09 & 2.8E-11 \\
7     & 3.3E-11 & 1.3E-12 && 1.3E-12 & 1.4E-12 & 3.7E-14 \\
8     & 1.9E-14 & 7.1E-16 && 6.7E-16 & 6.7E-16 & 1.8E-17 \\
\hline
    \end{tabular}%
      \caption{Simulation results for Example \ref{ex2}, where n=1000, m=40. $k(t)=\frac{25}{32\pi}(1+0.64 |t|^2)^{-3}$, the density function of multivariate $t-$distribution with degrees of freedom $4$ , and  $\mu=0$; $\Sigma=0.64 I_2$. }

\end{table}

\begin{example}\label{ex3}
  Consider a continuous inhomogenous Gaussian field on $T=[0,1]^2$ with mean and covariance function
  \begin{equation*}
    \mu(t)=0.1t_1+0.1t_2 \quad C(s,t)=e^{-|t-s|^2}.
  \end{equation*}
  Let $\xi(t)=1$, then $E(\int_{A_b}\xi(t)dt)$ is in a closed form and is given by
  \begin{equation*}
    E\Big(\int_{A_b}\xi(t)dt\Big)=E(mes(A_b))=\int_T P(f(t)>b)dt.
  \end{equation*}
  Table 3 shows the simulation results.
  \end{example}

\begin{table}[htbp]\label{Tableex3}
  \centering
    \begin{tabular}{lcccccc}
    \hline
     &\multicolumn{2}{c}{$P(\sup_T f(t)>b)$}&&\multicolumn{3}{c}{$E(mes(A_b))$}\\\cline{2-3}\cline{5-7}
    b     & est   & std err && true value & est   & std err \\
  \hline
3     & 1.2E-02 & 5.6E-04 && 1.9E-03 & 1.8E-03 & 5.4E-05 \\
4     & 5.0E-04 & 1.9E-05 && 4.8E-05 & 5.0E-05 & 1.4E-06 \\
5     & 7.2E-06 & 2.8E-07 && 4.9E-07 & 5.1E-07 & 1.4E-08 \\
6     & 3.5E-08 & 1.4E-09 && 1.9E-09 & 1.9E-09 & 5.4E-11 \\
7     & 6.7E-11 & 2.7E-12 && 2.7E-12 & 2.6E-12 & 7.7E-14 \\
8     & 4.5E-14 & 1.9E-15 && 1.5E-15 & 1.5E-15 & 4.3E-17 \\
\hline
\end{tabular}%
    \caption{Simulation result for Example \ref{ex3}, where n=1000, m=40, $k(t)$ is the same as that of Example 2.}

\end{table}
\begin{example}\label{ex4}
Consider the continuous Gaussian field living on $T=[0,1]^2$ with mean and covariance function
\begin{equation*}
  \mu(t)=0.1t_1+0.1t_2\quad C(s,t)=e^{-|t-s|/4}.
\end{equation*}
Let $\xi=1$ and then the true expectation is
\begin{equation*}
  E\Big(\int_{A_b}\xi(t)dt\Big)=E(mes(A_b))=\int_TP(f(t)>b)dt.
\end{equation*}
Table 4 shows the simulation results.
\end{example}

\begin{table}[tph]
\centering
\begin{tabular}{lcccccc}
\hline
  &\multicolumn{2}{c}{$P(\sup_T f(t)>b)$}&&\multicolumn{3}{c}{$E(mes(A_b))$}\\\cline{2-3}\cline{5-7}
  b & est & std err & &true value & est & std err\\\hline
3     & 1.4E-02 & 6.6E-04 && 1.9E-03 & 1.9E-03 & 5.3E-05 \\
4     & 7.4E-04 & 4.4E-05 && 4.9E-05 & 5.1E-05 & 1.4E-06 \\
5     & 1.5E-05 & 7.5E-07 && 4.9E-07 & 5.1E-07 & 1.4E-08 \\
6     & 9.9E-08 & 5.2E-09 && 1.9E-09 & 1.9E-09 & 5.4E-11 \\
7     & 2.9E-10 & 1.3E-11 && 2.7E-12 & 2.7E-12 & 7.8E-14 \\
8     & 2.6E-13 & 1.4E-14 && 1.5E-15 & 1.5E-15 & 4.3E-17 \\\hline
\end{tabular}
\caption{Simulation result for Example \ref{ex4}, where n=1000, m=40, $k(t)=\frac{1}{8\pi}(1+|t|^2)^{-3}$, the density function of multivariate $t-$distribution, with degrees of freedom $4$, $\mu=0$, $\Sigma=4I_2$.}
\end{table}

For all the examples, the ratios of standard error over the estimated value do not increase as $b$ increase. This is consistent with our theoretical analysis. Also note that $m$  does not increase as the level increase, which reduces the computational complexity significantly. Overall, the numerical estimates are very accurate.

\section{Proof of Theorem \ref{ThmCont}}\label{SecCont}

Throughout the proof, we will use $\kappa$ as a generic notation to denote  large and not-so-important constants whose value may vary from place to place. Similarly, we use $\varepsilon_0$ as a generic notation for  small positive constants.

The first result we cite is  the  Borel-TIS (Borel-Tsirelson-Ibragimov-Sudakov) inequality \cite{AdlTay07,BOR,CIS} that will be used very often in our technical development.

\begin{proposition}\label{PropBorel}
Let $f(t)$ be a centered Gaussian process almost surely bounded in $T$. Then,
$$E[\sup_{t\in T}f(  t)  ]< \infty$$
and
\[
P\Big(\sup_{t\in T}f\left(  t\right)  -E[\sup_{t\in T}f(  t)  ]\geq b\Big)\leq\exp\left(  -b^{2}/(2\sigma_{T}^{2})\right)  .
\]
\end{proposition}

In this proof we  need to establish a lower bound of the probability
\begin{equation*}
  w(b)= E_b\Big[\frac{1}{mes(A_\gamma)};M>b\Big]\int_T P(f(t)>\gamma)dt
\end{equation*}
and an upper bound of the second moment
\begin{equation*}
  E_b(Z_b^2)=E_b\Big[\frac{1}{mes^2(A_\gamma)};M>b\Big]\Big[\int_T P(f(t)>\gamma)dt\Big]^2
\end{equation*}
The central analysis lies in the following two quantities:
\begin{equation}\label{I1}
I_1=E_b\Big[\frac{1}{mes^2(A_\gamma)};M>b\Big], \qquad I_2=E_b\Big[\frac{1}{mes(A_\gamma)};M>b\Big].
\end{equation}
We will show that there exist constants $\kappa$ and $ \varepsilon_0$ such that
\begin{equation}\label{moment}
  I_1\leq \kappa \zeta^{2d}, \qquad I_2\geq \varepsilon_0 \zeta^d.
\end{equation}
If these inequalities are proved, then
\begin{equation*}
  \limsup_{b\to\infty} \frac{I_1}{I_2^2}<\infty
\end{equation*}
is in place, and we finish our proof for Theorem \ref{ThmCont}. For the rest of the proof, we establish these two inequalities.
%In particular, for the unit variance case, we will establish that there is $c_1$ and $c_2$ such that
%\begin{equation}\label{moment}
%I_1\leq c_1\zeta_1^{2d}, \qquad I_2\geq c_2\zeta_1^d,
%\end{equation}
%where $\zeta_1$ is defined in \eqref{zeta1}.
%If \eqref{moment} is in place, then the theorem is proved. For the rest of the proof, we establish these two inequalities.
%We use Borel-TIS inequality (Proposition \ref{PropBorel}) to control $I_1$ and $I_2$, and to control expectations of supremum defined in Proposition \ref{PropBorel}, result of Dudley \cite{dudley2010sample} is  used.

To proceed, we describe the conditional Gaussian random field given $f(\tau)$.
First, if we write $f(\tau)=\gamma+{z}/{b}$, then $z$ asymptotically follows an exponential distribution with expectation $\sigma^2(\tau)$.
Conditional on $f(\tau)=\gamma+{z}/{b}$, let
\begin{equation}\label{decomp}
f(t+\tau)=E[f(t+\tau)|f(\tau)=\gamma+{z}/{b}]+f_0(t).
\end{equation}
Thus, given $f(\tau)$, $f_0(t)$ is a zero-mean Gaussian process.
By means of conditional Gaussian calculation, the conditional mean and conditional covariance function are given by
\begin{eqnarray}\label{Cond}
\mu_\tau(t)&=&E(f(t+\tau)|f(\tau)=\gamma+{z}/{b})\\
&&~~~=\mu(t+\tau)+\frac{\sigma(\tau+t)}{\sigma(\tau)}r(\tau+t,\tau)(\gamma+{z}/{b}-\mu(\tau))\notag\\
C_{0}(s,t)&=&Cov(f_0(s),f_0(t))\notag\\
&&~~~=\sigma(\tau+s)\sigma(\tau+t)[r(s+\tau,t+\tau)-r(\tau+t,\tau)r(\tau+s,\tau)].\notag
\end{eqnarray}
%\begin{eqnarray*}
%\mu_\tau(t)&=&E(f(t+\tau)|f(\tau)=\gamma+\frac{z}{b})=\mu(t+\tau)+\frac{C(\tau+t,\tau)}{C(\tau,\tau)}(\gamma+\frac{z}{b}-\mu(\tau))\\
%C_{0}(s,t)&=&Cov(f_0(s),f_0(t))=C(\tau+s,\tau+t)-C(\tau+s,\tau)C(\tau+t,\tau)C(\tau,\tau)^{-1}.\notag
%\end{eqnarray*}
%Since the process has unit variance, we can replace $C(s,t)$ by $r(s,t)$ and obtain that
%\begin{eqnarray}\label{ConD}
%\mu_\tau(t)&=&E(f(t+\tau)|f(\tau)=\gamma+\frac{z}{b})=\mu(t+\tau)+\frac{r(\tau+t,\tau)}{r(\tau,\tau)}(\gamma+\frac{z}{b}-\mu(\tau))\\
%C_{0}(s,t)&=&Cov(f_0(s),f_0(t))=r(\tau+s,\tau+t)-r(\tau+s,\tau)r(\tau+t,\tau).\notag
%\end{eqnarray}
The next lemma  controls the conditional variance.
\begin{lemma}\label{LemmaMu}
Under condition A1-6,  there exists constants $\lambda_1>0$, such that for all $\tau\in T$, and $b$ large enough,
\begin{enumerate}
 % \item[(i)] For $t$ sufficiently small,
 % \begin{equation*}
  %  |\mu_\tau(t)-(\gamma+\frac{z}{b})|\leq \lambda_0 \sqrt{L(|t|)}|t|^{\alpha/2}+\lambda_1L(|t|)|t|^\alpha(\gamma+\frac{z}{b});
 % \end{equation*}
 % for $z\in(0,1)$,
  %\jc{
%  \begin{equation*}
%    |\mu_\tau(t/{\zeta})-\mu_\tau(s/{\zeta})|\leq \lambda_0 \zeta^{-\alpha/2}\sqrt{L(\zeta^{-1}|t-s|)}|t-s|^{\alpha/2}+\lambda_1 b \max(|t|,|s|)^{\beta_0}|t-s|^{\beta_1}L(\zeta^{-1}|t-s|)\zeta^{-\alpha}
%  \end{equation*}}
%\begin{equation*}
%    |\mu_\tau(t/{\zeta})-\mu_\tau(s/{\zeta})|\leq \lambda_0 \zeta^{-\alpha/2}|t-s|^{\alpha/2}+\lambda_1 b^{-1} |t-s|^{\alpha}
%  \end{equation*}
  \item[(i)] for all $t+\tau \in T$,
  \begin{equation*}
    C_0(t,t)\leq \lambda_1 L_1(|t|)|t|^{\alpha_1};
  \end{equation*}
  \item[(ii)] for $s,t\in T$
  \begin{equation*}
    Var(f_0(s)-f_0(t))\leq \lambda_1\max(L_1(|t-s|)|t-s|^{\alpha_1}, L_2(|t-s|)|t-s|^{\alpha_2});
  \end{equation*}
  \item[(iii)] for  any $\varepsilon>0$, there exists $\delta >0$ (independent of $b$) such that for each $t$
  \begin{equation*}
E(\sup_{|s-t|\leq \delta \zeta^{-1}} f_0(s))=\frac \varepsilon b.
  \end{equation*}
  %\item[(iv)] For $t$ small enough, there exists $\lambda_2>0$ such that
  %\begin{equation*}
  %  \mu_\tau(t)\leq \gamma+\frac{z}{\gamma}+\kappa_\mu\sqrt{L(|t|)}|t|^{\alpha/2}-\lambda_2 b L(|t|)|t|^\alpha
  %\end{equation*}
\end{enumerate}
\end{lemma}
The proofs for (i) and (ii) are an application of conditions A2, A3, A6, and elementary calculations. (iii) is a direct corollary of (ii) and  Dudley's entropy bound (Theorem 1.1 of \cite{dudley2010sample}). We omit the detailed derivations.
We proceed to the analysis of $I_1$ and $I_2$ by considering the Type 1 and Type 2 standard deviations function (condition A4) separately.

%\subsection{Proof of the theorem when $\sigma(t)$ is of Type 1 in Assumption A4}\label{SecT1}
\emph{In the main text, we only provide the proof when $\sigma(t)$ is of Type 1 in Assumption A4, that is, a constant variance. The proof of the non-constant case is similar. We present it in the Supplemental Material.}
For the constant variance case that corresponds to $\alpha_2=\infty$, the scaling factor is given by
$$\zeta=\zeta_1.$$
We aim at showing that
$I_1\leq \kappa\zeta_1^{2d}$ and  $ I_2\geq \varepsilon_0\zeta_1^d.$

%We use Borel-TIS inequality (Proposition \ref{PropBorel}) to control $I_1$ and $I_2$, and to control expectations of supremum defined in Proposition \ref{PropBorel}, result of Dudley \cite{dudley2010sample} is  used.

\subsection{The $I_1$ term}

For some $y_0>0$ chosen to be sufficiently small (independent of $b$) and to be determined in the later analysis,
 the $I_1$ term is bounded by
\begin{equation}\label{2momentbound}
E_b\Big[\frac{1}{mes^2(A_\gamma)};M>b\Big] \leq
%\omega_d^{-2}
 y_0^{-2d}\zeta_1^{2d}+E_b\Big(\frac{1}{mes^2(A_\gamma) };mes(A_\gamma)<y_0^d\zeta_1^d, M>b\Big),
\end{equation}
To control the second term of the above inequality, we need to provide a bound on the following tail probability for $0<y< y_0$
\begin{eqnarray}\label{expQ}
  &&Q(mes(A_\gamma)<y^{d}\zeta_1^{-d}, M>b)\notag\\
  &=&\int P(mes(A_\gamma)< y^{d}\zeta_1^{-d}, M>b  |f(\tau)=\gamma+{z}/{b})h_b(\tau)
  \frac{q_{b,\tau}(\gamma+{z}/{b})}{b}d\tau dz.
\end{eqnarray}
The probability inside the integral is with respect to the original measure $P$ because, conditional on $f(\tau)$, $f(t)$ follows the original conditional distribution.
We  develop bounds for $P(mes(A_\gamma)< y^{d}\zeta_1^{-d}, M>b  |f(\tau)=\gamma+{z}/{b})$ under two situations: $z>1$ and  $0<z\leq 1$.

\subsubsection*{Situation 1: $z> 1$.}

There exists some constant $c_d > 0$ only depending on the dimension $d$ such that
the event $\{mes(A_\gamma)< y^d\zeta_1^{-d}\}$ implies the event $\{\inf_{|t-\tau|\leq c_d y\zeta_1^{-1}}f(t)\leq\gamma\}$.
Thus, we have the bound
\begin{eqnarray}
  P\Big(mes(A_\gamma)\leq y^d\zeta_1^{-d}, M>b|f(\tau)=\gamma+\frac{z}{b}\Big)&\leq& P\Big(\inf_{|t-\tau|\leq c_d y\zeta_1^{-1}}f(t)\leq\gamma|f(\tau)=\gamma+\frac{z}{b}\Big)\nonumber
\end{eqnarray}
Using the representation in \eqref{decomp}, the right-hand-side of the above probability is given by
\begin{eqnarray}
 &=& P\Big(\inf_{|t|\leq c_d y\zeta_1^{-1}}f_0(t)+\mu_\tau(t)\leq \gamma\Big)\label{eqI11}
\end{eqnarray}
For $y < y_0$, according to  Condition A2 and properties of slowly varying function, the representation \eqref{Cond} yields that
\begin{equation}\label{eqI12}
\mu_\tau(t)\geq \gamma+  \frac{1}{2b} \qquad \mbox{for $|t|\leq c_d y\zeta_1^{-1}$.}
\end{equation}
To obtain the above bound, notice that $\mu_\tau (0) = \gamma + z/b> \gamma + 1/b$. In addition, for the constant variance case, expression \eqref{Cond} can be written as
\begin{equation}\label{CondType1}
  \mu_\tau(t)=\mu(t+\tau)+r(\tau+t,\tau)(\gamma+{z}/{b}-\mu(\tau)).
\end{equation}
According to the continuity condition A5, we have that $|\mu_\tau(t) - \mu_\tau (0)| = O(b L_1(t)|t|^{\alpha_1}) +O(\sqrt{L_t(t) |t|^{\alpha_1}})$. According to the definition of $\zeta_1$ as in \eqref{zeta1}, \eqref{eqI12} can be achieved by choosing $y_0$ small.
Furthermore, by Lemma \ref{LemmaMu}(i) the conditional variance is
\begin{eqnarray*}
  C_0(t,t)\leq \lambda_1 L_1(c_d y \zeta_1^{-1})c_d^{\alpha_1} y^\alpha_1 \zeta_1^{-\alpha_1}.
\end{eqnarray*}
Using the slowly varying property of $L_1(x)$ and the fact that $L_1(\zeta_1^{-1})\zeta_1^{-\alpha_1} = b^{-2}$, we have that
$$L_1(c_d y \zeta_1^{-1}) y^\alpha_1 \zeta_1^{-\alpha_1} =b^{-2} \frac{L_1(c_d y \zeta_1^{-1})}{L_1(\zeta_1^{-1})} y^{\alpha_1} = O(y ^{\alpha_1/2}b^{-2}).$$
For the last step of the above estimate, we use Lemma \ref{LemmaSlow}(i) on page \pageref{LemmaSlow}  that the ratio $L_1(c_dy \zeta^{-1}_1)/L_1(\zeta^{-1}_1)$ varies slower than any polynomial of $y$
%(c.f.~the representation of slowly varying function in \cite{galambos1973regularly}). Thus, the variance of $f_0(t)$ is bounded by
\begin{eqnarray}
  C_0(t,t) =O(y^{\alpha_1/2}b^{-2})\label{eqI13}.
\end{eqnarray}
By Lemma \ref{LemmaMu}(iii), $E(\sup_{|t|\leq c_dy_0\zeta_1^{-1}}b\times f_0(t))=o(1)$ as $y_0\to0.$ So we can pick $y_0$ small enough such that% for all $y < y_0$
\begin{equation}\label{y0}
  E(\sup_{|t|\leq c_d y_0\zeta_1^{-1}}f_0(t))\leq \frac{1}{4b}
\end{equation}
By the Borel-TIS inequality (Proposition \ref{PropBorel}), \eqref{eqI11}, \eqref{eqI12}, \eqref{eqI13}, and \eqref{y0}, there exists a positive constant $\varepsilon_0$, such that
\begin{equation*}
P(mes(A_\gamma)\leq y^d\zeta_1^{-d}, M>b|f(\tau)=\gamma+{z}/{b})\leq
P(\inf_{|t|\leq c_d y\zeta_1^{-1}} |f_0(t)| > \frac 1 {2b})
\leq \exp(-\varepsilon_0y^{-\alpha_1/2}).
\end{equation*}
%\end{proof}

\subsubsection*{Situation 2: $0<z\leq 1$.} We now proceed to the case where $0<z\leq 1$.
With $y_0$ defined to satisfy \eqref{eqI12} and \eqref{y0}, we let $c= c_d y_0$ and
%  $$E(\sup_{|t-t'|\leq 2c\zeta_1^{-1}}f_0(t))\leq \frac{1}{4b}.$$
define a finite  subset $\tilde{T}=\{t_1,...,t_N\}\subset T$  such that
    \begin{enumerate}
    \item For $ i\neq j,$ $i,j\in\{1,...,N\}$, $|t_i-t_j|\geq \frac{c}{2\zeta_1}$.
    \item For any $t\in T$, there exists $i\in\{1,...,N\}$, such that $|t-t_i|\leq \frac c{\zeta_1}$.
  \end{enumerate}
  Furthermore, let\label{latticeT}
  $$B_i=\{t\in T :|t-t_i|\leq c\zeta_1^{-1}\}\qquad \mbox{for $i\in\{1,2,...,N\}$. }$$
Thus, $\{B_i: i =1,...,N\}$ covers $T$, that is, $\cup_i B_i =T$.
  Note that
  \begin{equation*}
    P\left(mes(A_\gamma)\leq y^{d}\zeta_1^{-d}, M>b|f(\tau)=\gamma+\frac{z}{b}\right)\leq \sum_{i=1}^{N}P\Big(mes(A_\gamma)\leq y^{d}\zeta_1^{-d},\sup_{t\in B_i} f(t)>b|f(\tau)=\gamma+\frac{z}{b}\Big).
  \end{equation*}
	With $c_d$ as previously chosen,  each of the summands in the above display is bounded by
% \begin{equation*}P\Big(\frac{1}{mes(A_\gamma)}>y^{-d}\zeta_1^{d},\sup_{t\in A_i}f(t)>b|f(\tau)=\gamma+\frac{z}{b}\Big).\end{equation*}
%  We have:
  \begin{eqnarray}\label{eqI16}
  &&P\left(mes(A_\gamma)\leq y^{d}\zeta_1^{-d},\sup_{t\in B_i}f(t)>b|f(\tau)=\gamma+\frac{z}{b}\right)\nonumber\\
  &\leq& P\Big(\sup_{t\in B_i,|s-t|\leq c_d y\zeta_1^{-1}}|f(t)-f(s)|>\frac{1}{b},\sup_{t\in B_i}f(t)>b|f(\tau)=\gamma+\frac{z}{b}\Big).
  \end{eqnarray}
  The above inequality is derived from the following argument. The process exceeds the level $b$ at some point in $B_i$. However, the volume of the excursion set over the level $\gamma = b - 1/b$ has  to be less than $y^d/\zeta_1^d$. This suggests that $f(t)$ must have a fast drop from the level $b$ to $b- 1/b$. Therefore, %if we choose $c_d$ to be some small number,
  the event $\{mes(A_\gamma)>y^{d}\zeta_1^{-d},\sup_{t\in B_i}f(t)>b\}$ is a subset of $\{\sup_{t\in B_i,|s-t|\leq c_d y\zeta_1^{-1}}|f(t)-f(s)|>\frac{1}{b}, \sup_{t\in B_i}f(t)>b\}$.

For the case that $0< z<1$, we select $\delta_0,\delta_1>0$ small enough, and $\lambda$ large enough and provide a bound for \eqref{eqI16} under the following four cases:

\begin{itemize}
\item[]Case 1. $0<|t_i -\tau|< y^{-\delta_0}\zeta_1^{-1}$;
\item[]Case 2. $y^{-\delta_0}\zeta_1^{-1}<|t_i -\tau|<\delta_1$;
\item[]Case 3. $|t_i -\tau|\geq\delta_1$,  $y<b^{-\lambda}$;%, and $|t_i -\tau|>y^{-\delta_0}\zeta_1^{-1}$;
\item[]Case 4. $|t_i -\tau|\geq \delta_1$, $y\geq b^{-\lambda}$.
\end{itemize}
%For Case 3, we choose $\delta_0 \times \lambda$ to be small enough such that $b^{-\lambda\delta_0}\zeta_1^{-1}< \delta_1$ and thus $|t_i -\tau|\geq\delta_1$ implies $|t_i -\tau|>y^{-\delta_0}\zeta_1^{-1}$.
To facilitate the discussion, define
$$x_i \triangleq  \zeta_1\times |t_i -\tau|.$$

%\jc{  From Condition A5, and Lemma \ref{LemmaMu} (v), we can choose positive number $\delta_1, \varepsilon_1$, such that for $|t|<\delta_1,$ $$\mu_\tau(t)\leq b+\lambda_0\sqrt{L_1(|t|)}|t|^{\alpha_1/2}-\lambda_2L_1(|t|)|t|^\alpha_1;$$ and for $|t|>\delta_1/2$, $\mu_\tau(t)\leq (1-\varepsilon_1)b$.
%  Let $|t_i-\tau|=x\zeta_1^{-1}$, $\delta_0=\frac{\beta_1}{4\beta_0}$, and $\lambda=3/\beta_1.$ For some $y < y_0$, we discuss the following cases separately.
%}
%
  \paragraph{ Case 1: $0<|t_i -\tau|< y^{-\delta_0}\zeta_1^{-1}$.}

We provide a bound for \eqref{eqI16} via the conditional representation \eqref{decomp} and the calculation in \eqref{Cond}. According to conditions A2 and A5,  for $|t-s|\leq c_d y\zeta_1^{-1}$ and $t\in B_i$, we have
  \begin{eqnarray*}
    |\mu_\tau(t)-\mu_\tau(s)|&\leq&
    \kappa_\mu \zeta_1^{-\alpha_1/2}\sqrt{L_1(y/\zeta_1)}y ^{\alpha_1/2}+\kappa_r(x_i+1)^{\beta_0} L_1((x_i+1)\zeta_1^{-1})y^{\beta_1}\zeta_1^{-\alpha_1}b
  \end{eqnarray*}
According the definition of $\zeta_1$ in \eqref{zeta1}, and Lemma \ref{LemmaSlow}(i) the above display can be bounded by
    $$|\mu_\tau(t)-\mu_\tau(s)|\leq \frac{ 2\kappa_\mu y^{\alpha_1/4}+2 \kappa_r y^{-\delta_0\beta_0+\beta_1-\varepsilon_0}}{b}.$$
%    &=&o(1)b^{-1}
We choose $\delta_0$ small such that it is further bounded by
$$|\mu_\tau(t)-\mu_\tau(s)|\leq  \kappa y ^{\varepsilon_0} b^{-1}\qquad \mbox{for some possibly different $\varepsilon_0 >0$.}$$
Furthermore, we pick $y_0>0$ small enough such that for $0<y<y_0$ and $|s-t|<c_d y\zeta_1^{-1}$
\begin{equation}\label{mu1}
  |\mu_\tau(s)-\mu_\tau(t)|\leq \frac{1}{2b}.
\end{equation}
  The above inequality provides a bound on the variation of the mean function over the set $B_i$ when $t_i$ is within $y^{-\delta_0}\zeta_1^{-1}$ distance close to $\tau$.
  The probability in \eqref{eqI16} can be bounded by
  $$\eqref{eqI16} \leq P(\sup_{t\in B_i, |t-s|\leq c_d y\zeta_1^{-1}}|f_0(t)-f_0(s)|>\frac{1}{2b}).$$
  Note that by Lemma \ref{LemmaMu}(ii), for $|s-t|<c_d y\zeta_1^{-1}$, we have that
  \begin{eqnarray}
    Var(f_0(s)-f_0(t))&\leq& \lambda_1 \frac{L_1(c_dy\zeta_1^{-1})}{L_1(\zeta_1^{-1})}y^{\alpha_1} b^{-2}=O( y^{\alpha_1/2}b^{-2}) \qquad \mbox{  for $y < y_0$.}\label{condvar}
  \end{eqnarray}
  We apply the Borel-TIS inequality (Proposition \ref{PropBorel}) to the double-indexed Gaussian field $\xi(s,t) \triangleq f_0(s)-f_0(t)$ and   obtain that there exists a positive constant $\varepsilon_0$ such that
  \begin{eqnarray}
&&     P\left(\frac{1}{mes(A_\gamma)}>y^{-d}\zeta_1^{d},\sup_{t\in B_i}f(t)>b|f(\tau)=\gamma+\frac{z}{b}\right)\nonumber\\
    &\leq & P(\sup_{t\in B_i, |t-s|\leq c_d y\zeta_1^{-1}}|f_0(t)-f_0(s)|>\frac{1}{2b})\notag\\&\leq& \exp(-\varepsilon_0y^{-\alpha_1/2})\label{tailboundzg1}
  \end{eqnarray}
  We put together all the $B_i$'s such that $|t_i-\tau| < y^{-\delta_0}\zeta_1^{-1}$ and obtain that
  \begin{eqnarray*}
    &&P\Big(\frac{1}{mes(A_\gamma)}>y^{-d}\zeta_1^{d},\sup_{|t-\tau|\leq y^{-\delta_0}\zeta_1^{-1}}f(t)>b|f(\tau)=\gamma+\frac{z}{b}\Big)\\
    &=& O(y^{-\delta_0 d}\exp(-\varepsilon_0 y^{-\alpha_1/2})) \leq\exp(-y^{-\varepsilon_0})
  \end{eqnarray*}
possibly redefining $\varepsilon_0$.

  \paragraph{Case 2: $y^{-\delta_0}\zeta_1^{-1}<|t_i - \tau|<\delta_1.$}
For this case, we implicitly require that $y^{-\delta_0}\zeta_1^{-1}<\delta_1$.
  For $t\in B_i$ and $y$ small enough, we have that
  \begin{eqnarray*}
    P(\sup_{t\in B_i,|s-t|\leq c_d y\zeta_1^{-1}}|f(t)-f(s)|>\frac{1}{b},\sup_{t\in B_i}f(t)>b|f(\tau)=\gamma+\frac{z}{b})&\leq &P(\sup_{t\in B_i}f(t)>b|f(\tau)=\gamma+\frac{z}{b})
   \end{eqnarray*}
   According to condition A2, expression \eqref{CondType1}, and property of slowly varying functions, we have the bound for $\tau + t\in B_i$
  \begin{eqnarray}\label{mu2}
    \mu_\tau(t)
    &\leq& b-\frac{\Delta_\tau}{2} \frac{L_1(x_i\zeta_1^{-1})}{L_1(\zeta_1^{-1})}x_i^{\alpha_1} b^{-1}.
  \end{eqnarray}
  From Lemma \ref{LemmaMu} and definition of $\zeta_1$, the variance of $f_0(t)$ is controlled by
  \begin{equation}\label{eqvc2}
  C_0(t,t)\leq 2\lambda_1 \frac{L_1(x_i\zeta_1^{-1})}{L_1(\zeta_1^{-1})}x_i^{\alpha_1} b^{-2}.
  \end{equation}
  According to Proposition 1 and Lemma \ref{LemmaSlow}(ii) that $\frac{L_1(x_i\zeta_1^{-1})}{L_1(\zeta_1^{-1})}x_i^{\alpha_1}>x_i^{\alpha_1/2}$ for $y^{-\delta_0}<x_i<\delta_1\zeta_1.$, we continue the calculations
   \begin{eqnarray*}
   P(\sup_{t\in B_i}f(t)>b|f(\tau)=\gamma+\frac{z}{b})
    &\leq&P(\sup_{t+\tau\in B_i}f_0(t)>\frac{\Delta_\tau}{2} \frac{L_1(x_i\zeta_1^{-1})}{L_1(\zeta_1^{-1})}x_i^{\alpha_1} b^{-1})\\
    &\leq& \exp\big(-\frac{\Delta_\tau^2}{8\lambda_1}\frac{L_1(x_i\zeta_1^{-1})}{L_1(\zeta_1^{-1})}x_i^{\alpha_1}\big)\\
    &\leq& \exp(-\frac{\Delta_\tau^2}{8\lambda_1}x_i^{\alpha_1/2}).
  \end{eqnarray*}
  Putting together all the $B_i$'s such that $y^{-\delta_0}<x_i<\delta_1\zeta_1$, we have that
  \begin{eqnarray*}
    &&P(\frac{1}{mes(A_\gamma)}>y^{-d}\zeta_1^{d},\sup_{y^{-\delta_0}\zeta_1^{-1}<|t-\tau|<\delta_1}f(t)>b|f(\tau)=\gamma+\frac{z}{b})\\
    &\leq&  \sum_{k=0}^{\infty}\kappa (y^{-\delta_0} + k )^{d-1}\exp[-\frac{\Delta_\tau^2}{8\lambda_1}(y^{-\delta_0} + k )^{\alpha_1/2}]\\
    &\leq& \exp(-y^{-\varepsilon_0})
  \end{eqnarray*}
  for some constant $\varepsilon_0>0$.

  \paragraph{Case 3: $|t_i -\tau|\geq\delta_1$ and $y<b^{-\lambda}$.}
%, and $|t_i-\tau|>y^{-\delta_0}\zeta_1^{-1}$.}

 Since $C(s,t)$ is uniformly H\"older continuous, we can always choose $\lambda$ large such that   for $|s-t|\leq c_d y\zeta_1^{-1}\leq  c_d b^{-\lambda}\zeta_1^{-1}$,
  \begin{eqnarray}\label{mu3}
    |\mu_\tau(t)-\mu_\tau(s)|&\leq& \frac{1}{4b}.
  \end{eqnarray}
  By Lemma \ref{LemmaMu}(ii) and Lemma \ref{LemmaSlow}(i),  for $|s-t|\leq c_d y\zeta_1^{-1}$,  the conditional variance $Var(f_0(s)-f_0(t))$ is bounded by
  \begin{eqnarray*}
  Var(f_0(s)-f_0(t))&\leq &\lambda_1 \frac{L_1(c_dy\zeta_1^{-1})}{L_1(\zeta_1^{-1})}y^{\alpha_1} b^{-2}=O(y^{\alpha_1/2}b^{-2}).
  \end{eqnarray*}
  Thus, there exist a constant $\varepsilon_0>0$ such that
  \begin{eqnarray*}
    &&P\Big(\sup_{t\in B_i,|s-t|\leq c_d y\zeta_1^{-1}}|f(t)-f(s)|>\frac{1}{b},\sup_{t\in B_i}f(t)>b|f(\tau)=\gamma+\frac{z}{b}\Big)\\
    &\leq& P\Big(\sup_{ t\in B_i,|s-t|\leq c_d y\zeta_1^{-1}}|f_0(t)-f_0(s)|>\frac{1}{2b}\Big)\\
    &\leq& 2\exp(-\varepsilon_0 y^{-\alpha_1})
  \end{eqnarray*}
  Note that $\zeta_1\ll b^{4/\alpha_1}$, so for $y<b^{-\lambda}$, we have
  \begin{eqnarray}
    &&P\Big(\frac{1}{mes(A_\gamma)}>y^{-d}\zeta_1^{d},\sup_{|t-\tau|>\delta_1}f(t)>b|f(\tau)=\gamma+\frac{z}{b}\Big)\notag\\
    &\leq&O(\zeta_1^d )\sup_{i}P\Big(\sup_{t\in B_i,|s-t|\leq c_d y\zeta_1^{-1}}|f(t)-f(s)|>\frac{1}{b},\sup_{t\in B_i}f(t)>b|f(\tau)=\gamma+\frac{z}{b}\Big)\nonumber\\
    &\leq&O(b^{4d/\alpha_1} )\exp(-\varepsilon_0 y^{-\alpha_1/2})\nonumber\\
    &\leq&O(y^{-\frac{4d}{\alpha_1\lambda}})\exp(-\varepsilon_0 y^{-\alpha_1/2})\nonumber\\
    &\leq&\exp(-y^{-\varepsilon_0})\label{tailthm1c2}
  \end{eqnarray}
  for some possibly different constant $\varepsilon_0$.

  \paragraph{Case 4: $|t_i - \tau|\geq \delta_1$ and $y\geq b^{-\lambda}$.}\label{case4}
  Note that condition A3 implies that for any $\delta_1>0$, there exists $\varepsilon>0$ such that for $|s-t|>\delta_1$, $r(s,t)<1-\varepsilon$, and thus according to expression \eqref{CondType1}, there exists $\varepsilon>0$ such that $\mu_\tau(t)\leq(1-\varepsilon)b$. According to  Proposition \ref{PropBorel} (the Borel-TIS inequality), we have that for $b$ large enough,
  \begin{eqnarray*}
    &&P(\frac{1}{mes(A_\gamma)}>y^{-d}\zeta_1^{d},\sup_{|t-\tau|\geq\delta_1}f(t)>b|f(\tau)=\gamma+\frac{z}{b})\\
    &\leq& P(\sup_{|t|\geq\delta_1}f_0(t)+\mu_\tau(t)>b)\\
    &\leq& P(\sup_{|t|\geq\delta_1}f_0(t)>\varepsilon b)\\
    &\leq& \exp(-\frac{\varepsilon^2 b^2}{2\sigma_T^2})\\
    &\leq& \exp(-y^{-\varepsilon_0})
  \end{eqnarray*}
  for some constant $\varepsilon_0>0.$

  Combining Cases 1-4,   for some constants $\varepsilon_0$ and $y_0$ chosen to be small, we have that for $y \in (0,y_0]$% and $y_0$ sufficiently small
  \begin{equation}\label{tailbound}
    P\Big(\frac{1}{mes(A_\gamma)}>y^{-d}\zeta_1^d,M>b\Big| f(\tau)=\gamma+\frac{z}{b}\Big)
    \leq  \exp(-y^{-\varepsilon_0})
  \end{equation}
 Together with  \eqref{expQ}, we have
  \begin{equation}\label{mesub}
    Q\Big(\frac{1}{mes(A_\gamma)}>y^{-d}\zeta_1^d,M>b\Big)\leq \exp(-y^{-\varepsilon_0}).
  \end{equation}
Thus, according to \eqref{2momentbound}, for some $\kappa>0$, we have
\begin{eqnarray}\label{I1}
  E^Q\Big[\frac{1}{mes(A_\gamma)^2};M>b\Big]  \leq(\kappa + y_0^{-2d})\zeta_1^{2d}.
\end{eqnarray}

\subsection{The $I_2$ term}

To provide a lower bound of  $$I_2=E^Q\Big[\frac{1}{mes(A_\gamma)};M>b\Big],$$ we basically need to prove that $mes(A_\gamma)$ cannot be always very large. Thus, it is sufficient to show that $f(t)$ drop below $\gamma$ when $t$ is reasonably far away from $\tau$.
The next lemma shows that for any $\delta >0$, the process $f(t)$ drops below $\gamma$ almost all the time when $|t-\tau| > \delta$.

\begin{lemma}\label{LemmaI21}
Under conditions A1-6, for standard deviation of Type 1,  we have that
  \begin{equation}\label{eqlemma1}
    Q(\sup_{|t-\tau|>\delta}f(t)\geq\gamma)\leq e^{-\varepsilon_0 b^2}\qquad \mbox{  for some $\varepsilon_0>0.$}
  \end{equation}
\end{lemma}

%\bigskip
%
%The next lemma further provides a bound of the excursion set.
\begin{lemma}\label{LemmaI22}
  Under conditions A1-6, there exists $\delta$ small and $\kappa$ large  (independent of $b$), such that for $x> \kappa$ we have
  \begin{equation}\label{eqlemma2}
    Q\Big(\sup_{x\zeta^{-1}\leq |t-\tau|\leq \delta}f(t)\geq\gamma\Big)< e^{-\varepsilon_0 x^{\alpha_1/4}}.
  \end{equation}
\end{lemma}

The proof of these two Lemmas are provided in the Supplemental Material.
We proceed to developing a lower bound for $I_2$. First, notice that the event $\{M>b\}$ is a regular event under $Q$, that is,
\begin{equation*}
  Q(M>b)\geq Q(f(\tau)>b)>\frac{1}{2}e^{-1/\sigma^2(\tau)}.
\end{equation*}
Also, according to Lemma \ref{LemmaI21} and \ref{LemmaI22}, we choose $x$ such that
\begin{equation*}
  Q(\sup_{ |t-\tau|>x\zeta_1^{-1}}f(t)\geq\gamma)<\frac{1}{2}e^{-2/\sigma^2(\tau)}
\end{equation*}
Let $\omega_d$ be the volume of the $d$-dimensional unit ball.
Thus, we have
\begin{eqnarray}
  I_2
  &\geq& E^Q(\frac{1}{mes(A_\gamma)}; M>b, mes(A_\gamma)< \omega_d x^{d}\zeta_1^{-d})\notag\\
  &\geq& \omega_d^{-1}x^{-d}\zeta_1^{d} Q(mes(A_\gamma)< \omega_d x^{d}\zeta_1^{-d}, M>b)\notag\\
  &\geq&\omega_d^{-1}x^{-d}\zeta_1^d\big[Q(M>b)-Q(mes(A_\gamma)\geq \omega_d x^{d}\zeta_1^{-d})\big]\notag\\
  &\geq&\omega_d^{-1}x^{-d}\zeta_1^d\big[Q(M>b)-Q(\sup_{ |t-\tau|>x\zeta_1^{-1}}f(t)\geq\gamma)\big]\notag\\
  &\geq&\omega_d^{-1}x^{-d}\zeta_1^d(e^{-1/\sigma^2(\tau)}-e^{-2/\sigma^2(\tau)})\label{e:39}
\end{eqnarray}
Summarizing the results in \eqref{I1} and \eqref{e:39}, we have that
$$E_b (Z_b ^2)  \leq \kappa \zeta_1^{2d} \Big (\int P(f(t) > \gamma)dt \Big)^2,\quad P(M>b ) > \varepsilon_0 \zeta_1 ^{d}\int P(f(t) > \gamma)dt $$
and therefore
$$\sup_{b}\frac{E^Q Z_b ^2 }{ P^2(M>b)}< \infty.$$

\section{Proof of Theorem \ref{ThmDis}}\label{SecDis}

Let $T_m = \{t_1,...,t_m\}$ be generated in the step 3 of Algorithm 2.
We start the analysis with the following decomposition%
\begin{eqnarray*}
\hat{Z}_{b}-Z_{b} &=&\left[ \frac{I(\sup f(t)>b)}{mes(A_{\gamma })}-\frac{%
I(\max_{i=1}^{m}f(t_{i})>b)}{\widehat{mes}(A_{\gamma })}\right]
E(mes(A_{\gamma })) \\
&=&E(mes(A_{\gamma })) \\
&&~~\times \Big[ \frac{I(\sup f(t)>b)}{mes(A_{\gamma })}-\frac{%
I(\max_{i=1}^{m}f(t_{i})>b)}{mes(A_{\gamma })} \\
&&~~~~~~~~+\frac{I(\max_{i=1}^{m}f(t_{i})>b)}{mes(A_{\gamma })}-\frac{%
I(\max_{i=1}^{m}f(t_{i})>b)}{\widehat{mes}(A_{\gamma })}\Big],
\end{eqnarray*}
where $\widehat{mes} (A_\gamma)$ is defined as in \eqref{vol}.
According to the result in Theorem \ref{ThmCont}, we only need to show that $|E^Q(\hat Z_b - Z_b)|\leq \varepsilon P(M>b)$ and $Var (\hat Z_b - Z_b) = O(P^2(M>b))$.
We define notations%
\begin{eqnarray*}
J_{1} &=&\frac{I(\sup f(t)>b)}{mes(A_{\gamma })}-\frac{I(%
\max_{i=1}^{m}f(t_{i})>b)}{mes(A_{\gamma })} \\
J_{2} &=&\frac{I(\max_{i=1}^{m}f(t_{i})>b)}{mes(A_{\gamma })}-\frac{%
I(\max_{i=1}^{m}f(t_{i})>b)}{\widehat{mes}(A_{\gamma })}.
\end{eqnarray*}%
We control each of the two terms respectively.

\subsection{The $J_{1}$ term}

%\jc{change $A$ to some other letter.}

Note that $J_1$ is non-negative and
\begin{equation}
E_b(J_{1})=E_b\left( \frac{1}{mes(A_{\gamma })};M>b;\max_{i=1}^{m}f(t_{i})\leq b\right) .  \label{J1}
\end{equation}%
Note that the proof of Theorem \ref{ThmCont}, in particular \eqref{mesub} and \eqref{mesub1}, shows that
$
\frac{I(M>b)}{\zeta ^{d}mes(A_{\gamma })}
$
is uniformly integrable in the parameter $b$ where
$$\zeta = \max(\zeta_1,\zeta_2).$$
More precisely, for any $\delta $ small enough, we have that
\begin{equation}\label{UNI}
\sup_{Q(B)\leq \delta} E_b\left( \frac{1}{mes(A_{\gamma })};M>b;B\right)\leq (-\log \delta)^{1/\varepsilon_0} \delta \zeta^{d}.
\end{equation}
Therefore, it is sufficient to
focus on and derive a bound for the probaiblity
\begin{equation*}
Q(M>b;\max_{i=1}^{m}f(t_{i})\leq b).
\end{equation*}%
Let $x$ be large and $\delta'$ be small such that%
\begin{eqnarray}\label{UI}
&&Q\Big( M>b;\max_{i=1}^{m}f(t_{i})\leq b\Big)  \\
&\leq&  Q\Big( \sup_{x\zeta^{-1}<|t-\tau |<\delta'
}f(t)>b;\max_{i=1}^{m}f(t_{i})\leq b\Big) \notag\\
&+&Q\Big( \sup_{|t-\tau
|<x\zeta ^{-1}}f(t)>b,\sup_{|t-\tau|>x\zeta^{-1}}f(t)\leq b;\max_{i=1}^{m}f(t_{i})\leq b\Big)\notag\\
&+&Q\Big( \sup_{|t-\tau |\geq \delta'}f(t)>b;\max_{i=1}^{m}f(t_{i})\leq b\Big) .\notag
\end{eqnarray}%
We will provide a specific  choice of $m$ such that
$$Q\left( \sup f(t)>b;\max_{i=1}^{m}f(t_{i})\leq b\right)\leq \delta \triangleq  \varepsilon^{1+ \varepsilon_0},$$
where $\varepsilon$ is relative bias preset in the statement of the theorem.
We consider each of the three terms in \eqref{UI}.% and start with the first term.

\subsubsection{The first term in \eqref{UI}.}

We  choose
$$x = \min\{(-\log \delta)^{4/\alpha},\delta' \zeta\},\qquad \mbox{where $\alpha = \min \{\alpha_1,\alpha_2\}$.}$$
%where $\alpha$ is defined as the minimum of $\alpha_1$ and $\alpha_2$.
According to Lemma \ref{LemmaI22} and \eqref{L3}, the first term in \eqref{UI} is bounded by
\begin{equation*}
Q\Big( \sup_{x\zeta^{-1}<|t-\tau |<\delta'}f(t)>b;\max_{i=1}^{m}f(t_{i})\leq
b\Big) \leq Q\Big(\sup_{x\zeta^{-1}<|t-\tau |<\delta'}f(t)>b\Big)\leq \delta .
\end{equation*}
Notationally, we define that $\sup_{t\in \emptyset} f(t) =  - \infty$. Thus, when $x = \delta' \zeta$, the above probability is zero.

\subsubsection{The second term in \eqref{UI}.}
%Furthermore, we consider the second probability of \eqref{UI}%
Simple derivations yield that
\begin{eqnarray}\label{2nd}
&&Q\Big( \sup_{|t-\tau |<x\zeta ^{-1}}f(t)>b,\sup_{|t-\tau|>x\zeta^{-1}}f(t)\leq b,\max_{i=1}^{m}f(t_{i})\leq
b\Big)\notag\\
&=&E_b \Big[Q(\max_{i=1}^{m}f(t_{i})\leq
b|f); \sup_{|t-\tau |<x\zeta ^{-1}}f(t)>b,\sup_{|t-\tau|>x\zeta^{-1}}f(t)\leq b\Big ]\notag\\
&\leq &E_b\Big[ (1-\beta (A_{b}))^{m};\sup_{|t-\tau |<x\zeta ^{-1}}f(t)>b%
\Big]
\end{eqnarray}%
where $$\beta (A_{b})=\zeta^d\times mes(A_{b}\cap B(\tau, x/\zeta))\times \inf_{|t|\leq x}k(t)$$ is the lower bound of the probability  that $Q(t_i \in A_b|f)$ and
$B(\tau, x)$ is the ball centered around $\tau$ with radius $x$. In what follows, we need
to show that $mes(A_{b})$ cannot be too small on the set $\{\sup_{|t-\tau
|<x\zeta ^{-1}}f(t)>b\}$ and therefore $\beta(A_b)$ cannot be too small.
%To obtain a bound on (\ref{2nd}),
We write
\begin{equation*}
\mathcal E_{1}=\{ \sup_{|t-\tau |<x\zeta ^{-1}}f(t)>b\}
\end{equation*}%
and write (\ref{2nd}) as
\begin{equation*}
  E_b[(1-\beta(A_b))^m; \mathcal E_1]=E_b[(1-\beta(A_b))^m; \mathcal E_1, D_{\lambda_3, \delta_1}^c]+E_b[(1-\beta(A_b))^m; \mathcal E_1,D_{\lambda_3, \delta_1} ]
\end{equation*}
%\begin{equation*}
%Q (E_1) =Q(E_{1};D_{\lambda _{3},\delta _{1}}^{c})+Q(E_{1};D_{\lambda
%_{3},\delta _{1}}),
%\end{equation*}%
where, for some $\lambda_3$ and $\delta_1$ positive, we define
\begin{equation*}
D_{\lambda _{3},\delta _{1}}=\{\sup_{\substack{ _{\substack{ |s-t|\leq
\lambda _{3}\zeta ^{-1}}} \\ s,t\in B(\tau ,x\zeta ^{-1})}}%
|f(s)-f(t)|\leq\delta _{1}b^{-1}\}.
\end{equation*}%
For some $\varepsilon_0$ small, we choose $\delta_1=\varepsilon_0\delta$ and
$$\lambda _{3} = \varepsilon_0 \delta_1 ^{2/{\alpha}+1/{\beta_1+\varepsilon_0}}.$$
We apply the Borel-TIS lemma to the double-indexed process $\xi(s,t) = f(s) - f(t)$ whose variance is bounded by Lemma \ref{LemmaMu} (ii).
Thus, we obtain the following bound
$$E_b\Big[(1-\beta(A_b))^m; \mathcal E_1, D_{\lambda_3, \delta_1}^c\Big] \leq  Q(D_{\lambda _{3},\delta _{1}}^{c})\leq \delta . $$ Therefore, (\ref{2nd}) is bounded by
\begin{equation*}
\delta +E_b\Big[(1-\beta(A_b))^m; \mathcal E_1,D_{\lambda_3, \delta_1} \Big].
\end{equation*}%
We further split the expectation
\begin{eqnarray*}
E_b\Big[(1-\beta(A_b))^m; \mathcal E_1,D_{\lambda_3, \delta_1} \Big] &\leq& E_b\Big[(1-\beta (A_{b}))^{m};D_{\lambda _{3},\delta
_{1}};\sup_{|t-\tau |<x\zeta ^{-1}}f(t)>b+\delta _{1}b^{-1},\mathcal E_1\Big]\\
&& +{Q}\Big[D_{\lambda _{3},\delta
_{1}};b<\sup_{|t-\tau |<x\zeta ^{-1}}f(t)\leq b+\delta _{1}b^{-1},\mathcal E_1\Big].
\end{eqnarray*}
We proceed to  providing a bound of the second term by considering the standardized process $g(t)=b(f(\tau+t/\zeta)-b)$ conditional on $f(\tau)=\gamma+\frac{z}{b}$. $g(t)$ can be written as
\begin{equation}\label{gg}
  g(t)=\frac{C(t/\zeta+\tau)}{C(\tau,\tau)}z+l(t),
\end{equation}
where $l(t)$ is a random field whose distribution is independent of $z$.  So we have
\begin{eqnarray*}
  Q(b<\sup _{|t-\tau|< x\zeta^{-1}} f(t)<b+\delta_1b^{-1})
  &=& Q\Big(\sup_{|t|\leq x} \frac{C(t/\zeta+\tau)}{C(\tau,\tau)} z +l(t)\in (0,\delta_1)\Big)\\
%  &=& Q(z+\sup_{|t|\leq x}l(t) \in (0,\delta_1) )\\
  &=& O(\delta_1)
\end{eqnarray*}
The last equality holds because $z$ has a  density  bounded everywhere (asymptotically exponential), and $\frac{1}{2}<\frac{C(t/\zeta+\tau)}{C(\tau,\tau)}<\frac{\sigma_T^2}{\sigma^2(\tau)}$. Given a realization of $l(t)$ , $\sup_{|t|\leq x} \frac{C(t/\zeta+\tau)}{C(\tau,\tau)} z + l(t) \in (0,\delta_1)$ implies that $z$ has to fall in an interval with length less than $2\delta_1$.
%According to the standard results \cite{tsirel1976density}, $\sup _{|t-\tau|< A\zeta^{-1}} f(t)$ has a bounded density function at $b$. In addition, there exists $\varepsilon_1$ such that $Q(\sup _{|t-\tau|< A\zeta^{-1}} f(t) > b)\geq \varepsilon_1$.
Thus, if we choose $\varepsilon_0$ small and $\delta_1 =\varepsilon_0\delta$, then
\begin{equation*}
Q(b<\sup_{|t-\tau|<x\zeta^{-1}}f(t)<b+\delta_1\zeta^{-1})<\delta .
\end{equation*}%
Therefore, we have that (\ref{2nd}) is bounded by
\begin{equation*}
2\delta +E^{Q}[(1-\beta (A_{b}))^{m};D_{\lambda _{3},\delta
_{1}};\sup_{|t-\tau |<x\zeta ^{-1}}f(t)>b+\delta _{1}b^{-1},\mathcal E_1].
\end{equation*}%
Note that, on the set $D_{\lambda _{3},\delta _{1}}$, $mes(A_{b}\cap B(\tau, x\zeta^{-1}))$ is controlled by the overshoot $\sup_{|t-\tau |<x\zeta^{-1}}f(t)-b$, that is, if $\sup_{|t-\tau|<x\zeta^{-1}} f(t) > b+ \delta_1/b$, then $mes(A_b \cap B(\tau, x\zeta^{-1}))\geq \varepsilon_0\lambda_3^d \zeta^{-d}$.
%Thus, on the set $D_{\lambda _{1},\delta_{1}}\cap \{\sup_{|t-\tau |<A\zeta ^{-1}}f(t)>b+\delta _{1}b^{-1}\}$, the volume of the excursion set has a lower bound $mes(A_b\cap B(\tau, A\zeta^{-1}))) \geq \varepsilon_0\lambda_3^d\zeta^{-d}$.
In addition, the density $k_{\tau,\zeta}(t)$ is bounded from below by $x^{-d-\varepsilon_1}$ for $t\in B(\tau, x\zeta^{-1})$.
Thus, the probability $\beta(A_b)$ has a lower bound
$$\beta(A_b) \geq \varepsilon_0 x^{-d-\varepsilon_1} \lambda_3 ^d \geq  \varepsilon_0 \delta ^{2d/\alpha+d/\beta_1+2\varepsilon_0}.$$
The last step of the above inequality follows from that $x = \min\{(-\log \delta)^{4/\alpha},\delta' \zeta\}$.
Thus, we have that (\ref{2nd}) is bounded by%
\begin{equation*}
2\delta +(1-\varepsilon_0 \delta ^{2d/\alpha+d/\beta_1+2\varepsilon_0})^{m}.
\end{equation*}%
We choose for some $\varepsilon_0 $ and $\kappa>0$
 $$m = \kappa \delta ^{-2d/\alpha-d/\beta_1-2\varepsilon_0} $$
  and therefore
\begin{equation*}
Q\Big( \sup_{|t-\tau
|<x\zeta ^{-1}}f(t)>b,\sup_{|t-\tau|>x\zeta^{-1}}f(t)\leq b;\max_{i=1}^{m}f(t_{i})\leq b\Big)\leq 4\delta .
\end{equation*}%

\subsubsection{The last term in \eqref{UI}.}
%	We now proceed to the last term in \eqref{2nd}.
	According to the result in Lemma \ref{LemmaI21} (and the corresponding result as in \eqref{L2} for the non-constant variance case presented in the Supplemental Material), we can choose $\varepsilon_0$ and $\delta'$ such that
$$Q(\sup_{|t-\tau|\geq \delta'} f(t) \geq \gamma )\leq e^{-\varepsilon_0 b^2}.$$
There are two cases: $\delta > e ^{- \varepsilon_0 b^2}$ and $\delta \leq e ^{- \varepsilon_0 b^2}$.

\paragraph{Case 1: $\delta > e ^{- \varepsilon_0 b^2}$.} In this case, The last term in \eqref{UI} is bounded trivially by
$$Q\Big( \sup_{|t-\tau |\geq \delta'}f(t)>b;\max_{i=1}^{m}f(t_{i})\leq b\Big)\leq Q(\sup_{|t-\tau|\geq \delta'} f(t) \geq \gamma ) \leq  \delta.$$

\paragraph{Case 2: $\delta  < e ^{- \varepsilon_0 b^2}$.} We need a similar analysis to that of the second term.
We now split the probability for $\delta_2=\delta^{1+\varepsilon_0}$
\begin{eqnarray*}\label{split}
&&Q\Big( \sup_{|t-\tau |\geq \delta'}f(t)>b;\max_{i=1}^{m}f(t_{i})\leq b\Big) \notag\\
&&~~~\leq Q\Big( \sup_{|t-\tau |\geq \delta'}f(t)\in [b,\delta_2 b^{-\lambda}]\Big) + Q\Big( \sup_{|t-\tau |\geq \delta'}f(t)>b + \delta_2 b^{-\lambda};\max_{i=1}^{m}f(t_{i})\leq b\Big) .
%&&~~~\leq 3 \delta + Q\Big( \sup_{|t-\tau |\geq \delta'}f(t)>b + \delta_2 b^{-\lambda};\max_{i=1}^{m}f(t_{i})\leq b\Big) .
\end{eqnarray*}
%
%Let $f(\tau) = \gamma + z/b$. The last term is bounded by
%$$Q\Big( \sup_{|t-\tau |\geq \delta'}f(t)>b,\max_{i=1}^{m}f(t_{i})\leq b\Big) \leq Q(z\geq -\kappa_0 \log \delta) +  Q\Big( \sup_{|t-\tau |\geq \delta'}f(t)>b,\max_{i=1}^{m}f(t_{i})\leq b, z< -\kappa_0\log \delta\Big).$$
%Note that $z$ is the overshoot of a Gaussian distribution over the level $\gamma$. Thus, we can choose $\kappa_0$ large (and $\kappa_0 \varepsilon_0$ small) such that
%$$Q(z\geq \kappa_0 \log \delta)\leq \delta.$$
%We now consider the second term. Note that $z< \kappa_0\log \delta$ suggests that $z/ b < \varepsilon_0 \kappa_0$
%
%
%
%
%\jc{rewrite.}
%
%
We now consider the first term split the set $\{t:|t-\tau|>\delta'\}$ into two parts. Define the set $F=\{t:\frac{C(t,\tau)}{C(\tau,\tau)}>\frac{1}{(-\log\delta_2)^2}\}$,
%and the probability is split into
%\begin{eqnarray*}\label{split}
%&&Q\Big( \sup_{|t-\tau |\geq \delta'}f(t)>b;\max_{i=1}^{m}f(t_{i})\leq b\Big) \notag\\
%&&~~~\leq Q\Big( \sup_{|t-\tau |\geq \delta', t\in F}f(t)>b;\max_{i=1}^{m}f(t_{i})\leq b\Big) + Q\Big( \sup_{|t-\tau |\geq \delta',t\in F^c}f(t)>b;\max_{i=1}^{m}f(t_{i})\leq b\Big) .
%\end{eqnarray*}
We start with  the small overshoot probability on the set $F$
$$Q\Big(b< \sup_{|t-\tau|>\delta',t\in F} f(t)\leq b+ \delta_2/b\Big).$$
Using the representation \eqref{gg},  applying similar analysis as that of the second term, we have that
\begin{eqnarray}
  &&Q\Big(b<\sup_{|t-\tau|\geq\delta', t\in F}f(t)<b+\delta_2b^{-1}\Big)\notag\\
  &&~~~\leq Q\Big(\sup_{|\frac{t}{\zeta}|>\delta', \frac{t}{\zeta}+\tau\in F} \frac{C(t/\zeta+\tau)}{C(\tau,\tau)} z+l(t)\in (0,\delta_2)\Big)=O((-\log\delta_2)^2\delta_2)\leq \delta.\label{overshoot}
\end{eqnarray}
The last two steps are based on the fact that $z$ is a random variable independent of $l(t)$ and has bounded density. Thus, the above probability is bounded by
$$\sup_x P(x < z < x+ (\log \delta_2)^2 \delta_2)= O((\log \delta_2)^2 \delta_2).$$
%For $t\in F^c$ and $Z<-\log\delta_2$, we apply Theorem 2 of \cite{TsirelsonMaxDens} and obtain the following bound. There exists $\lambda > 0$ such that
We will return to this estimate soon.

We now consider $t$ in $F^c$. For some $\kappa_0$ large, we have that $Q(z> -\kappa_0 \log \delta_2) < \delta_2$.  Thus, we only consider $z<-\kappa_0\log\delta_2$.
Conditional on $f(\tau) = \gamma + z/b$, conditional mean is $\sup_{t\in F^c}\mu_\tau(t-\tau)  \leq C>0.$ In addition, the conditional variance of $f(t)$ on the set $F^c$ is almost $\sigma^2(t)$. Thus, we can apply classic results on the density estimation of the $\sup f(t)$ (c.f.~Theorem 2 of \cite{TsirelsonMaxDens}) and have that conditional on $f(\tau)=\gamma+\frac{z}{b}$, $\sup_{|t-\tau|\geq \delta', F^c}f(t)$ has a bounded density over $[b, b+\delta_2 b^{-\lambda}]$ for some $\lambda\geq 1$, and thus $$Q(\sup_{|t-\tau|\geq\delta', t\in F^c}f(t)\in [b, b+\delta_2 b^{-\lambda}]|f(\tau)=\gamma+\frac{z}{b})=O(\delta_2).$$
Summarizing the above results, we have that
\begin{eqnarray*}
&&Q(\sup_{|t-\tau|\geq\delta'}f(t)\in [b, b+\delta_2 b^{-\lambda}]) \\
&\leq& Q(\sup_{|t-\tau|\geq\delta', t\in F}f(t)\in [b, b+\delta_2 b^{-\lambda}]) \\
&&+ Q(z \geq -\kappa_0 \log\delta_2) +  Q(\sup_{|t-\tau|\geq\delta', t\in F^c}f(t)\in [b, b+\delta_2 b^{-\lambda}], z \leq -\kappa_0 \log\delta_2)\\
&\leq & 3 \delta.
\end{eqnarray*}
Thus, the last term in \eqref{UI} is bounded by
\begin{eqnarray*}\label{split}
&&Q\Big( \sup_{|t-\tau |\geq \delta'}f(t)>b;\max_{i=1}^{m}f(t_{i})\leq b\Big) \notag\\
%&&~~~\leq Q\Big( \sup_{|t-\tau |\geq \delta'}f(t)\in [b,\delta_2 b^{-\lambda}]\Big) + Q\Big( \sup_{|t-\tau |\geq \delta'}f(t)>b + \delta_2 b^{-\lambda};\max_{i=1}^{m}f(t_{i})\leq b\Big) \\
&&~~~\leq 3 \delta + Q\Big( \sup_{|t-\tau |\geq \delta'}f(t)>b + \delta_2 b^{-\lambda};\max_{i=1}^{m}f(t_{i})\leq b\Big) .
\end{eqnarray*}
For the second term, we apply the old trick by  choosing
$$\lambda_4 = \delta_2^{2/\alpha+1/\beta_1+\varepsilon_0} b^{-2\lambda/\alpha-\lambda/\beta_1},
$$
and thus
\begin{equation}\label{continuity}
Q(\sup_{|s - t| < \lambda _4} |f(s) - f(t) | >  \delta_2 b^{-\lambda}) <  \delta_2.
\end{equation}
Note that $b^2\leq -\varepsilon_0^{-1} \log \delta_2$, we can choose a different $\varepsilon_0$ such that  $\lambda_4$ can be simplified to
$$\lambda_4 = \delta_2^{2/\alpha+1/\beta_1+\varepsilon_0}.$$
If $\sup_{|s - t| < \lambda _4} |f(s) - f(t) | <  \delta_2 b^{-\lambda}$ and $\sup_{|t-\tau |\geq \delta'}f(t)>b + \delta_2 b^{-\lambda}$, we have that
$$\beta(A_b)\geq \varepsilon_0 \lambda _4^d\zeta ^{-d - \varepsilon _1}.$$
With a different choice of $\varepsilon_0$, we choose
\begin{equation}\label{mm}
m = -2 \lambda _4 ^{-d}\zeta ^{d + \varepsilon _1}\log \delta   = O(\delta^{-d(2/\alpha+1/\beta_1)-\varepsilon_0}),
\end{equation}
then we have
\begin{equation}\label{mmm}
E_b[(1-\beta (A_{b}))^{m};\sup_{|s - t| < \lambda _4} b<|f(s) - f(t) | <  \delta_2 b^{-\lambda},f(t)>b + \delta_2 b^{-\lambda}]\leq \delta %^{1+ \varepsilon_0}.
\end{equation}
Therefore, combining the bounds in \eqref{overshoot}, \eqref{continuity}, and \eqref{mmm}, if $\varepsilon  < e ^{- \varepsilon_0 b^2}$ and we choose $m$ as in \eqref{mm} and, then
\begin{equation*}
Q\Big( \sup_{|t-\tau|>\delta'} f(t)>b;\max_{i=1}^{m}f(t_{i})\leq b\Big) \leq 5\delta .
\end{equation*}%

Putting together the bounds for all the three terms in \eqref{UI}, we have that
$$Q\Big( M>b;\max_{i=1}^{m}f(t_{i})\leq b\Big) \leq 5\delta.$$
If we choose $\delta = \varepsilon^{1+ \varepsilon_0}$ and
$$m =O(\delta^{-d(2/\alpha+1/\beta_1+\varepsilon_0)})= O(\varepsilon^{-d(2/\alpha+1/\beta_1)-2d\varepsilon_0})$$
then according to the bounds on the bound in \eqref{UNI}, we have that
\begin{equation*}
E^{Q}J_{1}\leq \zeta ^{d}\varepsilon .
\end{equation*}%
Similarly, according to the integrability of $\zeta ^{-2d}/mes^{2}(A_{\gamma })$%
, by choosing the same $m$, there exists a $\kappa _{0}$ such that%
\begin{equation*}
E^{Q}(J_{1}^{2})\leq \kappa _{0}\zeta ^{2d}.
\end{equation*}

\subsection{The $J_{2}$ term}

We now proceed to
\begin{equation*}
J_{2}=I(\max_{i=1}^{m}f(t_{i})>b)\left[ \frac{1}{mes(A_{\gamma })}-\frac{1}{%
\widehat{mes}(A_{\gamma })}\right] .
\end{equation*}%
We study the behavior of $J_{2}$ by means of the scaled process $g(t)$ defined as in \eqref{gt}. For the analysis of $J_2$, we translate everything to the scale of $g(t)$. Recall the process $g(t)$ given by \eqref{gt} is
	\begin{equation}
	g(t) = b(f(\tau + t/\zeta)-b),
	\end{equation}
For each $t$, $f(\tau +t/\zeta )>\gamma $ if and only if $g(t)>-1$.

Conditional on $\tau$,  $t_1,...,t_m$ are i.i.d.~with density $k_{\tau,\zeta} (t)$ defined as in \eqref{l}. Let $s_{i}=(t_{i}-\tau )\zeta $ and thus $s_{1},...,s_{m}$ are i.i.d.~following density $k(s)$. We can then rewrite the estimator in \eqref{vol} as
\begin{equation*}
\widehat{mes}(A_{\gamma })=\frac{\zeta ^{-d}}{m}\sum_{i=1}^{m}\frac{%
I(g(s_{i})>-1)}{k(s_{i})}.
\end{equation*}%
Thus, $\widehat{mes}(A_{\gamma })$ is an unbiased estimator of
$mes(A_{\gamma })$, that is, $E(\widehat{mes}(A_{\gamma })|f)=mes(A_{\gamma })$. Conditional on a particular realization of $f(t)$ (or equivalently, g(t)), the variance of $\widehat{mes}(A_{\gamma })$ is given by
\begin{equation*}
Var(\widehat{mes}(A_{\gamma
})|f)= \frac{\kappa _{f}\zeta ^{-2d}}{m},
\end{equation*}%
where
\begin{equation}\label{term1}
\kappa _{f}=Var\Big[ \frac{I(g(S)>-1)}{k(S)}\Big|f\Big] \leq k^{-2}(t_{f})
\end{equation}%
and%
\begin{equation}
t_{f}=\max (|t|:g(t)>-1).  \label{tf}
\end{equation}%
Note that the following inequality $\frac{1}{1+x}-1\geq -x$. Therefore,%
\begin{equation*}
\frac{1}{mes(A_{\gamma })}-\frac{1}{\widehat{mes}(A_{\gamma })}\leq \frac{\widehat{%
mes}(A_{\gamma })-mes(A_{\gamma })}{mes^2(A_\gamma)}
\end{equation*}%
Therefore,%
\begin{equation*}
E\Big[ \Big( \frac{1}{mes(A_{\gamma })}-\frac{1}{\widehat{mes}(A_{\gamma })%
}\Big) ^{2};\widehat{mes}(A_{\gamma })>mes(A_{\gamma })~\Big|~f\Big] \leq
\frac{\kappa _{f}\zeta ^{-2d}}{m\times mes^4(A_\gamma)}.
\end{equation*}%
It is the expectation on the set $\{\widehat{mes}(A_{\gamma })<mes(A_{\gamma })\}$ that induces complications in that the factor $\frac{1}{\widehat{mes}(A_{\gamma })}$ can be very large when there are not many $t_i$'s in the excursion set $A_{\gamma}$.
We now proceed to this case. Conditional on a particular realization of $f$ (and equivalently the
process $g(t)$), the analysis consists of three steps.

\paragraph{Step 1.} Define the $f$-dependent probability%
\begin{equation}\label{pf}
p_{f}\triangleq Q(t_i \in A_{\gamma}|f)=\int_{A_{\gamma }}k_{\tau ,\zeta }(t)dt=\int_{A_{-1}^{g}}k(t)dt
\end{equation}%
Using standard exponential change of measure techniques for large deviations \cite{dembo2009large}, we obtain that
\begin{equation}\label{largedeviation}
Q\left[ \sum_{i=1}^m I(t_{i}\in A_{\gamma })\leq p_{f}(1-\delta_3)m\Big | f\right] \leq
e^{-mI(\delta_3,p_{f})}
\end{equation}%
for all $\delta_3\in (0,1)$, where the rate function%
\begin{equation*}
I(\delta _{3},p_{f})=\theta _{\ast }p_{f}(1-\delta _{3})-\varphi (\theta
_{\ast })
\end{equation*}%
and
\begin{equation*}
\varphi (\theta )=\log (1-p_{f}+p_{f}e^{\theta }),\quad \theta _{\ast }=\log
\left( 1-\frac{\delta _{3}}{1-p_{f}(1-\delta _{3})}\right) .
\end{equation*}%
By elementary calculus, if we choose $\delta _{3}=\frac{1}{2}$, then we have
that for some $\varepsilon_0 >0$
\begin{equation*}
I(\delta _{3},p_{f})\geq \varepsilon _{0}p_{f} \quad \mbox{for all $p_f>0$.}
\end{equation*}
Therefore,
\begin{eqnarray*}
&&  E\Big[ \Big( \frac{1}{mes(A_{\gamma })}-\frac{1}{\widehat{mes}(A_{\gamma })%
}\Big) ^{2};\widehat{mes}(A_{\gamma })\leq mes(A_{\gamma }),\max_{i=1}^{m}f(t_{i})>b, \sum_{i=1}^m I(t_{i}\in A_{\gamma })\leq \frac{p_{f}m}2~\Big|~f\Big]\\
&\leq&  E\Big[ \frac{4}{\widehat {mes}^2(A_{\gamma })};\widehat{mes}(A_{\gamma })\leq mes(A_{\gamma }),\max_{i=1}^{m}f(t_{i})>b, \sum_{i=1}^m I(t_{i}\in A_{\gamma })\leq \frac{p_{f}m} 2~\Big|~f\Big].
\end{eqnarray*}
Note that there is at least one $t_i$ in the excursion set $A_\gamma$. Therefore, the estimator $\widehat {mes}(A_{\gamma })\geq m^{-1}\zeta^{-d} k^{-1}(t_f)$. Thus, the above expectation is upper bounded by
$$\leq \kappa k(t_f)^{-2} m^2 \zeta ^ {2d} e^{-\varepsilon_0 m p_f}.$$
\paragraph{Step 2.} We consider the situation that $\sum I(t_{i}\in A_{\gamma })>\frac{p_{f}m}{2}$.
The unbiasedness of $\widehat{mes} (A_\gamma)$ suggests that
$$mes(A_\gamma) = E\Big(\frac{1}{\zeta^{d}k(S)} ~|~ S \in A_{-1}^g\Big) p_f,$$
where $S$ is a random index following density $k(s)$. Note that on the set $A_{-1}^g$, $k(t_f)\leq k(S)\leq \kappa_1$. Thus, if we let  $\lambda _{f}=\kappa^{-1} _{1}k(t_{f})$, then on the set $\{\sum I(t_{i}\in A_{\gamma })>\frac{p_{f}m}{2}\}$ we have
\begin{equation*}
\widehat{mes}(A_{\gamma })\geq \frac{\lambda _{f}mes(A_{\gamma })}2.
\end{equation*}%
Thus, using Taylor expansion, we have that%
\begin{eqnarray*}
&&E_b\Big[ \Big( \frac{1}{mes(A_{\gamma })}-\frac{1}{\widehat{mes}(A_{\gamma
})}\Big) ^{2};\widehat{mes}(A_{\gamma})< {mes}(A_{\gamma});\sum I(t_{i}\in A_{\gamma })>\frac{p_{f}m}{2}\Big|f\Big]  \\
&\leq &E_b\left[ \frac{2^4\left( mes(A_{\gamma })-\widehat{mes}(A_{\gamma })\right) ^{2}}{\lambda _{f}^{4}mes^4(A_\gamma)};\widehat{mes}(A_{\gamma})< {mes}(A_{\gamma});\sum I(t_{i}\in A_{\gamma })>\frac{p_{f}m}{2}~\Big |~f\right]  \\
&\leq &\frac{2^4\kappa _{f}\zeta ^{-2d}}{m\lambda _{f}^{4} mes^4(A_\gamma)}.
\end{eqnarray*}

\paragraph{Step 3.} We combine the previous analysis and have that%
\begin{equation}
E_b(J_{2}^{2}|f)\leq \frac{2^4\zeta ^{-2d}}{mes^4 (A_\gamma)} \frac{\kappa_1^4  }{k^{4} (t_{f})m}
+\frac{\kappa_f\zeta^{-2d}}{m\times mes^4(A_\gamma)}+ k(t_f)^{-2}m^2\zeta^{2d}e^{-\varepsilon_0 mp_f}.  \label{total}
\end{equation}%
The density $k(t)$ has a  heavy tail that is
\begin{equation*}
k(t)\sim \frac{1}{|t|^{d+\varepsilon _{1}}}
\end{equation*}%
and $k(t)\leq \kappa _{1}$ for all $t$.
In Step 3, we provide a bound on the distribution of $t_{f}$ and $%
p_{f}$.

We start with $t_{f}$. For each $s>0$, $t_{f}>s$ if and only if $%
\sup_{|t-\tau |>s}g(t)>-1$.
%Let
%\begin{equation*}
%f(\tau )=\gamma +z/b,
%\end{equation*}%
%where $z$ is an asymptotically exponential random variable under $Q_b$.
According to the results in Lemmas \ref{LemmaI21} and \ref{LemmaI22} (and the corresponding bounds in \eqref{L2} and \eqref{L3} for the non-constant variance case presented in the Supplemental Material),
for $s$ sufficiently large, there exists some $\varepsilon _{0}>0$ such that
\begin{equation}\label{tfpdf}
Q(t_{f}>s)=Q(\sup_{|t-\tau |>s}g(t)>-1)\leq \exp \{-s^{\varepsilon_0}\},\qquad \mbox{for $s< \delta ' \zeta$}
\end{equation}
and
$$Q(t_{f}>s)\leq \exp(- \varepsilon_0 b^2 ),\qquad \mbox{for $s > \delta' \zeta$.}$$
Therefore, all moments of $k^{-1}(t_f)$ is bounded.%
\begin{eqnarray*}
&&E_b[ k^{-l}(t_{f})] \leq  E_b[t_f^{(d+\epsilon_1)l}]\leq \kappa_l
%\\
%&=&O(1)E_b\left[
%|t_{f}|^{(d+\varepsilon _{1})l};t_f>A_0^\alpha, 2Z<A_0^\alpha\right] +Q(2Z>A_0^\alpha)+A_0^{(d+\varepsilon_1)l}\\
%&=&O(1)\left(e^{-A_0^{\frac{\alpha l(d+\varepsilon_1)}{4}}}+e^{-\frac{A_0^\alpha}{2\sigma(\tau)^2}}+A_0^{(d+\varepsilon_1)l}\right)\\
%&\leq&\kappa,
%E_b\left[
%\max (Z^{(8d+ 8 \varepsilon_1)/\alpha },\kappa _{1}^{-4})\right] \leq \kappa _{3}.
\end{eqnarray*}%
for some constant $\kappa_l$ possibly depending on $l$. Thus, by Cauchy-Schwarz inequality, the expectation of the first two terms in \eqref{total} can be bounded as follows
\begin{eqnarray*}
E\Big[\frac{2^4\zeta ^{-2d}}{mes^4 (A_\gamma)} \frac{\kappa_1^4  }{k^{4} (t_{f})m};M>b\Big]&\leq& \frac{O(1)} m \sqrt {E\Big[\frac{\zeta ^{-4d}}{mes^8 (A_\gamma)} \Big] E(k^{-8}(t_f))}\leq \frac {\kappa \zeta^2} {m} \\
E\Big[\frac{\kappa_f\zeta^{-2d}}{m\times mes^4(A_\gamma)}\Big] &\leq &\frac{O(1)} m \sqrt {E\Big[\frac{\zeta ^{-4d}}{mes^8 (A_\gamma)} \Big] E(k^{-4}(t_f))}\leq \frac {\kappa \zeta^2} {m}.
\end{eqnarray*}

We now proceed to the third term in (\ref{total}) concerning $p_{f}$.  The expectation of this term is bounded by
\begin{equation*}
E_b(m^2 k(t_f)^{-2}e^{-m\varepsilon _{0}p_{f}}; M>b)\leq\sqrt{E_b(m^4 e^{-2m\varepsilon _{0}p_{f}};M>b)}\sqrt{E_b (k^{-4}(t_f))}.
\end{equation*}
%where
%$$Q(p_f \leq m^{-1/2}, M>b).$$
The second term $\sqrt{E_b (k^{-4}(t_f))}$ is $O(1)$. We proceed to the first term
\begin{eqnarray*}
E_b(m^4 e^{-2m\varepsilon _{0}p_{f}};M>b) &= &E_b(m^4 e^{-2m\varepsilon _{0}p_{f}}; p_f \geq  m^{-1/2})+E_b(m^4 e^{-2m\varepsilon _{0}p_{f}}; p_f \leq  m^{-1/2}, M>b)\\
&\leq & m^4 e^{-2\varepsilon_0 \sqrt m} + m^4 Q(p_f \leq m^{-1/2}, M>b).
\end{eqnarray*}
We now proceeding to controlling $Q(p_f \leq m^{-1/2}, M>b)$.
Note that
\begin{equation*}
p_{f}\geq k(t_{f})mes(A_{-1}^{g}).
\end{equation*}%
For each $x>0$,%
\begin{eqnarray}\label{pf}
Q(p_{f}<x, M > b)&\leq& Q\Big(k(t_{f})<\sqrt{x}~~\text{ or }~~mes(A_{-1}^{g})<\sqrt{x}, M >b \Big)\notag\\
&\leq&
Q(t_{f}>x ^{- \frac 1 {2(d+\varepsilon_1)}})+Q(mes(A_{-1}^{g})<\sqrt{x}, M >b).
\end{eqnarray}%
According to the bounds in \eqref{tailbound} and \eqref{mesub1}, for some $\delta _{0}>0$ and $\varepsilon _{0}>0$, we have that%
\begin{equation*}
Q(mes(A_{-1}^{g})<\sqrt{x}, M > b)  = Q(mes (A_\gamma) < \zeta^{-d} \sqrt x, M > b )\leq e^{-x^{-\varepsilon _{0}/d}}
\end{equation*}%
for $x$ sufficiently small.
According to the previous result, we have that
\begin{equation*}
Q(t_{f}>x ^{- \frac 1 {2(d+ \varepsilon_1)}})\leq e^{-x^{-\varepsilon _{0}}}, \quad \mbox{for $ x ^{-  \frac 1 {2(d+ \varepsilon_1)}} < \delta' \zeta$}
\end{equation*}%
and
$$Q(t_{f}>x ^{- \frac 1 {2(d+ \varepsilon_1)}})\leq e^{-\varepsilon_0 b^2},\quad \mbox{for $ x ^{-  \frac 1 {2(d+ \varepsilon_1)}} \geq \delta' \zeta$.}$$
Thus, for some $\lambda$ large enough and $\varepsilon_0$ small enough, we have that
$$Q(p_f \leq m^{-1/2}, M>b)\leq e^{-m^{\varepsilon_0}},\quad \mbox{for $m < b^{\lambda}$};$$
 for $m> b^\lambda$ (with $\lambda$ sufficiently large),  $t_f > m ^{ \frac 1 {4(d+\varepsilon_1)}}$ implies that $\tau + t_f /\zeta \notin T$, that is, $m ^{ \frac 1 {4(d+\varepsilon_1)}}$ is too large and thus
$$Q(p_f < m^{-1/2})=0,\qquad \mbox{for $m> b^\lambda$.}$$
Therefore,
$$m^4 Q(p_f \leq m^{-1/2}, M>b) \leq \kappa m^4 e^{-m^{\varepsilon_0}}$$
for $m$ sufficiently large and furthermore
$$E_b(m^4 k(t_f)^{-2}e^{-m\varepsilon _{0}p_{f}}; M>b)\leq \kappa m^4 e^{-m^{\varepsilon_0}/2}.$$

%There are two cases.
%\begin{itemize}
%\item[Case 1.] For some $\lambda $ large and $m < b^\lambda$.
%We choose $x=m^{-1/2}$ and thus for $\varepsilon_0$ small
%\begin{eqnarray*}
%Q\Big( t_{f}>m ^{ \frac 1 {4(d+\varepsilon_1)}}\Big)  &\leq &Q(2Z\geq m ^{ \frac 1 {4(d+\varepsilon_1)}})+Q(2Z<m ^{ \frac 1 {4(d+\varepsilon_1)}},t_{f}>m ^{ \frac 1 {4(d+\varepsilon_1)}}) \\
%&\leq &e^{-m^{\varepsilon_0}}.
%\end{eqnarray*}%
%\item[Case 2.] When $m \geq b^\lambda$, note that $|t_f | = O(\zeta)$ and thus $k(t_f) = \Omega (b^{-\lambda_1})$ for some $\lambda_1$ large (depending on $\alpha$, $d$, and $\varepsilon_1$). Thus, we can choose $\lambda > \lambda_1$ such that
%    $$Q(k(t_{f})<\sqrt{x} = m^{-1/2})=0.$$
%\end{itemize}
%
%In summary, we have that%
%\begin{equation}\label{pfpdf}
%Q(p_{f}<m^{-1/2}, M >b)\leq e^{-m^{\varepsilon _{0}}}
%\end{equation}%
%and furthermore,%
%\begin{equation*}
%E_b(m^2 k(t_f)^{-2}e^{-m\varepsilon _{0}p_{f}})\leq\sqrt{E_b(m^4 e^{-2m\varepsilon _{0}p_{f}})}\sqrt{E_b (k(t_f)^{-4})} =O(e^{-m^{\varepsilon
%_{0}'}}).
%\end{equation*}

Summarizing the results in all the three steps, we have that
$$E_b (J_2^2 ) \leq \frac{\kappa\zeta^{-2d}}{m}.
%E_b \Big( \frac{\zeta ^{-2d}}{mes^4 (A_\gamma)}\frac{\kappa }{k^{4}(t_{f})m} + m^2e^{-m\varepsilon _{0}p_{f}} \Big )+k(t_f)^{-2} m^2 \zeta ^ {2d} e^{-\varepsilon_0 m p_f} \leq \frac{\kappa \zeta^{2d}}{ m}.
$$
Therefore, if we choose
$$m = \kappa \max \{\varepsilon ^{-2}, \varepsilon ^{-d(2/\alpha+1/\beta_1+\varepsilon_0)}\}=O(\varepsilon ^{-d(2/\alpha+2/\beta_1)})$$
then,
$$E_b|\hat Z_b - Z_b | = E_b |J_1 + J_2| \int_T P(f(t) > \gamma)dt \leq \varepsilon \zeta^{d} \int_T P(f(t) > \gamma)dt$$
and
$$E_b(\hat Z_b - Z_b )^2 \leq  \kappa \zeta^{2d}\Big ( \int_T P(f(t) > \gamma)dt\Big)^2.$$

\section{Proof of Theorem \ref{ThmInt}}\label{SecInt}

\subsection{The asymptotic lower bound and the continuous estimator}

We start the analysis by first establishing an asymptotic lower bound of $%
v(b)$. Note that%
\begin{equation*}
v(b)=E(mes(A_{\gamma }))E_b\left[ \frac{1}{mes(A_{\gamma })}%
\int_{A_{b}}\xi(t)dt\right] .
\end{equation*}%
Since $\xi(t)$ is bounded by $a_2 $, then $v(b)\leq a_2 E(mes(A_{\gamma
}))$. In addition, a lower bound can be given by
\begin{eqnarray*}
E\Big(\int_{A_b} \xi(t) dt\Big)\geq a_1 E(mes(A_b))
\end{eqnarray*}
%
%
%
%
%there exists $\varepsilon _{0}$ such that%
%\begin{equation*}
%E^{Q}\left[ \frac{1}{mes(A_{\gamma })}\int_{A_{b}}\xi(t)dt\right] \geq
%\varepsilon _{0}.
%\end{equation*}%
Thus%
\begin{equation*}
v(b)=\Theta (1)E(mes(A_{\gamma })).
\end{equation*}%
The second moment of the estimator is%
\begin{equation*}
E_b(Y_{b}^{2})=E^{2}(mes(A_{\gamma }))E_b\left[ \frac{\alpha ^{2}(b)}{%
mes^{2}(A_{\gamma })};M>b\right] \leq a_2 ^{2}E^{2}(mes(A_{\gamma }))\leq
\frac{a_2^2}{a_1^2}v(b).
\end{equation*}

\subsection{Analysis of the discrete estimator}

We start the analysis by the following decomposition%
\begin{eqnarray*}
\hat{Y}_{b}-Y_{b} &=&\left[ \frac{\alpha (b)}{mes(A_{\gamma })}I(\sup
f(t)>b)-\frac{\hat{\alpha}(b)}{\widehat{mes}(A_{\gamma })}%
I(\max_{i=1}^{m}f(t_{i})>b)\right] E(mes(A_{\gamma })) \\
&=&E(mes(A_{\gamma })) \\
&&\times \Big[ \frac{\alpha (b)I(\sup f(t)>b)}{mes(A_{\gamma })}-\frac{%
\alpha (b)I(\max_{i=1}^{m}f(t_{i})>b)}{mes(A_{\gamma })} \\
&&+\frac{\alpha (b)I(\max_{i=1}^{m}f(t_{i})>b)}{mes(A_{\gamma })}-\frac{\hat{%
\alpha}(b)I(\max_{i=1}^{m}f(t_{i})>b)}{\widehat{mes}(A_{\gamma })}\Big].
\end{eqnarray*}%
We redefine the terms%
\begin{eqnarray*}
J_{1} &=&\frac{\alpha (b)I(\sup f(t)>b)}{mes(A_{\gamma })}-\frac{\alpha
(b)I(\max_{i=1}^{m}f(t_{i})>b)}{mes(A_{\gamma })} \\
J_{2} &=&\frac{\alpha (b)I(\max_{i=1}^{m}f(t_{i})>b)}{mes(A_{\gamma })}-%
\frac{\hat{\alpha}(b)I(\max_{i=1}^{m}f(t_{i})>b)}{\widehat{mes}(A_{\gamma })}%
.
\end{eqnarray*}%
Note that the factor $\alpha (b)/mes(A_{\gamma })$ is bounded by $a_2$, so we have
\begin{equation*}
  E_b|J_1|\leq a_2 Q(\sup f(t)>b, \max_{i=1}^m f(t_i)>b),\qquad E_b(J_1^2)\leq a_2^2 Q(\sup f(t)>b, \max_{i=1}^m f(t_i)>b)
\end{equation*}
 According to
the previous analysis, for each $\varepsilon $, there exists an $m=O(\varepsilon^{-d(2/\alpha+1/\beta_1)-\varepsilon_0)}$ such that%
\begin{equation*}
E_b(|J_{1}| ~~|~f)\leq a_2\varepsilon ,\quad E_b(J_{1}^{2}|f)=a_2^2 \varepsilon.
\end{equation*}%
For the second term, we apply similar analysis as the proof for Theorem \ref{ThmDis}.
%Note $E_b(|\hat{\alpha}(b)-\alpha(b)||f)\leq a_2\zeta^{-d}\sqrt{\frac{1}{m k^2(t_f)}}$, and $E_b(|\widehat{mes}(A_\gamma)-mes(A_\gamma)||f)\leq \sqrt{\frac{1}{m k^2(t_f)}}.$
Note that $\alpha(b)\leq a_2mes(A_\gamma)$, so by rearranging terms in $J_2$, we have
\[
|J_{2}|\leq \left[ \frac{|\alpha (b)-\hat{\alpha}(b)|}{mes(A_{\gamma })}%
+a_{2}\frac{|mes(A_{\gamma })-\widehat{mes}(A_{\gamma })|}{mes(A_{\gamma })}%
\right] I(M>b).
\]%
Because $\hat{\alpha}(b)$ is an unbiased estimator for $\alpha(b)$ conditional on $f$, we have
\begin{equation*}
  E_b\Big[(\hat{\alpha}(b)-\alpha(b))^2|f\Big]\leq m^{-1}a_2^2k^{-2}(t_f)\zeta^{-2d}.
\end{equation*}
Thus,
\begin{eqnarray*}
  &&E_b\Big[(\hat{a}(b)-\alpha(b))^2+a_2(mes(A_\gamma)-\widehat{mes}(A_\gamma))^2\Big|f\Big]\\
  &\leq&2E_b\Big[(\hat{\alpha}(b)-\alpha(b))^2|f\Big]+2a_2^2E_b\Big[(mes(A_\gamma)-\widehat{mes}(A_\gamma))|f\Big]\\
  &\leq&4a_2^2m^{-1}k^{-2}(t_f)\zeta^{-2d}.
\end{eqnarray*}
%We continue our analysis by considering two cases.
%The first case is that $\sum_{i=1}^m I(f(t_i)>\gamma)\leq \frac{p_f}{2}m$, where $p_f$ is defined in \eqref{pf}. In this case, on the event $\{\max_{i=1}^m f(t_i)>b\}$, we have $\widehat{mes}(A_\gamma)\geq\kappa_1^{-1}m^{-1}\zeta^{-d}$. According to \eqref{largedeviation}, we have
%\begin{eqnarray*}
%  &&E_b\Big[|J_2|; \sum_{i=1}^m I(f(t_i)>\gamma)\leq \frac{p_f}{2}m, \max_{i=1}^m f(t_i)>b|f\Big]\\
%  &\leq&\kappa_1 m\zeta^d E_b\Big(|\hat{a}(b)-\alpha(b)|+a_2|mes(A_\gamma)-\widehat{mes}(A_\gamma)|; \sum_{i=1}^m I(f(t_i)>\gamma)\leq\frac{p_f}{2}m|f\Big)\\
%  &\leq&\kappa_1 2m a_2\kappa^{-1}(t_f)e^{-\frac{1}{2}m\varepsilon_0 p_f}.
%\end{eqnarray*}
%The last inequality holds due to H\"older inequality.
%
%The other case is that $\sum_{i=1}^m I(f(t)>\gamma)>\frac{p_f}{2}m$, then as proved in the previous section, we have $\widehat{mes}(A_\gamma)>\lambda_fmes(A_\gamma),$ where $\lambda_f=\kappa_1^{-1}k(t_f)$.
Therefore,
\begin{eqnarray*}
E_b\Big(|J_2|^2 |f\Big)&\leq & \frac{4a_2^2}{\lambda_f^2 mes(A_\gamma)^2\zeta^{2d}mk^{2}(t_f)}\\
E_b\Big(|J_2||f\Big)&\leq & \frac{2a_2}{\lambda_f mes(A_\gamma)\zeta^d\sqrt{m}k(t_f)}\\
\end{eqnarray*}
%
%Thus, conditional $f$,%
%\[
%E(|J_{2}||f)\leq \frac{a_{2}I(M>b)}{\sqrt{m}\times k^{2}(t_{f})\times
%mes(A_{\gamma })}.
%\]%
According the proof in Section \ref{SecDis}, there exists a $\kappa >0$ such that%
\[
E(|J_{2}|)\leq \frac{\kappa }{\sqrt{m}}.
\]%
With a similar argument, we have that%
\[
E(J_{2}^{2})\leq \kappa .
\]
Summarizing the result for $J_1$ and $J_2$, we can choose $m=O(\max(\varepsilon^{-d(2/\alpha+1/\beta_1+\varepsilon_0)}, \varepsilon^{-2}))=O(\varepsilon^{-d(2/\alpha+2/\beta_1)})$, such that
\begin{equation*}
  E_b(\hat{Y}_b-v(b))\leq \varepsilon v(b), \qquad Var(\hat{Y}_b)=O(1).
\end{equation*}
\bibliographystyle{plain}
\bibliography{bibprob,bibstat,RefGrant,bibforcont}

\begin{thebibliography}{10}

\bibitem{Adl81}
R.J. Adler.
\newblock {\em The Geometry of Random Fields}.
\newblock Wiley, Chichester, U.K.; New York, U.S.A., 1981.

\bibitem{adler2012efficient}
R.J. Adler, J.H. Blanchet, and J.~Liu.
\newblock Efficient monte carlo for high excursions of gaussian random fields.
\newblock {\em The Annals of Applied Probability}, 22(3):1167--1214, 2012.

\bibitem{AST09}
R.J. Adler, G.~Samorodnitsky, and J.E. Taylor.
\newblock High level excursion set geometry for non-{G}aussian infinitely
  divisible random fields.
\newblock {\em preprint}, 2009.

\bibitem{AdlTay07}
R.J. Adler and J.E. Taylor.
\newblock {\em Random fields and geometry}.
\newblock Springer, 2007.

\bibitem{ASMGLY07}
S.~Asmussen and P.~Glynn.
\newblock {\em Stochastic Simulation: Algorithms and Analysis}.
\newblock Springer, New York, NY, USA, 2007.

\bibitem{AW08}
J.~M. Azais and M.~Wschebor.
\newblock A general expression for the distribution of the maximum of a
  {G}aussian field and the approximation of the tail.
\newblock {\em Stochastic Processes and Their Applications}, 118(7):1190--1218,
  2008.

\bibitem{AW09}
J.~M. Azais and M.~Wschebor.
\newblock {\em Level sets and extrema of random processes and fields}.
\newblock Wiley, Hoboken, N.J., 2009.

\bibitem{Berman85}
S.~M. Berman.
\newblock An asymptotic formula for the distribution of the maximum of a
  {G}aussian process with stationary increments.
\newblock {\em Journal of Applied Probability}, 22(2):454--460, 1985.

\bibitem{bingham1989regular}
N.H. Bingham, C.M. Goldie, and J.L. Teugels.
\newblock {\em Regular variation}, volume~27.
\newblock Cambridge university press, 1989.

\bibitem{Bor75}
C.~Borell.
\newblock The {B}runn-{M}inkowski inequality in {G}auss space.
\newblock {\em Inventiones Mathematicae}, 1975.

\bibitem{BOR}
C.~Borell.
\newblock The {B}runn-{M}inkowski inequality in {G}auss space.
\newblock {\em Invent. Math.}, 30:205--216, 1975.

\bibitem{Bor03}
C.~Borell.
\newblock The {Ehrhard} inequality.
\newblock {\em Comptes Rendus Mathematique}, 337(10):663--666, 2003.

\bibitem{BUC04}
J.~Bucklew.
\newblock {\em Introduction to Rare Event Simulation}.
\newblock Springer, New York, NY, USA, 2004.

\bibitem{dembo2009large}
A.~Dembo and O.~Zeitouni.
\newblock {\em Large deviations techniques and applications}, volume~38.
\newblock Springer, 2009.

\bibitem{dudley2010sample}
R.M. Dudley.
\newblock Sample functions of the gaussian process.
\newblock {\em Selected Works of RM Dudley}, pages 187--224, 2010.

\bibitem{JunSha06}
S.~Juneja and P.~Shahabuddin.
\newblock Rare event simulation techniques: An introduction and recent
  advances.
\newblock {\em Handbook on Simulation}, pages 291--350, 2006.

\bibitem{LS70}
H.~J. Landau and L.~A. Shepp.
\newblock Supremum of a {G}aussian process.
\newblock {\em Sankhya-the Indian Journal of Statistics Series A},
  32(Dec):369--378, 1970.

\bibitem{LT91}
M.~Ledoux and M.~Talagrand.
\newblock Probability in {B}anach spaces : isoperimetry and processes.
\newblock 1991.

\bibitem{Liu10}
J.~Liu.
\newblock Tail approximations of integrals of {G}aussian random fields.
\newblock {\em Annals of Probability}, to appear, 2011.

\bibitem{liu2012conditional}
J.~Liu and G.~Xu.
\newblock On the conditional distributions and the efficient simulations of
  exponential integrals of gaussian random fields.
\newblock {\em Arxiv preprint arXiv:1204.5546}, 2012.

\bibitem{LiuXu11}
J.~Liu and G.~Xu.
\newblock Some asymptotic results of {G}aussian random fields with varying mean
  functions and the associated processes.
\newblock {\em Annals of Statistics}, Accepted.

\bibitem{MS70}
M.~B. Marcus and L.~A. Shepp.
\newblock Continuity of {G}aussian processes.
\newblock {\em Transactions of the American Mathematical Society}, 151(2),
  1970.

\bibitem{MitzUpf05}
M.~Mitzenmacher and E.~Upfal.
\newblock {\em Probability and Computing: Randomized Algorithms and
  Probabilistic Analysis}.
\newblock Cambridge University Press, 2005.

\bibitem{Pit95}
V.~I. Piterbarg.
\newblock {\em Asymptotic methods in the theory of {G}aussian processes and
  fields}.
\newblock American Mathematical Society, Providence, R.I., 1996.

\bibitem{ST74}
V.N. Sudakov and B.S. Tsirelson.
\newblock Extremal properties of half spaces for spherically invariant
  measures.
\newblock {\em Zap. Nauchn. Sem. LOMI}, 45:75--82, 1974.

\bibitem{Sun93}
J.~Y. Sun.
\newblock Tail probabilities of the maxima of {G}aussian random-fields.
\newblock {\em Annals of Probability}, 21(1):34--71, 1993.

\bibitem{TA96}
M.~Talagrand.
\newblock Majorizing measures: The generic chaining.
\newblock {\em Annals of Probability}, 24(3):1049--1103, 1996.

\bibitem{TTA05}
J.~Taylor, A.~Takemura, and R.~J. Adler.
\newblock Validity of the expected {E}uler characteristic heuristic.
\newblock {\em Annals of Probability}, 33(4):1362--1396, 2005.

\bibitem{TayAdl03}
J.E. Taylor and R.J. Adler.
\newblock Euler characteristics for {G}aussian fields on manifolds.
\newblock {\em Annals of Probability}, 31(2):533--563, 2003.

\bibitem{TraubWW88}
J.~Traub, G.~Wasilokowski, and H.~Wozniakowski.
\newblock {\em Information-Based Complexity}.
\newblock Academic Press, New York, NY, 1988.

\bibitem{CIS}
B.S. Tsirelson, I.A. Ibragimov, and V.N. Sudakov.
\newblock Norms of {G}aussian sample functions.
\newblock {\em Proceedings of the Third Japan-USSR Symposium on Probability
  Theory (Tashkent, 1975)}, 550:20--41, 1976.

\bibitem{TsirelsonMaxDens}
V.~S. Tsirel'son.
\newblock The density of the maximum of a gaussian process.
\newblock {\em Theory of Probab. Appl.}, 20:847--856, 1975.

\bibitem{Wos96}
H.~Wozniakowski.
\newblock Computational complexity of continuous problems, 1996.
\newblock Technical Report.

\end{thebibliography}

\newpage

\appendix

\centerline{\huge \bf Supplemental Material}
\section{Proof of Theorem \ref{ThmCont} when $\sigma(t)$ is of Type 2 in Assumption A4}

In our proof for Type 2 standard deviation, we use similar methods as that for Type 1.
%From the definition of measure $Q$, expression \eqref{I1} and change of variable, we have
%\begin{eqnarray}
%  I_1&=&\int E\Big[\frac{1}{mes^2(A_\gamma)};M>b \Big | f(\tau) = x\Big] h_b(\tau )q_{b,\tau} (x)d\tau dx\notag\\
%  &\leq& \sup_{z>0} \int E\Big[\frac{1}{mes^2(A_\gamma)};M>b \Big | f(\tau) = \gamma+\frac{z}{b}\Big] h_b(\tau)d\tau\label{CondI1}
%\end{eqnarray}
%\begin{eqnarray}
%  I_2&=&\int E\Big[\frac{1}{mes(A_\gamma)};M>b \Big | f(\tau) = x\Big] h_b(\tau )q_{b,\tau} (x)d\tau dx\notag\\
%  &\geq& \int_1^2 q_{b,\tau}(\gamma+\frac{z}{b}) bdz \inf_{z\in(1,2)} \int E\Big[\frac{1}{mes(A_\gamma)};M>b \Big | f(\tau) = \gamma+\frac{z}{b}\Big]h_b(\tau)d\tau\notag\\
%  &\geq& \frac{1}{2}\Big(e^{-1/\sigma(\tau)^2}-e^{-2/\sigma(\tau)^2}\Big) \inf_{z\in(1,2)} \int E\Big[\frac{1}{mes(A_\gamma)};M>b \Big | f(\tau) = \gamma+\frac{z}{b}\Big]h_b(\tau)d\tau \label{CondI2}
%\end{eqnarray}
%
%\eqref{CondI1},\eqref{CondI2} shows that $I_1$ and $I_2$ term can be bounded by considering both the conditional first and second moment $E\Big[\frac{1}{mes(A_\gamma)^2};M>b \Big | f(\tau) = \gamma+\frac{z}{b}\Big]$, $E\Big[\frac{1}{mes(A_\gamma)};M>b \Big | f(\tau) = \gamma+\frac{z}{b}\Big]$, and the density function $h_b(\cdot)$.
We are going to establish similar results as in \eqref{tailbound} and Lemmas \ref{LemmaI21} and \ref{LemmaI22} hold for Gaussian random field with type 2 standard deviation.
To proceed, we provide some bounds on the distribution of $\tau$. The next lemma suggests that  $\tau$  is close to
$$t^* = \arg\sup_{t\in T} \sigma(t).$$

\begin{lemma}\label{Lemmah}
There exists constants $\delta, ~\varepsilon_0>0$  small enough and $\kappa>0$ large enough (but independent of $b$), such that for $x>\kappa$ the following bounds hold
\begin{itemize}
  \item[(i)] $\int_{|t-t^*|\leq \zeta_2^{-1}} h_b(t) dt>\varepsilon_0,$
  \item[(ii)] $\int_{\delta>|t-t^*|>x\zeta_2^{-1}} h_b(t)dt< \exp(-x^{\alpha_2/2})$,
  \item[(iii)]$\int_{|t-t^*|>\delta}h_b(t)dt<\exp(-\varepsilon_0b^{2})$.
\end{itemize}
\end{lemma}

To continue the analysis of $I_1$ and $I_2$, we discuss two different scenarios:
% according to the value of $\alpha$ :
\begin{enumerate}
\item
$\alpha_1>\alpha_2$, or $\alpha_1=\alpha_2$ and $\lim_{x\to 0}\frac{L_1(x)}{L_2(x)}\in\{0,1\}$; that is, as $x\to 0$, $L_1(x) x^{\alpha_1}\leq (1+o(1)) L_2 (x) x^{\alpha_2}$.

\item $\alpha_1<\alpha_2$, or $\alpha_1=\alpha_2$ and $\lim_{x\to 0}\frac{L_1(x)}{L_2(x)}=\infty$; that is, as $x\to 0$, $ L_2 (x) x^{\alpha_2} = o(1) L_1(x) x^{\alpha_1}$.
\end{enumerate}
The proof of this lemma is provided in the Supplemental Material B.

%\jc{In scenario 1, $\alpha=\alpha_2$, $L=L_2$, and $\zeta=\zeta_2$; and in scenario 2, $\alpha=\alpha_1$, $L=L_1$ and $\zeta=\zeta_1$.}

\subsection{Proof for scenario 1: $\alpha_1>\alpha_2$, or $\alpha_1=\alpha_2$ and $\lim_{x\to 0}\frac{L_1(x)}{L_2(x)}\in\{0,1\}$.}

For the proof of this scenario, the variation of $\sigma (t)$ is the dominating term. According to A2, there exists a constant $\Delta$ such that
\begin{equation}\label{corbd1}
1-r(s,t) \leq \Delta L_2(|s-t|) |s-t|^{\alpha_2}
\end{equation}
In addition, we can further replace the slowly varying function $L_1$ in \eqref{corbd} by $L_2$ and the inequality still holds, that is,
\begin{equation}\label{A2'}
  |r(t,t+s_1)-r(t,t+s_2)|\leq\kappa_r \max(L_2(|s_1|)|s_1|^{\beta_0},L_2(|s_2|)|s_2|^{\beta_0})|s_1-s_2|^{\beta_1}.
\end{equation}
For the proof of this scenario, we work under the above two inequalities instead of A2.
%
%Since $\alpha_1 >\alpha_2$, according to condition A2, we can pick $\beta_0'\geq 0, \beta_1', \kappa_r'>0$, satisfying $\beta_0'+\beta_1'\geq\alpha_2$ such that
%\begin{equation}\label{A2'}
%  |r(t,t+s_1)-r(t,t+s_2)|\leq\kappa_r' \max(L_2(|s_1|)|s_1|^{\beta_0'},L_2(|s_2|)|s_2|^{\beta_0'})|s_1-s_2|^{\beta_1'}.
%\end{equation}
%Therefore, if we can prove Theorem \ref{ThmCont} for the case where $\alpha_1=\alpha_2, \lim_{x\to 0}\frac{L_1(x)}{L_2(x)}=1$ under Conditions A3, A4, A5, and \eqref{A2'}, then we prove Theorem \ref{ThmCont} in scenario 1. In the rest of the proof, we assume $\alpha_1=\alpha_2, \lim_{x\to 0}\frac{L_1(x)}{L_2(x)}=1$, and use $\kappa_r,\beta_0$, and $\beta_1$ in \eqref{A2'} instead of $\kappa_r', \beta_0'$ and $\beta_1'$ for convenience.
The proof follows a similar idea as that of the constant variance case by providing bounds for $I_1$ and $I_2$.

\paragraph{The $I_1$ term.}

For a given $\tau$ and $z$, we adopt a similar conditional representation as in \eqref{decomp}.
We start with establishing similar results as in Lemma \ref{LemmaMu}. Since $L_1(x) x^{\alpha_1}\leq (1+o(1)) L_2 (x) x^{\alpha_2}$, we can replace $\alpha_1$ and $L_1$ in the statement of Lemma \ref{LemmaMu} by $\alpha_2$ and $L_2$ and the statement still holds.
Now we proceed to prove \eqref{tailbound}.%, Lemmas \ref{LemmaI21} and \ref{LemmaI22}.
  According to the expression \eqref{expQ}, we proceed by deriving an upper bound of
\begin{equation}\label{tailint1}
 \int_{ T}P\left(\frac{1}{mes(A_\gamma)}>y^{-d}\zeta_2^{-1}, M>b|f(\tau)=\gamma+\frac{z}{b}\right)h_b(\tau)\frac{q_{b,\tau}(\gamma+{z}/{b})}{b}d\tau dz.\quad \mbox{for $y$ small enough.}
\end{equation}
We discuss two situation: $z>1$ and $0<z\leq 1$.

\textbf{Situation 1: $z>1$.}
From condition A2, A4, A5, \eqref{A2'} and Lemma \ref{LemmaSlow}(i), for $|t|<c_dy\zeta_2^{-1}$,  we have that
\begin{eqnarray*}
  |\mu_\tau(t)-\mu_\tau(0)|
    \leq \kappa_\mu\sqrt{L_2(|t|)}|t|^{\alpha_2/2}+\kappa b L_2(|t|)|t|^{\alpha_2}
  =O(y^{\alpha_2/4}b^{-1})
\end{eqnarray*}
Note that $\mu_\tau(0) = \gamma + z/b > \gamma  + 1/b$. Thus, by picking $y_0$ small enough, we have that
$$\mu_\tau(t)\geq \gamma+  \frac{1}{2b} \qquad \mbox{for $|t|\leq c_d y\zeta_2^{-1}$.}$$
With a similar development as in \eqref{eqI11} and the conditional variance calculation for $f_0(t)$ as in \eqref{eqI13}, that is,
\begin{eqnarray*}
  C_0(t,t) =O(y^{\alpha_2/2}b^{-2}),
\end{eqnarray*}
 we conclude that for some small $\varepsilon_0>0$
\begin{equation*}
Q\Big(\frac{1}{mes(A_\gamma)}>y^{-d}\zeta_2^{-1}, M>b\Big)\leq
P\Big(\inf_{|t|\leq c_d y\zeta_2^{-1}}(f_0(t)+\mu_\tau(t))\leq \gamma\Big)
\leq P(\inf_{|t|\leq c_d y\zeta_2^{-1}} |f_0(t)| > \frac 1{2b})\leq \exp(-y^{-\varepsilon_0}).
\end{equation*}

%Under the current scenario, we have that $\zeta = \zeta_2$. According to Lemma \ref{LemmaMu}(iii), we have\\ $E(\sup_{|t|\leq c_d y_0\zeta_1^{-1}}f_0(t))\leq \frac{1}{4b}$ for $y < y_0$.
%As a result of the Borel-TIS lemma, there exists $\varepsilon_0>0$ small enough, the tail probability in \eqref{tailint1} can be bounded by
%\begin{equation*}
%\leq \exp(-y^{-\varepsilon_0})\qquad \mbox{for $0<y<y_0$}.
%\end{equation*}

\subparagraph{Situation 2: $0<z\leq 1$.}

For $0<z<1$, we choose $\delta_0,\delta_1$ to be  small enough and $\lambda$ to be large enough and develop   bounds for the above probability under four cases (same as in the proof of constant variance case):
\begin{itemize}
\item []Case 1. $t\in C_1 \triangleq \{t: 0<|t -\tau|< y^{-\delta_0}\zeta_2^{-1}\}$,
\item []Case 2. $t\in C_2 \triangleq \{t: y^{-\delta_0}\zeta_2^{-1}<|t -\tau|<\delta_1\}$,
\item []Case 3. $t\in C_3 \triangleq \{t: |t -\tau|\geq\delta_1\}$ and $y<b^{-\lambda}$,
\item []Case 4. $t\in C_3$ and $y\geq b^{-\lambda}$.
\end{itemize}

With these notation, we have the following bound
\begin{eqnarray*}
&&Q\Big(\frac{1}{mes(A_\gamma)}>y^{-d}\zeta_2^{-1}, M>b\Big)\\
  & &~~~~~~~\leq \sum_{i=1}^3  \int_{ T}P\left(\frac{1}{mes(A_\gamma)}>y^{-d}\zeta^{d},\sup_{t\in C_i}f(t)>b|f(\tau)=\gamma+\frac{z}{b}\right)h_b(\tau)\frac{q_{b,\tau}(\gamma+{z}/{b})}{b} d\tau dz.
\end{eqnarray*}
With the same argument for \eqref{eqI16}, each of the summands on the right-hand-side is further bounded by
\begin{eqnarray}\label{tailboundsum}
  &&\int_{ T}P\left(\frac{1}{mes(A_\gamma)}>y^{-d}\zeta_2^{d},\sup_{t\in C_i}f(t)>b\Big |f(\tau)=\gamma+\frac{z}{b}\right)h_b(\tau)\frac{q_{b,\tau}(\gamma+{z}/{b})}{b}d\tau dz\\
  &&~~~~~~\leq\int_{ T}P\Big(\sup_{t\in C_i,|s-t|\leq c_d y\zeta_2^{-1}}|f(t)-f(s)|>\frac{1}{b},\sup_{t\in C_i}f(t)>b\Big|f(\tau)=\gamma+\frac{z}{b}\Big)h_b(\tau)\frac{q_{b,\tau}(\gamma+{z}/{b})}{b}d\tau dz.\nonumber
\end{eqnarray}
Similarly, we define
$$x_i \triangleq  \zeta_2\times |t_i -\tau|.$$

\subparagraph{Case 1: $0<|t-\tau|<y^{-\delta_0}\zeta_2^{-1}.$}
We adopt the same lattice and cover sets, $\tilde{T}$, and $B_i$, defined on page \pageref{latticeT} for the proof of the constant variance case, with $\zeta_1$ replaced by $\zeta_2$.
For this case, we bound the right-hand-side of \eqref{tailboundsum} by
$$\sum_{B_i \cap C_1 \neq \emptyset}\int_{ T}P\Big(\sup_{t\in B_i,|s-t|\leq c_d y\zeta_2^{-1}}|f(t)-f(s)|>\frac{1}{b}\Big|f(\tau)=\gamma+\frac{z}{b}\Big)h_b(\tau)\frac{q_{b,\tau}(\gamma+{z}/{b})}{b}d\tau dz$$
and take advantage of the conditional representation $f(t) = \mu_\tau(t) + f_0(t).$ We proceed to investigating the variation of $\mu_\tau(t)$ and $f_0(t)$. For $f_0(t)$ and $|s-t|\leq c_d y\zeta_2^{-1}$, with the same argument as in \eqref{condvar}, we have that $Var(f_0(t) - f_0(s)) \leq \kappa y^{\alpha_2/2} b^{-2}$. For the conditional mean, by means of the representation \eqref{Cond},
\begin{eqnarray*}
  |\mu_\tau(t)-\mu_\tau(s)|
  &\leq& \kappa\zeta_2^{-{\alpha_2}/2}\sqrt{L_2(y/\zeta_2)}y^{{\alpha_2}/2}
  +
  \kappa b L_2(c_dy/\zeta_2) y^{\alpha_2}\zeta_2^{-{\alpha_2}}
  +\kappa (x_i+1)^{\beta_0} b L_2((x_i+1)/\zeta_2)y^{\beta_1}\zeta_2^{-{\alpha_2}}\\
  &\leq & \kappa b^{-1}  y^{\varepsilon_0}
\end{eqnarray*}
for some small positive constant $\varepsilon_0$.
Now we pick $y_0$ small enough. For $0<y<y_0$ and   $|\mu_\tau(t)-\mu_\tau(s)|< \frac 1 {2b}$, together with the variance control of $f_0(t) - f_0(s)$, we have that
\begin{eqnarray*}
&&\int_{ T}P\Big(\sup_{t\in B_i,|s-t|\leq c_d y\zeta_2^{-1}}|f(t)-f(s)|>\frac{1}{b}\Big|f(\tau)=\gamma+\frac{z}{b}\Big)h_b(\tau)d\tau\\
&&~~~~~~~\leq\int_{ T}P\Big(\sup_{t\in B_i,|s-t|\leq c_d y\zeta_2^{-1}}|f_0(t)-f_0(s)|>\frac{1}{2b}\Big|f(\tau)=\gamma+\frac{z}{b}\Big)h_b(\tau)d\tau\\
&&~~~~~~~\leq\exp(-y^{-\varepsilon_0})
\end{eqnarray*}
for some $\varepsilon_0>0$. We sum up all the $B_i$'s such that $0<|t_i-\tau|<y^{-\delta_0}\zeta_2^{-1}$ and obtain that
  \begin{eqnarray*}
    P\Big(\frac{1}{mes(A_\gamma)}>y^{-d}\zeta_1^{d},\sup_{t\in C_1}f(t)>b|f(\tau)=\gamma+\frac{z}{b}\Big)\leq\exp(-y^{-\varepsilon_0})
  \end{eqnarray*}
  for which we may need to choose a smaller $\varepsilon_0$.

\subparagraph{Case 2: $y^{-\delta_0}\zeta_2^{-1}<|t-\tau|<\delta_1$.}

We split \eqref{tailboundsum} as follows
\begin{eqnarray}  \label{tailboundc2}
  \eqref{tailboundsum}&\leq& \sum_{B_i \cap C_2 \neq \emptyset} \int_{|\tau-t^*|\leq\frac{1}{3}y^{-\delta_0}\zeta_2^{-1}}P\Big(\sup_{t\in B_i}f(t)>b \Big|f(\tau)=\gamma+\frac{z}{b}\Big) h_b(\tau)   d\tau\notag\\
  &&+~~~~\int_{|\tau-t^*|>\frac{1}{3}y^{-\delta_0}\zeta_2^{-1}}h_b(\tau)d\tau.
\end{eqnarray}
For this case, we implicitly requires that $y^{-\delta_0} < \delta_1 \zeta_2$. Thus, Lemma \ref{Lemmah} (ii) and (iii) provide an upper bound of the second term in the above display
\begin{equation*}
  \int_{|\tau-t^*|>\frac{1}{3}y^{-\delta_0}\zeta_2^{-1}}h_b(\tau)d\tau\leq \exp(-y^{-\varepsilon_0})
\end{equation*}
for $\varepsilon_0$ and $y$ sufficiently small and $y^{-\delta_0}<\delta_1\zeta_2$.

For the first term on the right-hand-side of \eqref{tailboundc2}, we bound it in a similar way as in constant variance case. In particular, each summand is bounded by
$$\sup_{|\tau-t^*|\leq\frac{1}{3}y^{-\delta_0}\zeta_2^{-1}}P(\sup_{t\in B_i}f(t)>b|f(\tau)=\gamma+\frac{z}{b}) $$
For $y^{-\delta_0}\zeta_2^{-1}<|t-\tau|<\delta_1$ and $|\tau-t^*|\leq\frac{1}{3}y^{-\delta_0}\zeta_2^{-1}$ we have that $|t-t^*|>\frac 2 3 y^{-\delta_0}\zeta_2^{-1}.$
Using the expansion $\sigma(t^*) - \sigma (t) \sim \Lambda L_2(|t-t^*|) |t-t^*|^{\alpha_2}$,
we have that
\begin{equation}\label{vratio}
\frac{\sigma(t)}{\sigma(\tau)}\leq 1-\varepsilon_0  \frac{L_2(x_i\zeta_2^{-1})}{L_2(\zeta_2^{-1})}\frac{( \zeta_2 |t_i -\tau|)^{\alpha_2}}{b^{2}}, ~~~\mbox{for some small $\varepsilon_0>0$ and .}
\end{equation}
%The last inequality holds because for $|\tau-t^*|\leq 1/3 y^{-\delta_0}\zeta_2^{-1}$, from the local expansion of $\sigma(\tau)$ around $t^*$, $\sigma(\tau)\geq \sigma(t^*)-2\Lambda L_2(1/3y^{-\delta_0}\zeta_2^{-1})1/3^{\alpha_2} y^{-\delta_0{\alpha_2}}\zeta_2^{-{\alpha_2}}$; while $\sigma(t)\leq \sigma(t^*)-1/2\Lambda L_2(x_i\zeta_2^{-1})x_i^{\alpha_2}\zeta_2^{\alpha_2}$.
From the expression of \eqref{Cond} and the inequality \eqref{vratio}, for $t\in B_i\cap C_2 \neq \emptyset$ and $x_i = \zeta_2 |t_i -\tau|$, we have that
\begin{equation*}
\mu_\tau(t)\leq b+ \kappa \sqrt{\frac{L_2(x_i\zeta_2^{-1})}{L_2(\zeta_2^{-1})}}\frac{x_i^{{\alpha_2}/2}} b-\varepsilon_0  \frac{L_2(x_i\zeta_2^{-1})}{L_2(\zeta_2^{-1})} \frac{x_i^{\alpha_2}}b
\leq b-\frac{\varepsilon_0}{2} x_i^{\alpha_2}\frac{L_2(x_i\zeta_2^{-1})}{L_2(\zeta_2^{-1})}b^{-1}.
\end{equation*}
Furthermore, Lemma \ref{LemmaMu}(i) implies that
\begin{equation}
Var(f_0(t)) = C_0(t,t)\leq 2\lambda_2 \frac{L_2(x_i\zeta_1^{-1})}{L_2(\zeta_1^{-1})}x_i^{\alpha_2} b^{-2}.
\end{equation}
Thus, the Borel-TIS inequality suggests that
\begin{equation*}
\sup_{|\tau-t^*|\leq\frac{1}{3}y^{-\delta_0}\zeta_2^{-1}}P(\sup_{t\in B_i}f(t)>b|f(\tau)=\gamma+\frac{z}{b})\leq \exp(-x_i^{-\varepsilon_0}),
\end{equation*}
for some small constant $\varepsilon_0$.

Combining the upper bound for the two term on the right side of \eqref{tailboundc2}, and putting together all  $B_i$'s such that $y^{-\delta_0}<x_i<\delta_1$, we have that
\begin{eqnarray*}
  &&\int_T P\Big(\frac{1}{mes(A_\gamma)}>y^{-d}\zeta_2^d, \sup_{t\in C_2} f(t)>b|f(\tau)=\gamma+\frac{z}{b}\Big)h_b(\tau)\frac{q_{b,\tau}(\gamma+{z}/{b})}{b}d\tau dz\\
  &\leq&\exp(-y^{-\varepsilon_0})+\sum_{k=0}^{\infty}\kappa(y^{-\delta_0}+k)^{d-1}\exp(-(y^{-\delta_0}+k)^{\varepsilon_0})\\
  &\leq& \exp(-y^{-\varepsilon_0/2})
\end{eqnarray*}
for some large constant $\kappa>0$ and possible a different choice of $\varepsilon_0$.

\subparagraph{Case 3: $|t-\tau|\geq \delta_1$ and $y<b^{-\lambda}$.} The analysis is completely analogous to the Case 3 on Page \pageref{mu3}. The only difference is that the variance function $\sigma^2(t)$ is non-constant. Given that $\sigma(t)$ is H\"older continuous, all the calculations remain. Therefore, we omit the details and directly reach the bound that %Since under \eqref{A2'}, A4, and A6, $\mu(t)$, $r(s,t)$, and $\sigma(t)$ are all H\"older continuous, we can pick $\lambda$ large enough to make \eqref{mu3} holds. The rest of proof for case 3 in proof of \ref{ThmCont} still holds due to Lemma \ref{LemmaMu}, so we can arrive at the same conclusion as in the previous proof case 3.
%\begin{equation*}
%  P\left(\frac{1}{mes(A_\gamma)}>y^{-d}\zeta^d,\sup_{|t-\tau|>\delta_1}f(t)>b|f(\tau)=\gamma+\frac{z}{b}\right)\leq \exp(-y^{-\varepsilon_0}),
%\end{equation*}
%for some $\varepsilon_0>0.$ Therefore, we have
\begin{equation*}
  \int_T P\Big(\frac{1}{mes(A_\gamma)}>y^{-d}\zeta^d,\sup_{|t-\tau|>\delta_1}f(t)>b|f(\tau)=\gamma+\frac{z}{b}\Big)h_b(\tau)\frac{q_{b,\tau}(\gamma+{z}/{b})}{b}d\tau dz\leq \exp(-y^{-\varepsilon_0})
\end{equation*}
for all $y < b^{-\lambda}$.

\subparagraph{Case 4: $|t-\tau|\geq\delta_1, y\geq b^{-\lambda}$.}\label{case4s1}
We split the bound \eqref{tailboundsum} into two parts.
\begin{eqnarray}
  &&\int_{T}P\Big(\frac{1}{mes(A_\gamma)}>y^{-d}\zeta^d,\sup_{|t-\tau|>\delta_1}f(t)>b|f(\tau)=\gamma+\frac{z}{b}\Big)h_b(\tau)d\tau\notag\\
  &\leq&\sup_{ |\tau-t^*|\leq \delta_1/3}P\Big(\frac{1}{mes(A_\gamma)}>y^{-d}\zeta^d,\sup_{|t-\tau|>\delta_1}f(t)>b|f(\tau)=\gamma+\frac{z}{b}\Big)\notag\\
  && ~~~ +\int_{|\tau-t^*|>\delta_1/3}h_b(\tau)d\tau.
  \label{tailboundc3}
\end{eqnarray}
From Lemma \ref{Lemmah} (iii), the second term on the right side of last inequality can be bound by $\exp(-b^{\varepsilon_0})$ for some $\varepsilon_0>0$. Note that in Case 4, $y>b^{-\lambda}$, so this expression can be further bounded by
\begin{equation*}
   \int_{|\tau-t^*|>\delta_1/3}h_b(\tau)d\tau\leq \exp(-\varepsilon_0b^{2})\leq\exp(-y^{-\varepsilon_0/\lambda}).
\end{equation*}

Now we consider the first term on the right side of \eqref{tailboundc3}.
On the set $|\tau -t^*| < \delta_1/3$ and $|t-\tau|> \delta_1$, there exists some $\varepsilon_0$ such that the conditional mean can be bounded from below by
\begin{equation}\label{mu4}
 \mu_\tau(t)\leq (1-\frac{\varepsilon_0}{2})b.
\end{equation}
This is because from condition A4,  $\sigma(\tau)\geq\sigma(t^*)-\Lambda L(\delta_1/3)(\delta_1/3)^\alpha$, for $|\tau-t^*|\leq1/3\delta_1$; while $\sigma(t)\leq \sigma(t^*)-\Lambda L(2\delta_1/3)(2\delta_1/3)^{\alpha_2}$, for $|t-t^*|\geq 2\delta_1/3$. As a result, there exists a constant $\varepsilon_0>0$ such that $\frac{\sigma(t)}{\sigma(\tau)}\leq 1-\varepsilon_0$. In addition, the correlation function also drops.

For the rest of case 4, we follow the same analysis as that of Case 4 on page \pageref{case4} and derive an upper bound for the first term on the right side of \eqref{tailboundc3}.
\begin{eqnarray}\label{c4}
P\Big(\frac{1}{mes(A_\gamma)}>y^{-d}\zeta^d,\sup_{|t-\tau|\geq\delta_1}f(t)>b|f(\tau)=\gamma+\frac{z}{b}\Big)&\leq &
P\Big(\sup_{|t-\tau|\geq\delta_1}f(t)>b|f(\tau)=\gamma+\frac{z}{b}\Big)\notag\\
&\leq &
P\Big(\sup_{|t-\tau|\geq\delta_1}f_0(t) + \mu_\tau(t)>b\Big)\notag\\
&\leq &
P\Big(\sup_{|t-\tau|\geq\delta_1}f_0(t) >\varepsilon_0 b/2\Big)\notag\\
&\leq& \exp(-\frac{\varepsilon_0^2}{8\sigma_T^2} b^{2})\\
&\leq& \exp(-y^{-\varepsilon_0'}).\notag
\end{eqnarray}
for some $\varepsilon_0,\varepsilon_0'>0$.
Combining our result for the first and second term of \eqref{tailboundc3}, and for $C_i = C_3$ for $y\geq b^{-\lambda }$
\begin{equation*}
  \eqref{tailboundsum}\leq \exp(-y^{-\varepsilon_0}),\qquad \text{for some possibly smaller }\varepsilon_0>0.
\end{equation*}
%and further
%  \begin{equation}
%    Q\Big(\frac{1}{mes(A_\gamma)}>y^{-d}\zeta_1^d,M>b\Big)\leq \exp(-y^{-\varepsilon_0}).
%  \end{equation}

\paragraph{Summary of the analysis for $I_1$.} Putting all the results in Cases 1-4 together, we have that there exists a $y_0>0$ such that
\begin{equation}\label{mesub1}
Q\Big(\frac{1}{mes(A_\gamma)}>y^{-d}\zeta_2^{-1}, M>b\Big)\leq \exp(-y^{-\varepsilon_0}),
\end{equation}
for $0<y<y_0$.
Thus, for some $\kappa>0$, we have
\begin{equation*}
 I_1 = E^Q\Big(\frac{1}{mes(A_\gamma)^2};M>b\Big)\leq (\kappa+y_0^{-d}){\zeta_2}^{2d}.
\end{equation*}

\paragraph{The $I_2$ term.}
We are going to derive a lower bound for $I_2$ by showing that Lemma \ref{LemmaI21} and Lemma \ref{LemmaI22} are valid.
Following the same calculation for \eqref{tailboundc3}, we reach the result of Lemma \ref{LemmaI21} (on page \pageref{LemmaI21}) that
$$    Q(\sup_{|t-\tau|\geq\delta'}f(t)\geq\gamma) \leq
Q(|t^*-\tau|\geq\delta'/3) + Q(\sup_{|t-\tau|\geq\delta'}f(t)\geq\gamma, |t^*-\tau|< \delta'/3)
$$
The first term on the right-hand-side is controlled by Lemma \ref{Lemmah} (iii). The second term can be bounded by a similar analysis as in \eqref{c4}. Thus, we have that
\begin{equation}\label{L2}
Q(\sup_{|t-\tau|\geq\delta'}f(t)\geq\gamma) \leq e^{-\varepsilon_0 b^2}
\end{equation}
for some $\varepsilon_0$ small.

Now, we proceed to proving a similar result as in Lemma \ref{LemmaI22} (page \pageref{LemmaI22}).
Note that for $x{\zeta_2}^{-1} < \delta'$
\begin{eqnarray*}
Q\Big(\sup_{x{\zeta_2}^{-1}\leq|t-\tau|\leq\delta'} f(t)\geq\gamma \Big)&\leq& Q\Big(\sup_{x{\zeta_2}^{-1}\leq|t-\tau|\leq\delta'} f(t)\geq\gamma , |\tau - t^*| < x \zeta_2 ^{-1}/3\Big)\\
&& ~~~~+Q\Big( |\tau - t^*| > x \zeta_2 ^{-1}/3\Big).
\end{eqnarray*}
Thanks to Lemma \ref{Lemmah}, the second term on the right-hand-side is bounded by $e^{-x^{\varepsilon_0}}$.
For the first term, we follow a similar analysis as in Lemma \ref{LemmaI22}.
In particular, we can establish a bound for the conditional mean $\mu_\tau(t) = E(f(\tau+ t) | \tau,z)$ in the following form
$$\mu_\tau(t) \leq \gamma + \frac z b - \varepsilon_0\frac {x^{\alpha_2}}{b}$$
for all $x\zeta_2^{-1} <|t|< \delta'$ and $|\tau - t^*| < x\zeta_2^{-1}/3$. With this bound, we follow exactly the same analysis as in Lemma \ref{LemmaI22} and obtain that
\begin{equation}\label{L3}
  Q\Big(\sup_{x{\zeta_2}^{-1}\leq|t-\tau|\leq\delta'} f(t)\geq\gamma \Big)\leq e^{-x^{\alpha_2/4}}
\end{equation}
and thus a similar result in Lemma \ref{LemmaI22} has been proved.
With these results, we use the same analysis as that in \eqref{e:39} and obtain that for some $x$ sufficiently large
\begin{eqnarray*}
  I_2&\geq& \varepsilon_0  x^{-d}\zeta_2^d.
\end{eqnarray*}
Combining our upper bound for $I_1$  and  lower bound for  $I_2$, we conclude the proof for scenario 1.
\begin{equation*}
  \sup_{b}\frac{E^QZ_b^2}{P^2(M>b)}=\sup_b\frac{I_1}{I_2^2}<\infty.
\end{equation*}

\subsection{Proof for scenario 2: $\alpha_1<\alpha_2$, or $\alpha_1=\alpha_2$ and $\lim_{x\to 0}\frac{L_1(x)}{L_2(x)}=\infty$}

In scenario 2, we first consider the covariance function $C(s,t)=cov(f(s),f(t))$. It satisfies the following conditions:

\begin{enumerate}
\item[B1] There exists $\beta_0\geq 0$, $\beta_1>0$, such that $\beta_0+\beta_1\geq \alpha_1$, and
\begin{equation*}
  |C(\tau,t+s_1)-C(\tau,t+s_2)|\leq \kappa \max(L_2(|s_1|)|s_1|^{\beta_0},L_2(|s_2|)|s_2|^{\beta_0})|s_1-s_2|^{\beta_1}
\end{equation*}
\item[B2] As $|t-s|\to 0$,
\begin{equation*}
  C(s,s)-C(s,t)\sim \sigma(s)^2 \Delta_s L_1(|s-t|)|s-t|^{\alpha_1}
\end{equation*}
\item[B3] There exists $\varepsilon'', \delta''>0$ such that for $|s-t^*|<\delta'', |t-s|>2\delta''$, we have
\begin{equation*}
  C(s,s)-C(s,t)>\varepsilon''.
\end{equation*}
\end{enumerate}
Therefore, we can basically replicate the analysis in Section \ref{SecCont} for the constant mean by replacing the correlation function $r(s,t)$ with the covariance function $C(s,t)$ and all the derivations are exactly the same except for one place. In the analysis of Case 4 (Page \pageref{case4}), for which we need to provide a bound for
$$Q\Big(mes(A_\gamma)^{-1}>y^{-d}\zeta_1^d, \sup_{|t-\tau|>\delta_1} f(t) >b\Big).$$
For this part, we need to following the analysis of Case 4 for scenario 1 (page \pageref{case4s1}). Other analyses are all the same and therefore are omitted.

\section{Proof of Lemmas}
%\begin{proposition}\label{PropDudley}
%Suppose $(f(t))_{t\in \mathcal{U}}$ is a continuous centered Gaussian random field over a compact set $\mathcal{U}\subset\mathbb{R}^d$ satisfying condition A5. Define the metric on $\mathcal{U}$:
%$$D(s,t)=\sqrt{E[f(s)-f(t)]^2},$$ and let $\text{diam}(\mathcal{U})=\sup_{t,s\in\mathcal{U}}D(s,t)$, then there exists a universal constant $\kappa>0$ s.t. $$E[\sup_{t\in\mathcal{U}}f(t)]\leq\kappa\int^{\text{diam}/2}_0(\log(\mathcal{N}(\epsilon)))^{1/2}d\epsilon,$$
%where $\mathcal{N}(\epsilon)$ is the smallest number of d-balls of radius $\epsilon$ whose union covers $\mathcal{U}$.
%\end{proposition}
Throughout the proof, we used several properties of slowly varying function, which are stated in the next Lemma.
\begin{lemma}\label{LemmaSlow}
Suppose $L(x), x>0$ is a positive continuous slowly varying function, then it has the following properties.
\begin{enumerate}
%\item[(i)] $\forall a_1, a_2$ s.t. $  0<a_1,a_2, t\in[a_1,a_2]$, $\frac{L(tx)}{L(x)}$ is uniformly convergent in $t$ to 1 as $x \rightarrow 0+$.
%\item[(ii)] $\forall \beta>0$ \begin{equation*}\lim_{x\rightarrow 0+}L(x)x^\beta=0.\end{equation*}
\item[(i)]$\forall\beta>0, \exists \delta_\beta>0, \kappa_s$, s.t. for $\zeta$ satisfying $\zeta^{-1}<\delta_\beta$, $x\leq1$ we have
\begin{equation*}
  \frac{L(x\zeta^{-1})}{L(\zeta^{-1})}x^\beta\leq \kappa_s
\end{equation*}
\item[(ii)] $\forall \beta>0, \exists \delta_\beta>0, \kappa_s>0$, s.t. for $\zeta$ satisfying $\zeta^{-1}x<\delta_\beta, x\geq1$, we have
\begin{equation*}\frac{L(\zeta^{-1}x)x^\beta}{L(\zeta^{-1})}\geq \kappa_s^{-1}\end{equation*}
\end{enumerate}
\end{lemma}
This lemma is a direct application of Theorem 1.5.3, and Theorem 1.5.4 in  \cite{bingham1989regular}.

\begin{proof}[Proof of Lemma \ref{LemmaI21}]
For $|t-\tau|\geq \delta$, according to condition A3, there exits $\varepsilon>0$, such that $r(t,\tau)<1-\varepsilon$. For $b$ large enough, and $0<z<\frac{\varepsilon}{4}b^2$, we have
\begin{eqnarray*}
\mu_\tau(t)&=&\mu(t+\tau)+\frac{r(t+\tau,\tau)}{r(\tau,\tau)}(\gamma+\frac{z}{b}-\mu(\tau))\\
&\leq&2\mu_T+(1-\varepsilon)(\gamma+\frac{z}{b})\\
&\leq&(1-\varepsilon/2)b
\end{eqnarray*}
and the conditional variance $C_0(t,t)=C(t+\tau,t+\tau)-C(t+\tau,\tau)^2 C(\tau,\tau)^{-1}$ is bounded by $\sigma_T^2$.
Then by the Borel-TIS inequality (Proposition \ref{PropBorel}), we have that
\begin{eqnarray}\label{condbound1}
%\lim_{b\to\infty}
P(\sup_{|t-\tau|\geq\delta}f(t)\geq\gamma|f(\tau)=\gamma+\frac{z}{b})\leq e^{-\frac{\varepsilon^2}{8\sigma_T^2} b^2}
\end{eqnarray}
Since $z$ is asymptotically exponentially distributed with mean $\sigma(\tau)^2$ and $\tau$ is asymptotically uniformly distributed, we have
\begin{equation*}
  Q(\sup_{|t-\tau|>\delta}f(t)\geq\gamma)\leq \sup_{z<\frac{\varepsilon b^2}{4}}P(\sup_{|t-\tau|\geq\delta}f(t)\geq\gamma|f(\tau)=\gamma+\frac{z}{b})+Q(z>{\varepsilon b^2}/{4})\leq e^{-\varepsilon_0 b^2}.
\end{equation*}
\end{proof}

\bigskip

\begin{proof}[Proof of Lemma \ref{LemmaI22}]
According to conditional Gaussian calculation, we have that $$Q(b\times(f(\tau) - \gamma)\geq   x^{\alpha_1/2})\leq e^{-\varepsilon_0 x^{\alpha_1/2}}.$$ Therefore, we only need to consider that $f(\tau) = \gamma + \frac z b $ for $z< x^{\alpha_1/2}$.
Let $\tilde{T}=\{t_1,...,t_N\}$ such that:
\begin{enumerate}
  \item For $i\neq j$, $i,j\in\{1,...,N\}$, $|t_i-t_j|>\zeta_1^{-1}$
  \item For any $t\in T$, there exists $i\in\{1,...,N\}$, such that $|t-t_i|\leq 2\zeta_1^{-1}$.
\end{enumerate}
  Furthermore, let $B_i=\{t:|t-t_i|\leq 2\zeta_1^{-1}\}$, $i\in\{1,2,...,N\}$.
  First calculate the upper bound for conditional mean and variance.
  For $k/\zeta_1\leq |t_i - \tau | \leq  (k+1)/\zeta_1$, $t\in B_i$, and $z<x^{\alpha_1/2}$ according to condition A2 and A5, we have that
%  For $b, k$ large enough, $\delta>0$ small enough, by Lemma \ref{LemmaMu}, we have
  \begin{eqnarray}
    \mu_\tau(t)&\leq &b + \frac z b +\kappa_\mu\sqrt{L_1(|t|)}|t|^{\alpha_1/2}-\Delta_\tau b L_1(|t|)|t|^{\alpha_1}\nonumber\\
  %&\leq& \lambda_0\sqrt{L_1(|t|)}|t|^{\alpha_1/2}+\lambda_1 L_1(|t|)|t|^\alpha_1+(1-\lambda_2L_1(|t|)|t|^\alpha_1)(\gamma+\frac{z}{b})\nonumber\\
  %&\leq& \lambda_0\sqrt{L_1((k+3)\zeta_1^{-1})}(k+3)^{\alpha_1/2}\zeta_1^{-\alpha_1/2}+\lambda_1 L_1((k+3)\zeta_1^{-1})(k+3)^\alpha_1\zeta_1^{-\alpha_1}\\
  %&&+(1-\lambda_2L_1((k-2)\zeta_1^{-1})(k-2)^\alpha_1\zeta_1^{-\alpha_1})(\gamma+\frac{z}{b})\nonumber\\
  %&=& \lambda_0 (\frac{L_1((k+3)\zeta_1^{-1})}{L_1(\zeta_1^{-1})})^{1/2}(k+3)^{\alpha_1/2} b^{-1}+\lambda_1 \frac{L_1((k+3)\zeta_1^{-1})}{L_1(\zeta_1^{-1})}(k+3)^\alpha_1 b^{-2}\\
  %&&+(1-\lambda_2\frac{L_1((k-2)\zeta_1^{-1})}{L_1(\zeta_1^{-1})}(k-2)^\alpha_1 b^{-2})(\gamma+\frac{z}{b})\nonumber\\
%  &\leq&b+\lambda_0\sqrt{L_1(|t|)}|t|^{\alpha_1/2}-\lambda_2L_1(|t|)|t|^\alpha_1\nonumber\\
  &\leq&b-\frac{\Delta_\tau}{2}\frac{L_1(k\zeta_1^{-1})}{L_1(\zeta_1^{-1})}k^{\alpha_1} b^{-1}.\label{eqI21}
  \end{eqnarray}
 % The last inequality holds because by lemma \ref{LemmaSlow}, $\forall \varepsilon>0, \exists \delta>0$ s.t. for $k\zeta_1^{-1}\leq 2\delta$,
%\begin{eqnarray*}
%|\frac{L_1((k+3)\zeta_1^{-1})}{L_1(k\zeta_1^{-1})}-1|<\varepsilon\\
%|\frac{L_1((k-2)\zeta_1^{-1})}{L_1(k\zeta_1^{-1})}-1|<\varepsilon.
%\end{eqnarray*}
For the conditional variance, by Lemma \ref{LemmaMu}(i),  when $t \in B_i$ and $k$ large enough,  we have
\begin{eqnarray}
  C_0(t,t)&\leq&\lambda_1L_1((k+3)\zeta_1^{-1})(k+3)^{\alpha_1}\zeta_1^{-\alpha_1}\notag\\
  &\leq&2\lambda_1 \frac{L_1(k\zeta_1^{-1})}{L_1(\zeta_1^{-1})}k^{\alpha_1} b^{-2}\label{eqI22}
\end{eqnarray}
According to Lemma \ref{LemmaMu} (iii), $E(\sup_{|t+\tau-t_i|\leq 2\zeta_1^{-1}}f_0(t)) = O(b^{-1})$ as $b\to\infty$. So for $k$ large enough, we have
\begin{equation}\label{eqI23}
E\Big[\sup_{t\in B_i}f_0(t)\Big ]\leq \frac{\Delta_\tau}{4}\frac{L_1(k\zeta_1^{-1})}{L_1(\zeta_1^{-1})}k^{\alpha_1} b^{-1}.
\end{equation}
By Proposition \ref{PropBorel}, \eqref{eqI21}, \eqref{eqI22}, and \eqref{eqI23}, we have
\begin{eqnarray}
P(\sup_{|t-t_i|\leq2\zeta_1^{-1}}f(t)\geq\gamma|f(\tau)=\gamma+\frac{z}{b})
%&=&P(\sup_{|t-t_i|\leq2\zeta_1^{-1}}f_\tau(t)\geq\gamma)\nonumber\\
%&\leq&\exp(-\frac{L_1(k\zeta_1^{-1})k^\alpha_1}{16\lambda_1})\nonumber\\
&\leq&\exp(-\frac{\Delta_\tau^2L_1(k\zeta_1^{-1})k^{\alpha_1}}{64L_1(\zeta_1^{-1})\lambda_1})\leq \exp(-\frac{\Delta_\tau^2 k^{\alpha_1/2}}{64\lambda_1}).\label{eqI24}
\end{eqnarray}
The last inequality of the above display is due to Lemma \ref{LemmaSlow}(ii).
Note that
\begin{equation*}
  P(\sup_{x\zeta_1^{-1}<|t-\tau|<\delta}f(t)>\gamma|f(\tau)=\gamma+\frac{z}{b})\leq \sum_{x\zeta_1^{-1}<|t_i-\tau|<\delta'} P(\sup_{t\in B_i}f(t)\geq\gamma|f(\tau)=\gamma+\frac{z}{b}).
\end{equation*}
According to \eqref{eqI24}, we further bound the above probability by
  \begin{eqnarray*}
%    &&P(\sup_{A\zeta_1^{-1}\leq|t-\tau|\delta}f(t)\geq\gamma|f(\tau)=\gamma+\frac{z}{b})\\
    %&=&P(\sup_{\delta>|t|\geq A\zeta_1^{-1}}f(t)\geq\gamma|f(\tau)=\gamma+\frac{z}{b})\\
    %&\leq&\sum_{k=\lfloor A \rfloor }^{\lceil \delta\zeta_1\rceil} P(\sup_{k\zeta_1^{-1}\leq|t-\tau|\leq (k+1)\zeta_1^{-1}}f(t)\geq\gamma|f(\tau)=\gamma+\frac{z}{b})\\
    %&\leq&\sum_{k=\lfloor A \rfloor }^{\lceil \delta\zeta_1\rceil}\sum_{k\zeta_1^{-1}\leq|t_i-\tau|\leq(k+1)\zeta_1^{-1}}P(\sup_{2\zeta_1^{-1}}f_\tau(t)\geq\gamma)\\
    %&\leq&\sum_{k=\lfloor A \rfloor }^{\lceil \delta\zeta_1\rceil} C(d)k^{d-1}\exp(-c_4k^{\alpha_1/2})\\
    \sum_{x\zeta_1^{-1}<|t_i-\tau|<\delta} P(\sup_{t\in B_i}f(t)\geq\gamma|f(\tau)=\gamma+\frac{z}{b})&\leq&O(1)\sum_{k=\lfloor x \rfloor }^{\delta\zeta_1} k^{d-1}\exp(-\frac{\Delta_\tau k^{\alpha_1/2}}{64\lambda_1})\\
    &\leq& e^{-x^{\alpha_1/2-\varepsilon_0}}
  \end{eqnarray*}
for $x$ sufficiently large and $\varepsilon_0$ small.
We integrate the above bound with respect to $(z,\tau)$ under the measure $Q$ and conclude the proof.
\end{proof}

\bigskip

\begin{proof}[Proof of Lemma \ref{Lemmah}]
The proof of this lemma is based on the fact that $P(f(t)>\gamma)$ has the approximation
\begin{equation*}
P(f(t)>\gamma)=\frac{1}{\sqrt{2\pi}}\frac{\sigma(t)}{\gamma-\mu(t)}\exp\left(-\frac{\gamma-\mu(t)}{2\sigma(t)}\right)(1+o(1)),
\end{equation*}
combined with the expansion of $\sigma(t)^2$ around $t^*$,
\begin{equation*}
\sigma(t)^2=\sigma(t^*)^2-2\sigma(t^*)\Lambda L_2(|t-t^*|)|t-t^*|^{\alpha_2}(1+o(1)).
\end{equation*}

After basic calculation of expansion and integration, we can prove that there exist $\varepsilon_0, \kappa>0$, such that for $x>\kappa$, we have
\begin{eqnarray*}
  &\int_{|t-t^*|\leq\zeta_2^{-1}}P(f(t)>\gamma)dt&\geq \frac{1}{\sqrt{2\pi}}\frac{\sigma(t^*)/2}{\gamma+\mu_T}\zeta_2^{-d}\exp\left(-\frac{(\gamma-\mu(t^*))^2}{2\sigma(t^*)^2}\right)\cdot\varepsilon_0\\
  &\int_{x\zeta_2^{-1}<|t-t^*|<\delta}P(f(t)>\gamma)dt&\leq\frac{1}{\sqrt{2\pi}}\frac{\sigma(t^*)}{\gamma-\mu_T}\zeta_2^{-d}\exp\left(-\frac{(\gamma-\mu(t^*))^2}{2\sigma(t^*)^2}\right)\exp\left(-x^{\alpha_2/2}\right)\\
   &\int_{|t-t^*|>\delta}P(f(t)>\gamma)dt&\leq\frac{1}{\sqrt{2\pi}}\frac{\sigma(t^*)}{\gamma-\mu_T}\exp\left(-\frac{(\gamma-\mu(t^*))^2}{2\sigma(t^*)^2}\right)\exp(-\varepsilon_0b^{2})
\end{eqnarray*}
Combining the three inequalities above, and noticing that $h_b(t)=\frac{P(f(t)>\gamma)}{\int_{t\in T}P(f(t)>\gamma)dt}$, we have the result in this lemma.
\end{proof}

\end{document}